\documentclass[11pt,reqno]{amsart}

\usepackage{amsmath,amssymb,amsthm,amsfonts,latexsym}
\usepackage{stmaryrd}

\usepackage{mathrsfs}

\usepackage{relsize}
\usepackage{comment}

\usepackage{graphicx,subfigure}

\usepackage{cite}

\numberwithin{equation}{section}

\usepackage[colorlinks=true,urlcolor=blue,
citecolor=red,linkcolor=blue,linktocpage,pdfpagelabels,
bookmarksnumbered,bookmarksopen,backref=page]{hyperref}

\usepackage{fullpage}

\usepackage{ifthen} 

\provideboolean{shownotes} 
\setboolean{shownotes}{true} 
\newcommand{\margnote}[1]{
\ifthenelse{\boolean{shownotes}}%
{\marginpar{\raggedright\tiny\texttt{#1}}}%
{}%
}

\newcommand{\hole}[1]{
\ifthenelse{\boolean{shownotes}}%
{\begin{center} \fbox{ \rule {.25cm}{0cm}
\rule[-.1cm]{0cm}{.4cm} \parbox{.85\textwidth}{\begin{center}
\texttt{#1}\end{center}} \rule {.25cm}{0cm}}\end{center}}
{}
}

\newcommand{\R}{{\mathbb R}}
\newcommand{\M}{{\mathbb M}}
\newcommand{\Z}{{\mathbb Z}}
\newcommand{\C}{{\mathbb C}}

\newcommand{\N}{{\mathbb N}}

\newcommand{\Id}{{\mathbb I}_d}

\newcommand{\In}{{\mathbb I}_n}
\newcommand{\Idd}{{\mathbb I}_{d^2}}
\newcommand{\Itr}{{\mathbb I}_3}
\newcommand{\Ido}{{\mathbb I}_2}

\newcommand{\cT}{{\mathcal{T}}}
\newcommand{\cU}{{\mathcal{U}}}
\newcommand{\cA}{{\mathcal{A}}}

\newcommand{\cR}{{\mathcal{R}}}

\newcommand{\cS}{{\mathcal{S}}}
\newcommand{\cG}{{\mathcal{G}}}
\newcommand{\cK}{{\mathcal{K}}}

\newcommand{\AT}{Q}
\newcommand{\tAT}{\mathcal{Q}}

\newcommand{\xim}{\widetilde{\xi}}

 \newcommand{\llb}{\llbracket}
\newcommand{\rrb}{\rrbracket}

\newcommand{\eu}{\hat{e}_1}

\newcommand{\ei}{\hat{e}_i}
\newcommand{\ej}{\hat{e}_j}
\newcommand{\eq}{\hat{e}_q}
\newcommand{\ep}{\hat{e}_p}

\newcommand{\hnu}{\hat{\nu}}

\newcommand{\Cof}{\mathrm{Cof}\,}
\newcommand{\sgn}{\mathrm{sgn}\,}

\newcommand{\ii}{\mathrm{i}}

\renewcommand{\Re}{\mbox{\rm Re}\,}
\renewcommand{\Im}{\mbox{\rm Im}\,}

\newcommand{\Tr}{\mbox{\rm tr}\,}
\newcommand{\Div}{\mbox{\rm div}}

\theoremstyle{plain}

\newtheorem{lemma}{Lemma}[section]

\newtheorem{theo}[lemma]{Theorem}
\newtheorem*{teorema}{Theorem}
\newtheorem{proposition}[lemma]{Proposition}
\newtheorem{corollary}[lemma]{Corollary}

\theoremstyle{definition}
\newtheorem{defin}[lemma]{Definition}

\newtheorem{remark}[lemma]{Remark}

\theoremstyle{remark}

\newcommand{\txi}{\tilde{\xi}}

\begin{document}


\title[Stability of shock fronts for Hadamard hyperelastic materials]{Stability of classical shock fronts for compressible hyperelastic materials of Hadamard type}

\author[R. G. Plaza]{Ram\'on G. Plaza}

\address{{\rm (R. G. Plaza)} Instituto de Investigaciones en Matem\'aticas Aplicadas y en Sistemas\\Universidad Nacional 
Aut\'onoma de M\'exico\\Circuito Escolar s/n, Ciudad Universitaria, Ciudad de M\'{e}xico C.P. 04510 (Mexico)}

\email{plaza@mym.iimas.unam.mx}

\author[F. Vallejo]{Fabio Vallejo}
 
\address{{\rm (F. Vallejo)}  Instituto de Investigaciones en Matem\'aticas Aplicadas y en Sistemas\\Universidad Nacional 
Aut\'onoma de M\'exico\\Circuito Escolar s/n, Ciudad Universitaria, Ciudad de M\'{e}xico C.P. 04510 (Mexico)}

\email{fvallejo@unam.mx}

\subjclass[2010]{35Q74, 35L65, 35L67, 74J40}

\keywords{Classical hyperelastic shocks, compressible Hadamard materials, multidimensional stability, Lopatinski\u{\i} determinant.}

\begin{abstract}
This paper studies the uniform and weak Lopatinski\u{\i} conditions associated to classical
(Lax) shock fronts of arbitrary amplitude for compressible hyperelastic materials of
Hadamard type in several space dimensions. Thanks to the seminal works of
Majda \cite{M1,M2} and M\'{e}tivier \cite{Me2,Me1,Me5}, the uniform Lopatinski\u{\i} condition ensures the local-in-time, multidimensional, nonlinear stability of such
fronts. The stability function (also called \emph{Lopatinski\u{\i} determinant}) for shocks of arbitrary amplitude in this large class of hyperelastic materials is computed explicitly. This information is used to establish the conditions for uniform and weak shock stability in terms of the parameters of the shock and of the elastic moduli of the material.
\end{abstract}

\maketitle


\setcounter{tocdepth}{1}
\tableofcontents


\section{Introduction}

In this paper, we consider planar shock fronts occurring in an ideal, non-thermal, compressible hyperelastic medium in several space dimensions. Shock waves are important in many applications such as gas dynamics, acoustics, material sciences, geophysics and even in medicine and health sciences. They appear as idealized, abrupt disturbances (discontinuous, in the absence of dissipation effects) which carry energy and propagate faster than the characteristic speed of the medium in front of them. In the mathematical theory of hyperbolic systems, shock waves are represented by weak solutions to nonlinear systems of conservation laws which satisfy classical jump conditions of Rankine-Hugoniot type plus admissibility/entropy conditions of physical origin (see, e.g., \cite{Da4e,Ser1,BS} and the references therein). A fundamental property from both the mathematical and physical perspectives is their \emph{stability} under small perturbations. The shock stability theory has its origins in the physics literature and, more concretely, in the context of gas dynamics, where shock waves for the (inviscid) Euler equations constitute the main paradigm. The inviscid shock stability analysis for gas dynamics (at least from a formal viewpoint) dates back to the mid-1940s (cf. \cite{Roberts45,Beth}) and was thereafter pursued by many physicists and engineers in the following decades (for an abridged list of references, see \cite{Dyak54,Erp62,CSGard63}). The nonlinear theory of stability and existence of shock fronts for general systems of conservation laws started with the seminal work of Majda \cite{M1,M2} (see also the nonlinear analysis of Blokhin \cite{Blok82} for the equations of gas dynamics) and was later extended and revisited by M\'etivier \cite{Me1,Me2,Me5}. As a result from their pioneering work, it is now known that the nonlinear stability of shock fronts depends upon the \emph{Lopatinski\u{\i} conditions} for linear hyperbolic initial boundary value problems \cite{Kre70,Lopa70}.  

Given a small multidimensional smooth perturbation impinging on the shock interface, one may ask whether it leads to a local solution with the same shock structure (smooth regions separated by a modified, curved shock discontinuity), or if this structure is destroyed. By a suitable change of coordinates (a shock localization procedure permitted by the finite speed of propagation and originally conceived by Erpenbeck \cite{Erp62}), the transmission problem can be reduced to an initial boundary value problem in a half space. The resulting mixed problem is non-standard, in the sense that the conditions at the boundary are of differential type in the shock location, expressing the Rankine-Hugoniot jump conditions across the shock. Still, the linearized problem can be treated by a normal modes analysis resulting into the uniform and weak Lopatinski\u{\i} conditions for $L^2$ well-posedness of the linearized problem. Majda \cite{M1} named the latter the \emph{uniform and weak Lopatinski\u{\i} conditions for shock stability}. Under the uniform (linearized) stability condition, Majda \cite{M2} proved the local-in-time existence and uniqueness of uniformly stable shock waves for general nonlinear systems (satisfying some block structure condition) by means of a fixed-point argument and a suitable iteration scheme. The weak Lopatinski\u{\i} condition, in contrast, is equivalent to the absence of Hadamard-type \cite{Hada1902} ill-posedness for the problem and in this case the shock is called weakly stable. Hence, the nonlinear stability problem reduces to verifying the linear stability conditions, which can be recast in terms of an analytic function in the frequency space known as the \emph{Lopatinski\u{\i} determinant (or stability function)}. The uniform Lopatinski\u{\i} condition plays an important role in the stability of viscous shock profiles as well (in which the Lopatinski\u{\i} determinant arises as a limit of associated Evans functions for the viscous linearized problem), as shown by Zumbrun and Serre \cite{ZS1} (see also \cite{Z3,Z6} and the references therein). The original works by Majda \cite{M1,M2} pertain to classical (or Lax) shocks. It is to be noted, however, that the analysis and methods have been extended to other situations and the theory now encompasses non-classical (undercompressive and over-compressive) shocks, vortex sheets, phase boundaries and detonation fronts (cf. \cite{Fr3,B1,B2,BF,CJLW07,FrP1,CoS04,CoS08}). For completeness, Appendix \ref{sechypall} contains a summary of the stability conditions for multidimensional (Lax) shock fronts as well as the definition of the Lopatinski\u{\i} determinant. A detailed account of the methodology and their numerous implications can be found in the monograph by Benzoni-Gavage and Serre \cite{BS}.

In the case of the equations of hyperelasticity, the literature on (multidimensional) shock stability is scarce. Corli \cite{Co1} proved that the elastodynamics equations for hyperelastic materials satisfy the block structure of Majda \cite{M1} and examined the stability of small-amplitude shocks for St. Venant-Kirchhoff materials, verifying for this particular model the general result of M\'etivier \cite{Me1}, which assures that all sufficiently weak extreme shocks are stable. Other studies on small-amplitude, weakly anisotropic elastic shocks can be found in \cite{KuCh00}. In a later contribution, Freist\"uhler and the first author \cite{FrP1} studied the Lopatinski\u{\i} condition and stability of hyperelastic \emph{phase boundaries}, which can be identified as non-classical shocks of undercompressive type (cf. Freist\"uhler \cite{Fr3}). The fundamental difference with Lax shocks is that, apart from the Rankine-Hugoniot jump conditions, for phase boundaries one has to determine an extra jump condition in the form of a kinetic relation (or kinetic rule; see also \cite{B1,B2}). The stability conditions found in \cite{FrP1} have been numerically verified for martensite twins in two \cite{FrP3} and three space dimensions \cite{Pl2}, under perturbations of the kinetic equal area rule. There is a recent result on the stability of quasi-transverse elastic shocks subjected to dissipation (viscosity) effects \cite{CIS19}, which makes use of Evans functions techniques. Up to our knowledge, there are no other results (either numerical or analytical) on stability of hyperelastic shocks in the literature. In this work, we study for the first time the stability conditions for classical shocks fronts of arbitrary amplitude within hyperelastic media belonging to the large class of \emph{compressible Hadamard materials}.

Compressible hyperelastic materials of Hadamard type are characterized by stored energy density functions of the form
\begin{equation}
\label{Hadamardmat}
W(U) = \frac{\mu}{2} \Tr (U^\top U) + h(\det U), 
\end{equation}
where $U \in \R^{d \times d}$, with $\det U >0$, denotes the deformation gradient of the elastic material in dimension $d \geq 2$, $\Tr(\cdot)$ is the trace function and $\mu > 0$ is a positive constant known as the \emph{shear modulus}. The volumetric density function $h = h(J)$, for $J = \det U \in (0,\infty)$, accounts for energy changes due to changes in volume. In this paper, we assume the following about the function $h$:
\begin{equation}
\label{H1}
\tag{H$_1$}
h \in C^3 ((0,\infty); \, \R),
\end{equation}
\begin{equation}
\label{H2}
\tag{H$_2$}
h''(J) > 0, \quad \text{for all } \, J > 0,
\end{equation}
\begin{equation}
\label{H3}
\tag{H$_3$}
h'''(J) < 0, \quad \text{for all } \, J > 0.
\end{equation}
Hypothesis \eqref{H1} is a minimal regularity requirement. The convexity of the volumetric energy density function \eqref{H2} is a sufficient condition for the material to be strongly elliptic. In the materials science literature, those energies that satisfy conditions \eqref{H1} and \eqref{H2} are known as compressible Hadamard materials. Hypothesis \eqref{H3} can be interpreted as a further \emph{material convexity} property (see Remark \ref{remmatconvex} below), which is needed for the shock stability analysis.

The term Hadamard material was coined by John \cite{Joh66} (based on an early description by Hadamard \cite{Hada1903}) to account for a large class of elastic media where purely longitudinal waves may propagate in every direction, in contrast with other elastic, compressible, isotropic materials which, subjected to large homogeneous static deformations, underlie purely longitudinal waves only in the directions of the principal axes of strain (cf. Truesdell \cite{True61}). Knowles \cite{Kno77} proved, for instance, that this class of materials admits non-trivial states of finite anti-plane shear. The most natural interpretation of a compressible elastic material of Hadamard type is, however, as a compressible extension of a neo-Hookean incompressible solid as described by Pence and Gou \cite{PnGo15}. For convenience of the reader, we have included in Appendix \ref{haddimate} a comprehensive and self-contained introduction to compressible Hadamard materials from the viewpoint of the theory of infinitesimal strain, in which we extend to arbitrary space dimensions the nearly incompressible versions of the neo-Hookean models which are compatible with the small-strain regime. It is to be observed, though, that the class of Hadamard materials considered in this paper also includes materials which may undergo large volume changes. Section \S \ref{subsecexamples} contains a list of energy densities which can be found in the materials science literature and belong to the compressible Hadamard class.

The goal of this paper is, therefore, to determine the stability conditions for shock fronts in compressible non-thermal Hadamard materials in terms of the shock parameters and the elastic moduli of the medium, just as in the case for isentropic gas dynamics \cite{Erp62,M3,BS}. In $d \geq 2$ space dimensions, the $n = d^2 +d$ dynamical equations of hyperelasticity outnumber the Euler equations for isentropic fluid flow ($n = d +1$) and the calculations are thereby much more involved. Nevertheless, in this work we \emph{explicitly} compute the Lopatinski\u{\i} determinant associated to such configurations.  We start (see section \S \ref{secstab}) by describing the dynamical equations of hyperelasticity and, notably, the Legendre-Hadamard condition on the stored energy density, which guarantees the hyperbolicity of the system of conservation laws. We verify this condition in any space dimension, $d \geq 2$, and prove that the constant multiplicity condition of M\'etivier \cite{Me5} is also fulfilled. Section \S \ref{secshock} is devoted to describe classical shocks ocurring in this class of materials. We introduce a scalar parameter $\alpha \in \R$, $\alpha \neq 0$, which completely determines the shock and its amplitude once a base elastic state is selected. We call it the \emph{intensity} of the shock. It is shown that only \emph{extreme} classical shocks are possible and that they satisfy the strict Lax entropy conditions. In section \S \ref{secnmodes} we perform the normal modes analysis prior to the establishment of the stability results. In particular, we compute all the necessary ingredients to assemble the Lopatinski\u{\i} determinant associated to the shock front. A few remarks are in order. One of the main contributions of \cite{FrP1} is the observation that the normal modes analysis for hyperelastic shocks can be performed on a lower dimensional frequency manifold which encodes all the stability information and rules out some spurious frequencies of oscillation. This property is particularly important in the case of phase boundaries which are associated to subsonic speeds and to non-extreme characteristic fields, reducing the otherwise cumbersome (but still hard) calculations on the full system of dimension, $d^2 +d$. In the present case of Lax shocks, in which the characteristic fields are necessarily extreme, it is well-known that the Lopatinski\u{\i} determinant reduces to a scalar product of two complex vector fields (cf. \cite{Ser2,BS,FrP1}). Therefore, there is no considerable profit from such reduction and we opted to compute the full Lopatinski\u{\i} function in the frequency space. Still, the normal modes analysis is convoluted and section \S \ref{secnmodes} is devoted to it\footnote{as Benzoni-Gavage and Serre point out, ``...all methods require some care and are a little lengthy"; \cite{BS}, p. 431.}. We have tried to optimize the exposition and to keep the details to a minimum, but without sacrificing the direct verification and validation of the results. The reader will be a better judge if we succeeded or not. The remarkable feature is that it is possible to explicitly perform all the calculations for this large class of elastic materials, even for shocks of arbitrary amplitude (there are no restrictions on $|\alpha|$, except for those imposed by natural orientation preserving considerations; see Proposition \ref{proppadonde} below). These calculations result into the complete characterization and analysis of the Lopatinski\u{\i} determinant for classical Hadamard shocks, which yields, in turn, our main stability conclusions. This is the content of the central section \S \ref{seclopdet}. For instance, it is proved that (just as in the case of isentropic gas dynamics \cite{BS,M3}) \emph{all Lax shocks are, at least, weakly stable}. This implies that the shocks are never violently unstable with respect to multidimensional perturbations. Moreover, we introduce a material (scalar) stability parameter, $\rho = \rho(\alpha)$ (see Definition \ref{defirho} below), which determines the transition from weak to uniform shock stability. In gloss terms, our main result (see Theorem \ref{stabcriteria} below for its precise statement) can be described as follows:
\begin{teorema}
For a compressible hyperelastic Hadamard material satisfying assumptions \eqref{H1}, \eqref{H2} and  \eqref{H3}, consider a classical (Lax) shock with intensity $\alpha \neq 0$. If $\rho(\alpha)\leq0$ then the shock is uniformly stable. In the case where $\rho(\alpha)>0$, the shock is uniformly stable if and only if a material condition holds. Otherwise the shock is weakly stable.
\end{teorema}

Of course, in the case of uniformly stable shocks the stability conclusions are also nonlinear, inasmuch as the Majda-M\'{e}tivier theory applies. In section \ref{secappli} we verify the stability conditions for the Ciarlet-Geymonat \cite{CiGe82} and Blatz \cite{Bla71} models, both belonging to the class of  Hadamard materials, and exemplifying elastic media for which there is either uniform stability for all shocks, or a transition to weak stability. Finally, some conclusions and general remarks can be found in section \S \ref{secdisc}. 

\subsection*{On notation}

The unit imaginary number is denoted by $\mathrm{i} \in \C$, $\ii^2 = -1$, and $i,j \in \Z$ indicate integer indices.
$\{ \hat{e}_j \}_{j=1}^d$ denotes the canonical basis of $\R^d$ and $\Id$ denotes the identity $d \times d$ matrix, for each $d \in \N$, $d \geq 2$. $\M_+^d$ denotes the set of all real $d \times d$ matrices with positive determinant. In this paper, the elements of a real matrix $A$ will be denoted as $A_{ij}$ and $A_j$ will denote the $j$-th column vector. We slightly modify the customary notation and the principal invariants of a real $d \times d$ matrix $A$, understood as the coefficients appearing in its characteristic polynomial, will be denoted as $I^{(1)}(A) = \Tr A$, $I^{(d-1)}(A) = \Tr (\Cof A)$ or $I^{(d)}(A) = \det A$. We denote the cofactor matrix of any real square matrix $A \in \R^{d\times d}$ to be $(\Cof A)_{ij} = (-1)^{i+j} \det (A'_{(i,j)})$, $1 \leq i,j \leq d$, where $A'_{(i,j)}$ is the $(d-1) \times (d-1)$ matrix obtained by deleting the $i$-th row and the $j$-th column of $A$. Hence,
\begin{equation}
\label{exprcof}
(\Cof A)^\top A = A (\Cof A)^\top = (\det A) \Id.
\end{equation}For any $a, b \in \R^d$, $a \otimes b \in \R^{d \times d}$ denotes the standard tensor product matrix whose $(i,j)$-entry is $a_i b_j$. For a complex number $\lambda$, we denote complex conjugation by ${\lambda}^*$ and its real
and imaginary parts by $\Re \lambda$ and $\Im \lambda$, respectively. Real matrices are denoted by capital roman font letters (e.g. $A \in \R^{d \times d}$), except for the first Piola-Kirchhoff stress tensor, denoted by $\sigma$. Complex matrix fields in the space of frequencies will be denoted with calligraphic letters (e.g., $\mathcal{A} \in \C^{n \times n}$).
Complex transposition of block matrices are indicated by the symbol ${}^*$ (e.g., $\mathcal{A}^*$),
whereas simple transposition is denoted by the symbol $\mathcal{A}^\top$. For any (scalar or matrix valued) function $g$ of the state variables $u$, the jump across a shock discontinuity will be denoted as $\llb g(u) \rrb := g(u^+) - g(u^-)$. Throughout this paper we use the non-standard symbol $U \in \R^{d \times d}$ to denote the deformation gradient of an elastic material (usually denoted with the symbol $F$ in the literature). For convenience of the reader we have kept the notation used in \cite{FrP1} for the shock stability analysis, because we constantly refer to the formulae and the results in that paper. The same notation was also used by John \cite{Joh3} in his classical paper on elasticity theory. 

\section{Elastodynamics and hyperbolicity}
\label{secstab}

\subsection{The equations of elastodynamics}

The elastic body under consideration is identified at rest by its 
reference configuration, which is an open, connected set $\Omega \subseteq \R^d$, $d \geq 1$. Here $d \in \N$ denotes the dimension of the physical space and, typically, $d = 1,2$ or $3$. Since we are interested in the multidimensional stability of shock fronts we assume that $d \geq 2$ for the rest of the paper. The motion of the elastic body is described by the Lagrangian mapping coordinate, $(x,t) \mapsto y(x,t)$, $y : \Omega \times [0,\infty) \to \R^d$, that is, $y = y(x,t)$ denotes the position at time $t > 0$ of the material particle that was located at $x \in \Omega$ when $t = 0$. It determines the deformed position of the material point $x \in \Omega$. It is assumed that the Lagrangian mapping is smooth enough, say, at least of class $C^2(\Omega \times (0,\infty); \R^d)$ and one-to-one with a locally Lipschitz inverse. The \emph{local velocity} at the material point is defined as $v(x,t) := y_t (x,t)$, $v : \Omega \times [0,\infty) \to \R^d$, or component-wise, as 
\[
v_i(x,t) = \frac{\partial y_i(x,t)}{\partial t}, \qquad i = 1, \ldots, d.
\]
The \emph{local deformation gradient}, $U(x,t) := \nabla_x y (x,t)$, $U :  \Omega \times [0,\infty) \to \R^{d \times d}$, is a real $d \times d$ matrix whose $(i,j)$-component is given by
\[
U_{ij} (x,t) = \frac{\partial y_i}{\partial x_j} (x,t), \quad 1 \leq i,j \leq d.
\]
Following the notation in \cite{FrP1}, $U_j \in \R^d$ will denote the $j$-th column of $U$, that is,
\[
U_j = \begin{pmatrix} U_{1j} \\ \vdots \\ U_{dj} \end{pmatrix} \in \R^d, \qquad j = 1, \ldots, d.
\]
By physical considerations (namely, that the material does not change orientation and that it is locally invertible \cite{Ci}) one usually requires that 
\begin{equation}
\label{defJ}
J = J(U) := \det U > 0. 
\end{equation}
Thus, it is assumed that $U(x,t) \in \M_+^d$ for all $(x,t) \in \Omega \times (0,\infty)$.

Supposing that no thermal effects are taken into consideration and in the absence of external forces, the principles of continuum mechanics (cf. \cite{Ci,Da4e,Sil97,TrNo3ed}) yield the basic equations of elastodynamics,
\begin{equation}
\label{elasto}
 y_{tt} - \Div_{x} \! \, \sigma = 0,
\end{equation}
for $(x,t) \in \Omega \times [0,\infty)$ where $\sigma$ is the (first)
Piola-Kirchhoff stress tensor and whose $(i,j)$-component is denoted as $\sigma_{ij}$, $1 \leq i,j \leq d$. System \eqref{elasto} is a short-cut for the system of $d$ equations,
\begin{equation}
\label{elastoij}
\frac{\partial^2 y_{i}}{\partial t^2} - \sum_{j=1}^d \frac{\partial \sigma_{ij}}{\partial x_j} = 0, \qquad i = 1, \ldots, d,
\end{equation}
expressing conservation of momentum. 

An elastic material is called \emph{hyperelastic} if there exists a single stored energy density function $W : \M_+^d\to \R$, defined per unit volume in the reference configuration, from which all stress fields can be derived. In particular, the first Piola-Kirchhoff stress tensor (cf. \cite{Ci,TrNo3ed}), $\sigma = \sigma(U)$, derives from $W$ as
\[
\sigma(U) = {\partial W \over \partial U}, \qquad U \in \M_+^d,
\]
or component-wise as
\[
\sigma_{ij}(U) = \frac{\partial W}{\partial U_{ij}}, \qquad 1 \leq i,j \leq d.
\]
We adopt the notation in \cite{FrP1}, under which $\sigma_j = \sigma(U)_j \in \R^d$ denotes the $j$-th column of $\sigma(U)$; more precisely,
\[
\sigma(U)_j = \begin{pmatrix} W_{U_{1j}} \\ \vdots \\ W_{U_{dj}}\end{pmatrix}, \qquad j = 1, \ldots, d.
\]

Basic restrictions on the function $W$ include, for instance, the \emph{principle of frame indifference} (cf. \cite{Ci,Da4e,Og84}),
\[
W(U) = W(OU), \quad \text{for all } \; O \in \mathrm{SO}_d(\R), \,\; U \in \M_+^d,
\]
where $\mathrm{SO}_d(\R)$ denotes the set of all orthogonal real $d \times d$ matrices (rotations); \emph{normalization}, requiring $W(U) \geq 0$ for all $U \in \M_+^d$ (cf. \cite{Ci,Og84}); and \emph{material symmetry or isotropy} (see \cite{Og84,TrNo3ed}),
\[
W(U) = W(UO), \quad \text{for all } \; O \in \mathrm{SO}_d(\R), \; U \in \M_+^d.
\]

It is assumed that $W$ is \emph{objective}, so that it depends on the deformation gradient $U$ only through the right Cauchy-Green tensor, $C = U^\top U$ (see, for example, Ogden \cite{Og84}), which is symmetric positive definite by definition and measures the length of an elementary vector after deformation in terms of its definition in the reference configuration. Furthermore, it is well-known that the energy density function, $W = W(U) = \widetilde{W}(C)$, of any frame-indifferent, isotropic material, is a function of the principal invariants of the symmetric Cauchy-Green tensor $C$, $W = \overline{W}(I^{(1)}, \ldots, I^{(d)})$. This is called the \emph{Rivlin-Ericksen representation theorem} \cite{RivEr55} (see Ciarlet \cite{Ci}, section 3.6 for the statement and proof in dimension $d=3$, and Truesdell and Noll \cite{TrNo3ed}, section B-10, p. 28, in arbitrary dimensions.)

\subsection{Legendre-Hadamard condition and hyperbolicity}

The hyperbolicity of system \eqref{elasto} is a necessary condition for the well-posedness of the Cauchy problem and the corresponding numerical methods of Godunov type \cite{Da4e,GNH16}. The criterion for hyperbolicity of system \eqref{elasto} is related to another condition on the stored energy density function, known as the \emph{Legendre-Hadamard condition} \cite{Da4e,Sil97}. In order to state the latter we follow the notation of \cite{FrP1} and express the second derivatives of the energy density in terms of the following $d \times d$ matrices defined by
\begin{equation}
\label{defBij}
B_i^j(U) := \frac{\partial \sigma_j}{\partial U_i} = \begin{pmatrix} W_{U_{1j}U_{1i}} & \cdots & W_{U_{1j}U_{di}} \\ \vdots & & \vdots \\ W_{U_{dj} U_{1i}} & \cdots & W_{U_{dj} U_{di}}\end{pmatrix} \in \R^{d \times d},
\end{equation}
for each pair $1 \leq i,j \leq d$. That is, the $(l,k)$-component of the matrix $B_i^j$ is $W_{U_{lj} U_{ki}} = \partial^2 W / \partial U_{lj} \partial U_{ki}$, for each fixed $1 \leq i,j \leq d$.  Notice that the matrices $B_i^i(U)$ are symmetric, $B^i_i(U)^\top = B^i_i(U)$ for all $U$ and all $i$, and that $B^j_i(U) = B^i_j(U)^\top$ for all $U$ and all $i,j$ by definition. Whence, the $d \times d$ \emph{acoustic tensor} can be defined as
\begin{equation}
\label{defacouten}
\AT(\xi, U) := \sum_{i,j=1}^d \xi_i \xi_j B_i^j(U) \in \R^{d \times d},
\end{equation}
for all $\xi \in \R^d$, $U \in \M_+^d$. Note that the acoustic tensor is symmetric.

\begin{defin}[Legendre-Hadamard condition]
The energy density function $W = W(U)$ satisfies the \emph{Legendre-Hadamard} condition at $U \in \M_+^d$ if
\begin{equation}
\label{LHcond}
\eta^\top \AT(\xi,U) \eta > 0, \quad \text{for all } \; \xi, \eta \in \R^d \backslash \{0\}.
\end{equation}
In other words, the acoustic tensor is positive definite for all frequencies $\xi \neq 0$, $\eta \neq 0$.
\end{defin}
\begin{remark}
The Legendre-Hadamard condition is tantamount to the convexity of $W$ along any direction $\xi \otimes \eta$ with rank one. It is also said that $W$ is a \emph{rank-one convex function} of the deformation gradient $U$. For an hyperelastic medium, this condition is equivalent to the strong ellipticity of the operator $y \mapsto \mathrm{div}_x (\sigma(\nabla_x y))$ (cf. Dafermos \cite{Da86a}) and, consequently, in the context of elastostatics the rank-one convexity condition is also called \emph{strong ellipticity} (see, e.g., \cite{Ba4,Dac01,Sfy11}). Even though it is well-known that rank-one convexity of the energy function is equivalent to the hyperbolicity of the equations of elastodynamics for an hyperelastic material (see \cite{Da4e,FrP1,Sil97} or Lemma \ref{LHimpliesHyp} below), this property is difficult to verify in practice, even in the case of isotropic materials (cf. \cite{Dac01,Davi91,DeTPZ12,GNH16,Horg96}). Necessary and sufficient conditions of strong ellipticity for two-dimensional isotropic materials have been discussed in \cite{Aube95,Dac01,Davi91,KnSt76}, and for three-dimensional media in \cite{Dac01,SimSpe83,WanAro96}. It is to be noted, however, that the elastic media considered in this paper constitute a wide class of materials for which the rank-one convexity assumption is remarkably easy to verify even in higher space dimensions (see Section \ref{secHadhyp} below). Finally, we observe that the Legendre-Hadamard condition plays a crucial role in delimiting ``stability'' boundaries for weak local minima in the calculus of variations (cf. Giaquinta and Hildebrandt \cite{GiaHil96vI}; see also \cite{BJ87,BJ92,FR-C00,GrTr16,Og84}) assuring its importance in elasticity theory, in general, and in the analysis of elastic shocks and phase boundaries, in particular.
\end{remark}

The equations of elastodynamics \eqref{elasto} can be recast a first-order system of conservation laws of the form \eqref{HSCL} in Appendix \ref{sechypall} below when they are written in terms of the local velocity $v$ and of the deformation gradient $U$ (see \cite{Co1,FrP1,FrP3,Pl2}). Indeed, upon substitution we arrive at
\begin{equation}
\label{model}
\begin{array}{r}
U_t - \nabla_{x}  v = 0, \\
v_t - \Div_{x}  \, \sigma(U) = 0,
\end{array}
\end{equation}
where $t \in [0,\infty)$, $x \in \Omega \subseteq \R^d$, which is subject to the additional physical constraint
\begin{equation}
\label{curlfree}
\hbox{curl}_x \, U = 0.
\end{equation}
Equations \eqref{model} and \eqref{curlfree} are concise forms of the $d^2 + d$ first order equations of motion,
\begin{equation}
\label{modelij}
\begin{aligned}
\partial_t U_{ij} - \partial_{x_j} v_i &= 0, & \quad i,j = 1, \ldots, d,\\
\partial_t v_i - \sum_{j=1}^d \partial_{x_j} \Big( \frac{\partial W}{\partial U_{ij}} \Big) &= 0, &\quad i = 1, \ldots, d,
\end{aligned}
\end{equation}
and of the constraints
\begin{equation}
\label{curlfreeij}
\partial_{x_k} U_{ij} = \partial_{x_j} U_{ik}, \qquad i,j,k = 1, \ldots, d,
\end{equation}
respectively. Therefore, if we denote 
\[
u = \begin{pmatrix} U_1 \\ \vdots \\ U_d \\ v \end{pmatrix} \in \R^{d^2 + d}, \qquad  f^j(u) = - \begin{pmatrix} 0 \\ \vdots \\ v \\ \vdots \\ 0 \\ \sigma(U)_j \end{pmatrix} \in \R^{d^2 + d}, \; j = 1, \ldots, d,
\]
where the vector $v$ appears in the $j$-th position in the expression for $f^j(u)$, system \eqref{model} can be written as a system of $n = d^2 + d$ conservations laws of the form \eqref{HSCL}, with conserved quantities $u \in \R^n$ and fluxes $f^j(u) \in C^2(\cU; \R^n)$, $1 \leq j \leq d$. Here the open, connected set of admissible states is 
\[
\cU = \{ (U,v) \in \R^{d \times d} \times \R^d \, : \, \det U > 0 \}. 
\]
Under this notation, the Jacobians 
$A^j(u) := Df^j(u) \in \R^{n \times n}$ are given by
\[
A^j(u) = - \begin{pmatrix} & & & 0 \\ & & & \vdots  \\ & 0 & & \Id \\ & & & \vdots  \\ & & & 0 \\ B_1^j(U) & \cdots & B_d^j(U) & 0  \end{pmatrix} \in \R^{(d^2 + d) \times (d^2 + d)},
\]
for all $j = 1, \ldots, d$ (see \cite{FrP1} for details). Notice that the Jacobians depend on $u = (U,v)^\top$ only through the deformation gradient. Thus, with a slight abuse of notation we write, from this point on,
\[
A^j = A^j(U), \qquad U \in \M_+^d, \quad j = 1, \ldots, d.
\]
The symbol \eqref{defA} is then defined as
\[
A(\xi,U) = \sum_{j=1}^d \xi_j A^j(U), \qquad \xi \in \R^d, \; \; U \in \M_+^d.
\]

As discussed in \cite{FrP1}, due to technical reasons that pertain to the applicability of the stability theory of shocks, we also require the following constant multiplicity property.
\begin{defin}[constant multiplicity assumption]
\label{constmultass}
The energy density function $W = W(U)$ satisfies the constant multiplicity property at $U$ if, for all frequencies $\xi \in \R^d$, $\xi \neq 0$, the eigenvalues of the acoustic tensor $\AT = \AT(\xi,U)$ are all semi-simple (their geometric and algebraic multiplicities coincide) and their multiplicity is independent of $\xi$ and $U$.
\end{defin}

\begin{lemma}
\label{LHimpliesHyp}
Suppose that $W = W(U)$ satisfies the Legendre-Hadamard condition \eqref{LHcond} and the constant multiplicity assumption. Assume that for each $(\xi,U) \in \R^d \backslash \{0\} \times \M^d_+$, the associated acoustic tensor $\AT = \AT(\xi,U)$ has $k$ distinct semi-simple positive eigenvalues, $0 < \kappa_1(\xi,U) < \ldots < \kappa_k(\xi,U)$, $1 \leq k \leq d$, with constant multiplicities $\widetilde{m}_l$, $1 \leq l \leq k$, such that $\sum_{l=1}^k \widetilde{m}_l = d$. Then system \eqref{model} (or equivalently, system \eqref{elasto}) is hyperbolic with characteristic velocities (eigenvalues of the symbol $A(\xi,U)$) given by:
\begin{itemize}
\item[(i)] $\widetilde{a}_0(\xi,U) = 0$ with constant multiplicity $\widetilde{m}_0 = d^2 - d$;
\item[(ii)] $\widetilde{a}_1(\xi,U) = - \sqrt{\kappa_k(\xi,U)} < \ldots < \widetilde{a}_k(\xi,U) = - \sqrt{\kappa_1(\xi,U)} < 0$, each with constant multiplicity $\widetilde{m}_l$, $1 \leq l \leq k$; and,
\item[(iii)] $0 < \widetilde{a}_{k+1}(\xi,U) = \sqrt{\kappa_1(\xi,U)} < \ldots < \widetilde{a}_{2k}(\xi,U) =  \sqrt{\kappa_k(\xi,U)}$, each with constant multiplicity $\widetilde{m}_{k+l} := \widetilde{m}_l$, $1 \leq l \leq k$.
\end{itemize}
\end{lemma}
\begin{proof}
See Lemma 2 and Corollary 2 in \cite{FrP1}.
\end{proof}

We have an immediate
\begin{corollary}
\label{correlabel}
If $W = W(U)$ satisfies the Legendre-Hadamard condition \eqref{LHcond} and the constant multiplicity assumption for each $U \in \M^d_+$, then system \eqref{model} is hyperbolic in the connected open domain $\cU$ of state variables. Moreover, the characteristic velocities can be relabeled as
\[
\begin{aligned}
a_1(\xi,U) &:= \tilde{a}_1(\xi,U) = - \sqrt{\kappa_k(\xi,U)}, \\
\vdots & \\
a_k(\xi,U) &:= \tilde{a}_k(\xi,U) = - \sqrt{\kappa_1(\xi,U)}, \\
a_{k+1}(\xi,U) &:= \tilde{a}_0(\xi,U) \equiv 0, \\
a_{k+2}(\xi,U) &:= \tilde{a}_{k+1}(\xi,U) =  \sqrt{\kappa_1(\xi,U)}, \\
\vdots & \\
a_{2k + 1}(\xi,U) &:= \tilde{a}_{2k}(\xi,U) =  \sqrt{\kappa_k(\xi,U)}, \\
\end{aligned}
\]
so that
\[
a_1(\xi,U) < \ldots < a_k(\xi,U) < a_{k+1}(\xi,U) = 0 < a_{k+2}(\xi,U) < \ldots < a_{2k+1}(\xi,U),
\]
for each $(\xi,U)  \in \R^d \backslash \{0\} \times \M^d_+$, denoting the $2k+1$ distinct eigenvalues of $A(\xi,U)$, with constant algebraic (and geometric) multiplicities $\widetilde{m}_l$ for $1 \leq l \leq k$, $\widetilde{m}_{k+1} = d^2 -d$ and $\widetilde{m}_{k+1+l}:= \widetilde{m}_l$ for $1 \leq l \leq k$ with $\sum_{l=1}^k \widetilde{m}_l = d$.
\end{corollary}

\subsection{Strong ellipticity of compressible Hadamard materials}
\label{secHadhyp}

Compressible hyperelastic materials of Hadamard type (cf. \cite{Hay68,Joh66})  are characterized by energy density functions of the form \eqref{Hadamardmat} where $\mu > 0$ is the constant shear modulus and $h$ is the volumetric energy density satisfying hypotheses \eqref{H1} and \eqref{H2}. In the present context, condition \eqref{H1} is a minimal regularity requirement needed for the stability calculations. Assumption \eqref{H2} is a sufficient convexity condition for the material to be strongly elliptic as we shall see below. From its definition, it is then evident that any energy density $W$ for this class of elastic materials satisfies the principles of frame indifference, material symmetry and objectivity. For a discussion on the physical model and its main properties, see Appendix \ref{haddimate}. 

Let us now compute the acoustic tensor for the class of compressible Hadamard materials and verify the Legendre-Hadamard condition in any space dimension.  It is already known that, for Hadamard materials with energy density of the form \eqref{Hadamardmat}, condition \eqref{H2} is equivalent to Legendre-Hadamard condition for all deformations (see, e.g., \cite{ArAi98,CMR94,JiKn91}). In this paper, we also provide a proof of this fact in view that the calculation of the acoustic tensor and of its eigenvalues is mandatory for the shock stability analysis (see Corollary \ref{corHadLG} below). The contributions are, (i) that our proof holds for any space dimension $d \geq 2$, and, (ii) that we also verify the constant multiplicity assumption (see Definition \ref{constmultass}). We start by proving an auxiliary result.

\begin{lemma}
\label{lemmaux}
For any $U \in \M^d_+$ with $J = \det U > 0$ there holds
\begin{equation}
\label{usefula}
\frac{\partial J}{\partial U_{ij}} = ( \Cof U)_{ij}, 
\end{equation}
for all $1 \leq i,j \leq d$, and 
\begin{equation}
\label{usefulb}
\frac{\partial }{\partial U_{qi}} (\Cof U)_{pj} = \frac{1}{J} \left( (\Cof U)_{qi} (\Cof U)_{pj} - (\Cof U)_{pi} (\Cof U)_{qj} \right), 
\end{equation}
for all $1 \leq i,j,p,q \leq d$.
\end{lemma}
\begin{proof}
Formula \eqref{usefula} follows directly from expression \eqref{exprcof} above. Let us prove \eqref{usefulb}. Differentiating the relation $\big(\Cof U\big) U^\top= J \Id$ with respect to $U_{qi}$ and multiplying from the right by $\Cof U$ we obtain
\[
\frac{\partial}{\partial U_{qi}}\big(\Cof U \big)U^\top\Cof U+\Cof U\Big(\frac{\partial}{\partial U_{qi}}U^\top\Big)\Cof U=\Big(\frac{\partial J}{\partial U_{qi}} \Big)\Cof U,
\]
that is,
\[
J \Big(\frac{\partial}{\partial U_{qi}}(\Cof U)\Big) +\Cof U\big(\ei\otimes \eq \big)\Cof U = (\Cof U)_{qi}\Cof U.
\]
Solving for $\frac{\partial}{\partial U_{qi}}(\Cof U)$ yields
\[
\frac{\partial}{\partial U_{qi}}(\Cof U)=\frac{1}{J}\left( (\Cof U)_{qi} \Cof U-\Cof U\big(\ei\otimes \eq\big)\Cof U\right),
\]
for any $1 \leq q,i \leq d$. Therefore, for all $1 \leq p,j \leq d$,
\[
\begin{aligned}
\frac{\partial }{\partial U_{qi}} (\Cof U)_{pj} &= \ep^\top \frac{\partial}{\partial U_{qi}}(\Cof U) \ej \\
&= \frac{1}{J} \left((\Cof U)_{qi} \, \ep^\top (\Cof U) \ej - \ep^\top (\Cof U)\big(\ei\otimes \eq\big)(\Cof U) \ej \right)\\
&= \frac{1}{J} \left((\Cof U)_{qi} (\Cof U)_{pj} - \ep^\top ((\Cof U) \ei )  (\eq^\top (\Cof U)) \ej \right)\\
&= \frac{1}{J} \left((\Cof U)_{qi} (\Cof U)_{pj} - (\Cof U )_{pi} (\Cof U)_{qj} \right).
\end{aligned}
\]
\end{proof}

\begin{lemma} 
\label{lemBsH}
For a compressible Hadamard material in dimension $d \geq 2$ the matrices \eqref{defBij} are given by
\begin{equation}
\label{exprBijs}
\begin{aligned}
B_i^j(U) &= \mu \, \delta_i^j\, \Id + h''(J) \big( (\Cof U)_j \otimes (\Cof U)_i \big) + \\ & \; + \frac{h'(J)}{J} \Big( (\Cof U)_j \otimes (\Cof U)_i - (\Cof U)_i \otimes (\Cof U)_j \Big),
\end{aligned}
\end{equation}
where $J = \det U > 0$, $(\Cof U)_k$ denotes the $k$-th column of the cofactor matrix $\Cof U$ and $\delta_i^j$ is the Kronecker symbol, $\delta_i^j =\begin{cases} 1, & i=j,\\ 0, &i\neq j.\end{cases}$
\end{lemma}
\begin{proof}
By definition of the matrices \eqref{defBij}, and by Corollary \ref{corPKHad} in Appendix \ref{haddimate} and Lemma \ref{lemmaux}, for each $1 \leq p, q \leq d$ there holds
\begin{equation}
\label{cof0}
\begin{split}
B_i^j(U)_{pq} =\frac{\partial \sigma_{pj}}{\partial U_{qi}} &=\mu\frac{\partial U_{pj}}{\partial U_{qi}}+\frac{\partial}{\partial U_{qi}}\Big(h'(J)(\Cof U)_{pj}\Big)\\
&= \mu\delta_i^j \delta_p^q + h''(J) \frac{\partial J}{\partial U_{qi}} (\Cof U)_{pj} + h'(J)\frac{\partial}{\partial U_{qi}}(\Cof U)_{pj}\\
&= \mu\delta_i^j \delta_p^q + h''(J)(\Cof U)_{qi}(\Cof U)_{pj} + \\ &\; \, + \frac{h'(J)}{J} \left((\Cof U)_{qi} (\Cof U)_{pj} - (\Cof U)_{pi} (\Cof U)_{qj} \right).
\end{split}
\end{equation}
Now, since 
\[
\big(\Cof U\big)_{pi}\big(\Cof U\big)_{qj} = \Big((\Cof U)_{i}\otimes(\Cof U)_{j}\Big)_{pq},
\]
for all $1 \leq i,j,p,q \leq d$, substituting into \eqref{cof0} we arrive at
\[
\begin{aligned}
B_i^j(U)_{pq} &= \mu\delta_i^j \delta_p^q + h''(J) \Big((\Cof U)_{j}\otimes(\Cof U)_{i}\Big)_{pq} + \\ & \;\; +\frac{h'(J)}{J} \left( \Big((\Cof U)_{j}\otimes(\Cof U)_{i}\Big)_{pq} - \Big((\Cof U)_{i}\otimes(\Cof U)_{j}\Big)_{pq}\right),
\end{aligned}
\]
yielding the result.
\end{proof}

\begin{corollary}
(a) In dimension $d =2$ and for each $U \in \M_2^+$ we have
\[
\begin{aligned}
B_{i}^{i} (U)&=\mu \mathbb{I}_2+h''(J)\Big((\Cof U)_{i}\otimes(\Cof U)_{i}\Big),\qquad i=1,2\\
B_{1}^{2}(U)&=h''(J)\Big((\Cof U)_{2}\otimes(\Cof U)_{1}\Big)+ h'(J)(\hat{e}_2\otimes \hat{e}_1 - \hat{e}_1\otimes \hat{e}_2)\\
B_{2}^{1}(U)&= B_{1}^{2}(U) ^\top.
\end{aligned}
\]

(b) In dimension $d=3$ and for each $U \in \M_3^+$ we have
\[
\begin{aligned}
B_{i}^{i}  (U)&=\mu \mathbb{I}_3+h''(J)\Big((\Cof U)_{i}\otimes(\Cof U)_{i}\Big),\quad i=1,2,3\\
B_{1}^{2} (U)&=h''(J)\Big((\Cof U)_{2}\otimes(\Cof U)_{1}\Big)+ h'(J)[U_3]_{\times}\\
B_{1}^{3} (U)&=h''(J)\Big((\Cof U)_{3}\otimes(\Cof U)_{1}\Big)- h'(J)[U_2]_{\times}\\
B_{3}^{2} (U)&=h''(J)\Big((\Cof U)_{2}\otimes(\Cof U)_{3}\Big)+h'(J)[U_1]_{\times}\\
B_{2}^{1} (U)&=B_{1}^{2} (U)^\top, \quad B_{3}^{1} (U)=B_{1}^{3} (U)^\top,\quad B_{2}^{3} (U)=B_{3}^{2} (U)^\top ,
\end{aligned}
\]
where, for any vector $b = (b_1, b_2, b_3)^\top \in \R^3$, $[b]_{\times}$ is the skew-symmetric matrix that represents the vector cross product, that is, $[a]_{\times} = \left( \begin{smallmatrix}
  0&-b_3 &b_2\\ b_3&0 & -b_1\\-b_2 & b_1 & 0
\end{smallmatrix} \right)$.
\end{corollary}

\begin{lemma}[acoustic tensor for Hadamard materials]
For any Hadamard material in dimension $d \geq2$ its acoustic tensor is given by
\begin{equation}
\label{exprAT}
\AT(\xi,U) = \mu |\xi|^2 \Id + h''(J) \Big( \big( (\Cof U) \xi \big) \otimes \big( (\Cof U) \xi \big) \Big),
\end{equation}
for $\xi \in \R^d$, $\xi \neq 0$, $U \in \M^d_+$.
\end{lemma}
\begin{proof}
First we notice that
\[
\begin{split}
B_{i}^{i}(U)&=\mu \Id + h''( J)\Big((\Cof U)( \ei \otimes \ei )(\Cof U)^\top\Big)\\
B_{i}^{j}(U)+B_{j}^{i}(U)&= h''( J)\Big((\Cof U)( \ei \otimes \ej +  \ej \otimes \ei )(\Cof U)^\top \Big),\quad i\neq j.
\end{split}
\]
Upon substitution of these formulae into the definition of the acoustic tensor \eqref{defacouten},
\[
\begin{split}
\AT(\xi,U)=\sum\limits_{i,j=1}^{d}\xi_i\xi_j B_{i}^{j}(U)&=\sum\limits_{i=1}^{d}\xi_{i}^{2} B_{i}^{i}(U)+\sum\limits_{i\neq j}\xi_i\xi_j\Big(B_{i}^{j}(U)+B_{j}^{i}(U)\Big)\\
&=\mu\Big(\sum\limits_{i=1}^{d}\xi_i^2\Big)\Id+h''(J) (\Cof U)\Big(\sum\limits_{i=1}^{d}\xi_{i}^{2}(\ei\otimes \ei)\Big)(\Cof U)^\top+\\& \quad \quad + h''(J) (\Cof U)\Big(\sum\limits_{i\neq j}\xi_{i}\xi_{j}(\ei\otimes \ej+\ej\otimes \ei)\Big)(\Cof U)^\top\\
&=\mu|\xi|^2 \Id+ h''(J) (\Cof U)\Big(\sum\limits_{i,j=1}^{d}\xi_{i}\xi_{j}(\ei\otimes \ej)\Big)(\Cof U)^\top\\
&=\mu|\xi|^2 \Id+ h''(J) (\Cof U)(\xi\otimes \xi)(\Cof U)^\top\\
&=\mu|\xi|^2 \Id+ h''(J) \big( (\Cof U )\xi\big) \otimes \big( (\Cof U)\xi\big),
\end{split}
\]
for all $\xi \in \R^d$, $\xi \neq 0$, $U \in \M^d_+$, as claimed.
\end{proof}

\begin{lemma}
\label{lemevQ}
For each $U \in \M^d_+$, $\xi \in \R^d$, $\xi \neq 0$, the eigenvalues of the acoustic tensor of a Hadamard material are $\kappa_1 (\xi,U) = \mu |\xi|^2$, with algebraic multiplicity equal to $d-1$, and $\kappa_2 (\xi,U) = \mu |\xi|^2 + h''(J)  \big|(\Cof U )\xi\big|^2$,
with algebraic multiplicity equal to one.
\end{lemma}
\begin{proof}
By inspection of expression \eqref{exprAT} for the acoustic tensor, which is of the form $a \Id + b ( w \otimes w)$ with $a, b \in \R$ and $w \in \R^d$, one applies Sylvester's determinant identity \cite{AAM96} to obtain
\[
\begin{aligned}
\det \big( \AT(\xi,U) - \kappa \Id \big) &= \det \Big( (\mu|\xi|^2 - \kappa) \Id + h''(J) \big( (\Cof U )\xi\big) \otimes \big( (\Cof U)\xi\big)\Big)\\
&= ( \mu |\xi|^2 - \kappa)^{d-1} \big( \mu |\xi|^2 - \kappa + h''(J)  \big|(\Cof U )\xi\big|^2 \big),
\end{aligned}
\]
yielding the result.
\end{proof}
\begin{corollary}
\label{corHadLG}
If the energy density function of an hyperelastic Hadamard material satisfies assumptions \eqref{H1} and \eqref{H2} then it satisfies the Legendre-Hadamard condition \eqref{LHcond} and the constant multiplicity assumption.
\end{corollary}
\begin{proof}
Since for all $\xi \neq 0$ the eigenvalues of the acoustic tensor are strictly positive, it clearly satisfies the Legendre-Hadamard condition \eqref{LHcond}. Regarding the constant multiplicity assumption, notice that $\kappa_2 (\xi,U)$ has algebraic and geometric multiplicities equal to one for each $U \in \M^d_+$, $\xi \neq 0$. Also notice that $(\Cof U)\xi \neq 0$ and hence $(\Cof U)\xi  \otimes (\Cof U)\xi$ has rank equal to one. This implies that the geometric multiplicity of $\kappa_1 (\xi,U)$ is $d-1$ for each $U \in \M^d_+$, $\xi \neq 0$. This shows that $\kappa_1$ is a semi-simple eigenvalue with constant multiplicity.
\end{proof}

\begin{remark}
The significance of Corollary \ref{corHadLG} is precisely that, for this large class of compressible hyperelastic materials and in any space dimension $d \geq 2$, the equations of elastodynamics are hyperbolic with constant multiplicity in the  whole open set of admissible states, $\cU = \{ (U,v) \in \R^{d \times d} \times \R^d \, : \, \det U > 0\}$, allowing us to consider elastic shocks of arbitrary amplitude.
\end{remark}

As a by-product of Lemma \ref{lemevQ} and Corollary \ref{corHadLG} we have the following
\begin{lemma}
\label{lemQev}
For each $U \in \M^d_+$, $\xi \in \R^d$, $\xi \neq 0$, the eigenvector of the acoustic tensor of a Hadamard material associated to the simple eigenvalue $\kappa_2(\xi,U) = \mu |\xi|^2 + h''(J)  \big|(\Cof U )\xi\big|^2$ is given by $w(\xi,U) := (\Cof U) \xi \in \R^{d \times 1}$.

\end{lemma}
\begin{proof}
Follows by direct computation:
\[
\begin{aligned}
\AT(\xi,U) w &=  \left[ \mu |\xi|^2 \Id + h''(J) \Big( \big( (\Cof U) \xi \big) \otimes \big( (\Cof U) \xi \big) \Big) \right] \! w\\
&= \mu |\xi|^2 w + h''(J) (w \otimes w) w\\
&= (\mu |\xi|^2  + h''(J) |w|^2) w\\
&= \kappa_2(\xi,U) w.
\end{aligned}
\]
\end{proof}

\section{Classical shock fronts for compressible Hadamard materials}
\label{secshock}

In this section we describe classical (or Lax) non-characteristic shock fronts for compressible Hadamard materials. Elastic shock front solutions of the general form \eqref{gshock} (see Appendix \ref{sechypall}) can be recast in terms of the deformation gradient and the local velocity as (cf. \cite{Co1,FrP1,Pl2}),
\begin{equation}
 \label{shock}
(U,v)(x,t) = \begin{cases} (U^-,v^-),  & x\cdot \hat{\nu} < st,\\ (U^+,v^+),  & x \cdot \hat{\nu} > st,
\end{cases}
\end{equation}
where $\hnu \in \R^d$, $|\hnu| = 1$, is a fixed direction of propagation,  $s \in \R$ is a finite shock propagation speed and $(U_\pm, v_\pm) \in \M_+^d \times \R^d$ are constant values for the deformation gradient and local velocity satisfying $(U^+, v^+) \neq (U^-, v^-)$. The dynamics of such fronts are determined by the classical Rankine-Hugoniot jump conditions \eqref{gRH}. Since
\begin{equation}
\label{jumpstatevar}
\llb u \rrb = \begin{pmatrix}
\llb U_1 \rrb \\ \vdots \\ \llb U_d \rrb \\ \llb v \rrb
\end{pmatrix}, \qquad \llb f^j(u) \rrb = - \,  \begin{pmatrix} 0 \\ \vdots \\  \llb v \rrb \\ \vdots \\ 0 \\ \llb \sigma(U)_j \rrb
\end{pmatrix}, \;\;\; \text{for all } \; 1 \leq j \leq d,
\end{equation}
then it is easy to verify that the Rankine-Hugoniot conditions \eqref{gRH} take the form (see \cite{FrP1,Pl2})
\begin{equation}
 \label{RHcond}
\begin{array}{r}
-s \llb U \rrb - \llb v \rrb \otimes \hnu = 0,\\
-s \llb v \rrb - \llb \sigma(U) \rrb \hnu = 0, 
\end{array}
\end{equation}
expressing conservation across the interface, together with the additional jump conditions
\begin{equation}
 \label{curljump}
\llb U \rrb \times \hnu = 0,
\end{equation}
expressing the constraint \eqref{curlfree}. The jump conditions \eqref{RHcond} determine the shock speed $s \in \R$ uniquely.

In addition, thanks to Lemma \ref{LHimpliesHyp} and Corollary \ref{correlabel}, (strict) Lax entropy conditions \eqref{gLax} hold if there exists an index $p$ such that
\[
\begin{aligned}
a_{p-1}(\hat{\nu}, U^-) & < s  < a_p(\hat{\nu}, U^-),\\
a_{p}(\hat{\nu},U^+) & < s  < a_{p+1}(\hat{\nu},U^+),
\end{aligned}
\]
where $1 \leq p \leq 2k+1$ and $a_l(\hat{\nu},U)$, $1 \leq l \leq 2k+1$ denote the $2k + 1$ distinct eigenvalues of $A(\hat{\nu},U)$ as relabeled in Corollary \ref{correlabel}. In other words, to have strict inequalities in \eqref{gLax} we require the shock speed to be \emph{non-sonic} and to lie in between distinct characteristic velocities. 

The nonlinear stability behavior of the configuration solution \eqref{shock} is controlled by the Lopatinski\u{\i} conditions discussed in Section \ref{sechypall} and it is based on the normal modes analysis of solutions to the linearized problem around the shock front. Such conditions determine whether small perturbations impinging on the shock interface produce solutions to the nonlinear elastodynamics equations \eqref{model} which remain close and are qualitatively similar to the shock front solution (well-posedness of the associated Cauchy problem with piecewise smooth initial data). Thanks to finite speed of propagation and since we are interested in the local-in-space, local-in-time evolution near the shock interface, from this point on we assume that the reference configuration is the whole Euclidean space, $\Omega = \R^d$, without loss of generality.

Following \cite{FrP1}, we make some simplifying assumptions. For concreteness and without loss of generality we assume that the shock front propagates in the normal direction of the half plane $\{ x_1 = 0\}$ and, hence, $\hat{\nu} = \hat{e}_1$. Thus, the shock front solution \eqref{shock} has now the form
\begin{equation}
 \label{shocke1}
(U,v)(x,t) = \begin{cases}
(U^-,v^-),  & x_1 < st,\\ 
(U^+,v^+),  & x_1 > st,
\end{cases}
\end{equation}
where $(U^+, v^+) \neq (U^-, v^-)$ and it satisfies Rankine-Hugoniot jump conditions \eqref{RHcond} together with the curl-free jump conditions \eqref{curljump}. In this case with $\hat{\nu} = \hat{e}_1$, these conditions now read
\begin{equation}
\label{RHe1}
\begin{aligned}
-s \llb U_1 \rrb - \llb v \rrb &= 0, \\
-s \llb v \rrb - \llb \sigma(U)_1 \rrb &= 0,\\
\llb U_j \rrb &= 0, \quad \text{for all } \: j \neq 1. 
\end{aligned}
\end{equation}

In view of Lemma \ref{lemevQ}, let us define (with a slight abuse of notation)
\begin{equation}
\label{defkaps}
\begin{aligned}
\kappa_1(U) &:= \kappa_1(\eu,U) = \mu,\\
\kappa_2(U) &:= \kappa_2(\eu,U) = \mu + h''(J) \big| (\Cof U)_1 \big|^2,
\end{aligned}
\qquad U \in \M^d_+,
\end{equation}
denoting the two distinct semi-simple eigenvalues of the acoustic tensor $\AT(\eu,U)$ with constant multiplicities $\widetilde{m}_1 = d-1$ and $\widetilde{m}_2 = 1$, respectively. Henceforth, the (distinct) characteristic velocities defined in Lemma \ref{LHimpliesHyp} and Corollary \ref{correlabel} are described in Table \ref{tableevalues} below.

\begin{table}[h]
\caption{Distinct semi-simple eigenvalues $a_j(U)$ defined in Lemma \ref{LHimpliesHyp} and Corollary \ref{correlabel} with their corresponding constant multiplicities $m_j$.}
\begin{center}\begin{tabular}{p{6.5cm}p{5.5cm}}
Eigenvalue $a_j$								& 	Algebraic multiplicity $m_j$	\\
\noalign{\smallskip}\hline\noalign{\smallskip}
$a_1 (U) = - \sqrt{\mu + h''(J) \big| (\Cof U)_1 \big|^2}$	& $m_1 =	1$\\	
$a_2 (U) = - \sqrt{\mu}$							& $m_2 = d-1$	\\	
$a_3 (U) = 0$									& $m_3 = d^2 -d$	\\	
$a_4(U) = \sqrt{\mu}$							& $m_4 = d-1$	\\	
$a_5(U) = \sqrt{\mu + h''(J) \big| (\Cof U)_1 \big|^2}$		& $m_5 = 1$	\\	
\noalign{\smallskip}\hline\noalign{\smallskip}
\end{tabular}\end{center}
\label{tableevalues}
\end{table}

An important consequence of the structure of the characteristic fields is the following
\begin{lemma}
\label{lemonlyextr}
For compressible Hadamard materials, all Lax shock fronts are necessarily extreme.
\end{lemma}
\begin{proof}
From the expressions for the characteristic velocities computed above, it is clear that $D_u a_2 = D_u a_3 = D_u a_4 = 0$ for all $U \in \M^d_+$. Therefore, the $j$-characteristic fields with $j =2,3,4$ are linearly degenerate.  In such cases weak solutions of form \eqref{shock} correspond to contact discontinuities for which $a_j(U^+) = s = a_j(U^-)$. Hence, any classical, non-characteristic shock that satisfies strict Lax entropy conditions \eqref{gLax} is necessarily associated to an extreme characteristic field with $j = 1$ or $j = 5$.
\end{proof}

For convenience, let us denote the characteristic fields evaluated at the constant states at each side of the shock as
\[
\begin{aligned}
\kappa_i^\pm &:= \kappa_i(U^\pm), \qquad i =1,2,\\
a_j^\pm &:= a_j(U^\pm), \qquad j=1, \ldots, 5
\end{aligned}
\]
so that
\[
a_1^\pm = - \sqrt{\kappa_2^\pm}, \quad
a_2^\pm = - \sqrt{\mu},\quad
a_3^\pm = 0, \quad
a_4^\pm = \sqrt{\mu},\quad
a_5^\pm = \sqrt{\kappa_2^\pm}.
\]

In view of Lemma \ref{lemonlyextr} a strict classical shock is necessarily associated to an extreme principal characteristic field with index either $p=1$ or $p=5$. For concreteness and without loss of generality, we assume from this point on that the shock front \eqref{shocke1} is an extreme Lax shock associated to the first characteristic field, $p = 1$, or, in short, a $1$-shock (see also \cite{FSz11}). In such a case, Lax entropy conditions \eqref{gLax} read
\begin{equation}
\label{Laxece1}
\begin{aligned}
s &< a_1^-,\\
a_1^+ < s &< a_2^+,
\end{aligned}
\end{equation}
or equivalently,
\begin{equation}
\label{Laxece1v2}
\begin{aligned}
s &< - \sqrt{\mu + h''(J^-) \big| (\Cof U^-)_1 \big|^2},\\
- \sqrt{\mu + h''(J^+) \big| (\Cof U^+)_1 \big|^2} < s &< - \sqrt{\mu}.
\end{aligned}
\end{equation}
Notice, in particular, that these conditions imply that $s < 0$ and $s^2 \neq \mu$.

\begin{lemma}
\label{lemsimpRH}
Consider an elastic 1-shock, $(U^\pm, v^\pm, s)$, for a compressible Hadamard material, with $(U^+,v^+) \neq (U^-,v^-)$, $J^\pm = \det U^\pm > 0$, $s^2 \neq \mu$, satisfying Rankine-Hugoniot conditions \eqref{RHe1} and Lax entropy conditions \eqref{Laxece1v2}. Then there exists a parameter value, $\alpha \in \R$, $\alpha \neq 0$, such that
\begin{equation}
\label{RHmod}
\begin{aligned}
\llb U \rrb &= \alpha\big((\Cof U^+)_1\otimes \eu \big),\\
\llb J \rrb &= \alpha \big| \big(\Cof U^+\big)_1\big|^2.\\
\end{aligned}
\end{equation}
Moreover, the shock speed satisfies 
\begin{equation}
\label{speedsq}
s^2 = \mu + \frac{1}{\alpha} \llb h'(J) \rrb.
\end{equation}
\end{lemma}
\begin{proof}
From expression \eqref{HadPiolKir}, the jump of the Piola-Kirchhoff stress tensor across the shock is given by
\[
\llb \sigma(U) \rrb = \mu \llb U \rrb + h'(J^+) \Cof U^+ - h'(J^-) \Cof U^-.
\]
Therefore, the jump of its first column across the shock is
\[
\llb \sigma(U)_1 \rrb = \mu \llb U_1 \rrb + h'(J^+) (\Cof U^+)_1 - h'(J^-) (\Cof U^-)_1.
\]
From jump conditions \eqref{RHe1} we know that $U_j^+ = U^-_j$ for all $j \neq 1$. This implies that
\begin{equation}
\label{jumpCofU1}
(\Cof U^+)_1 = (\Cof U^-)_1.
\end{equation}
Making use of jump relations \eqref{RHe1} we arrive at
\begin{equation}
\label{eqrh}
(s^2-\mu) \llb U_1 \rrb - \llb h'(J) \rrb (\Cof U^+)_1=0.
\end{equation}

By hypothesis $s^2 \neq \mu$ (it is a Lax shock) and hence $\llb h'(J) \rrb \neq 0$ (otherwise one would have $\llb U_1 \rrb = 0$ and $\llb v \rrb = 0$, a contradiction with $(U^+, v^+) \neq (U^-, v^-)$). This shows that the vectors $\llb U_1 \rrb$ and $(\Cof U^+)_1$ are linearly dependent. Therefore there exists $\alpha \neq 0$ such that
\[
\llb U_1 \rrb = \alpha (\Cof U^+)_1.
\]
The jump condition $\llb U_j \rrb = 0$ for $j \neq 1$ implies that
\[
U^+ = U^- + \alpha\big((\Cof U^+)_1\otimes \eu \big),
\]
yielding the first relation in \eqref{RHmod}. Substitute $\llb U_1 \rrb = \alpha(\Cof  U^+)_1 \neq 0$ in \eqref{eqrh} to obtain \eqref{speedsq}.

Finally, from \eqref{exprcof} we clearly have the relation $J = \eu^\top (J \Id) \eu = \eu^\top U^\top (\Cof U) \eu = U_1^\top (\Cof U)_1$ and, therefore,
\[
\begin{aligned}
J^- = (U^-_1)^\top (\Cof U^-)_1 &= (U^-_1)^\top (\Cof U^+)_1 \\&= (U_1^+ - \alpha (\Cof U^+)_1)^\top (\Cof U^+)_1 \\&= (U_1^+)^\top (\Cof U^+)_1 - \alpha \big| (\Cof U^+)_1 \big|^2 \\&= J^+ - \alpha \big| (\Cof U^+)_1 \big|^2,
\end{aligned}
\]
yielding the second formula in \eqref{RHmod}. This shows the lemma.
\end{proof}

\begin{remark}
\label{remalpss}
Suppose that one selects $(U^+, v^+) \in \M^d_+ \times \R^d$ as a base state. Lemma \ref{lemsimpRH} then implies that the shock is completely determined by the parameter value of $\alpha \neq 0$, which measures the strength of the shock, that is, $\llb U \rrb , \llb v \rrb = O(|\alpha|)$. Indeed, given $(U^+,v^+) \in \M_+^d \times \R^d$ and $\alpha \neq 0$, we apply Rankine-Hugoniot and Lax entropy conditions to define
\[
\begin{aligned}
U^- &:= U^+ - \alpha \big( (\Cof U^+)_1 \otimes \hat{e}_1 \big),\\
J^\pm &:= \det U^\pm,\\
s &:= - \sqrt{\mu + \frac{1}{\alpha} \llb h'(J) \rrb},\\
v^- &:= v^+ + s \alpha (\Cof U^+)_1.
\end{aligned}
\]
Then, on one hand, it is clear that $| \llb U \rrb | = | \llb U_1 \rrb | = |\alpha| \big| (\Cof U^+)_1 \big| = O(|\alpha|)$. On the other hand, $J^- = J^+ - \alpha \big| (\Cof U^+)_1\big|^2$ yields
\[
s^2 = \mu+\frac{1}{\alpha} \llb h'(J) \rrb = \mu + h''(J^+) \big| (\Cof U^+)_1\big|^2 + O(|\alpha|).
\]
Upon substitution we obtain $\llb v \rrb^2 = s^2 \alpha^2  \big| (\Cof U^+)_1\big|^2 = O(\alpha^2)$. This proves the claim. It is to be noticed, as well, that once the base state $(U^+, v^+) \in \M^d_+ \times \R^d$ is selected, then the value of $\alpha$ ranges within the set
\[
\alpha \in (-\infty,0) \cup (0, \alpha_{\mathrm{max}}^+),
\]
where
\begin{equation}
\label{defalpmax}
\alpha_{\mathrm{max}}^+ := \frac{J^+}{\big| (\Cof U^+)_1 \big|^2}\, ,
\end{equation}
due to the physical requirement that $\det U^- = J^- > 0$. Observe in particular that, necessarily, $J^+ \neq J^-$ as $\alpha \neq 0$.
\end{remark}
\begin{remark}
\label{remgoodsignalp}
Thanks to the convexity condition $\eqref{H2}$ we have that
\[
 \frac{1}{\alpha} \llb h'(J) \rrb = \frac{1}{\alpha} \big( h'(J^+) - h'(J^+ - \alpha \big| (\Cof U^+)_1\big|^2) \big) \, > 0
\]
independently of the sign of $\alpha$, because $h'(J)$ is strictly increasing. Therefore, if Lax entropy conditions \eqref{Laxece1v2} hold then
\begin{equation}
\label{hstarrel}
0 < h''(J^-) < \frac{s^2 - \mu}{\big| (\Cof U^+)_1\big|^2} < h''(J^+),
\end{equation}
where we have used the fact that $(\Cof U^+)_1 = (\Cof U^-)_1$. Let us denote the open interval
\[
I(J^+, J^-) := \begin{cases}
(J^-, J^+), & \text{if } \; J^+ > J^-,\\
(J^+, J^-), & \text{if } \; J^+ < J^-.
\end{cases}
\]
From the observations above, we conclude that the following statements hold: 
\begin{itemize}
\item[(a)] If $h'''(J) > 0$ for all $J \in I(J^+, J^-)$ ($h''$ increasing) then Lax entropy conditions hold if $0 < \alpha < \alpha_{\mathrm{max}}^+$.
\item[(b)] If $h'''(J) < 0$ for all $J \in I(J^+, J^-)$ ($h''$ decreasing) then Lax entropy conditions hold if $\alpha < 0$.
\end{itemize}
\end{remark}

Next lemma verifies that the requirement for $h'''$ to have a definite sign on $I(J^+, J^-)$ is also a necessary condition to have a genuinely nonlinear characteristic field.
\begin{lemma}
\label{lemgnl}
For any $U \in \M_+^d$, let $r \in \R^n$ be the right eigenvector of $A(\eu,U)$ associated to the simple eigenvalue $a_1(U) = a_1(\eu,U) < 0$ in the case of a compressible Hadamard material. Then,
\[
(D_u a_1)^\top r = \frac{1}{2 a_1^2} | (\Cof U)_1 |^4 h'''(J).
\]
\end{lemma}
\begin{proof}
First, let us denote $r = (z_1, \, \ldots, z_d, \, w)^\top \in \R^{n \times 1 }$, $n = d^2 +d$, with $z_j, w \in \R^d$, $1 \leq j \leq d$, the right eigenvector such that $A(\eu,U) r = A^1(U) r = a_1(U) r$, with $a_1(U) = - \sqrt{\kappa_2(\eu,U)} < 0$. Upon inspection of the expression for $A^1(U)$ we observe that
\[
A^1(U) r = - \begin{pmatrix} 
w \\ 0 \\ \vdots \\ 0 \\ \sum_{j=1}^d B^1_j(U) z_j
\end{pmatrix}
= a_1(U) \begin{pmatrix} 
z_1 \\ \vdots \\ z_d \\ w
\end{pmatrix},
\]
or, equivalently, we obtain the system
\begin{equation}
\label{eigsyst}
\begin{aligned}
w + a_1 z_1 &= 0,\\
a_1 z_j &= 0, \quad j \neq 1,\\
a_1 w + \sum_{j=1}^d B^1_j z_j &= 0.
\end{aligned}
\end{equation}
From this system of equations we obtain $\AT(\eu,U) w = a_1(U)^2 w = \kappa_2(\eu,U) w$, and $z_j \equiv 0$ for all $j \neq 1$. Therefore, from Lemma \ref{lemQev} we arrive at the following expression for the right eigenvector,
\[
r = \begin{pmatrix}
- (a_1)^{-1} (\Cof U)_1 \\ 0 \\ \vdots \\ 0 \\ (\Cof U)_1
\end{pmatrix}.
\]

Now, let us write $a_1(U) = -\sqrt{\psi(U)}$, where $\psi(U) := \mu + h''(J) \big| (\Cof U)_1 \big|^2$. Since, clearly, $\partial \psi / \partial v = 0$, we then have
\[
D_u a_1 = \frac{1}{2a_1} \begin{pmatrix}
\psi_{U_1} \\ \vdots \\ \psi_{U_d} \\ 0
\end{pmatrix},
\]
where $\psi_{U_j} \in \R^d$ is the vector whose $i$-component is $\partial \psi / \partial U_{ij}$ for each pair $i,j$. Let us compute such derivatives. Use relations \eqref{usefula} and \eqref{usefulb} to obtain
\[
\begin{aligned}
\frac{\partial \psi}{\partial U_{ij}} &= h'''(J) \frac{\partial J}{\partial U_{ij}} \big| (\Cof U)_1 \big|^2 + h''(J) \frac{\partial}{\partial U_{ij}} \Big( \big| (\Cof U)_1 \big|^2 \Big)\\ 
&= h'''(J) \big| (\Cof U)_1 \big|^2 (\Cof U)_{ij} + 2 h''(J) \sum_{k=1}^d (\Cof U)_{k1} \frac{\partial}{\partial U_{ij}} \big( (\Cof U)_{k1}\big)\\
&= h'''(J) \big| (\Cof U)_1 \big|^2 (\Cof U)_{ij} + \\
& \; \; + 2 \frac{h''(J)}{J} \sum_{k=1}^d (\Cof U)_{k1} \Big( (\Cof U)_{k1} (\Cof U)_{ij} - (\Cof U)_{kj} (\Cof U)_{i1}\Big)\\
&= \Big( h'''(J) + \frac{2}{J} h''(J) \Big)  \big| (\Cof U)_1 \big|^2 (\Cof U)_{ij} -  2 \frac{h''(J)}{J} (\Cof U)_{i1} \sum_{k=1}^d (\Cof U)_{kj} (\Cof U)_{k1},
\end{aligned}
\]
 for each $1 \leq i,j \leq d$. Therefore, $D_u a_1 = \varsigma_1 + \varsigma_2$ with
\[
\varsigma_1 := \frac{1}{2a_1} \Big( h'''(J) + \frac{2}{J} h''(J) \Big)\big| (\Cof U)_1 \big|^2 \begin{pmatrix} (\Cof U)_1 \\ \vdots \\ (\Cof U)_d \\ 0 \end{pmatrix},
\]
\[
\varsigma_2 := - \,  \frac{1}{a_1} \frac{h''(J)}{J} \begin{pmatrix} \big[\sum_{k=1}^d (\Cof U)_{k1}^2\big] (\Cof U)_1 \\ \big[\sum_{k=1}^d (\Cof U)_{k2} (\Cof U)_{k1} \big] (\Cof U)_1 \\ \vdots  \\ \big[\sum_{k=1}^d (\Cof U)_{kd} (\Cof U)_{k1} \big] (\Cof U)_1 \\ 0 \end{pmatrix}.
\]
Computing the products with $r$ yields
\[
\begin{aligned}
\varsigma_1^\top r &= - \, \frac{1}{2a_1^2} \big| (\Cof U)_1 \big|^4 \Big( h'''(J) + \frac{2}{J} h''(J) \Big), \\
\varsigma_2^\top r &= \frac{1}{a_1^2} \frac{h''(J)}{J} \sum_{k=1}^d (\Cof U)_{k1}^2 \big| (\Cof U)_1 \big|^2 = \frac{1}{a_1^2} \frac{h''(J)}{J} \big| (\Cof U)_1 \big|^4.
\end{aligned}
\]
Hence, we arrive at
\[
(D_u a_1)^\top r  = - \, \frac{1}{2 a_1^2} | (\Cof U)_1 |^4 h'''(J),
\]
as claimed.
\end{proof}

\begin{corollary}
\label{corgnliffhtp}
The $1$-characteristic field is genuinely nonlinear in the $\eu$-direction for all state variables $(U,v) \in \cU$ if and only if $h'''(J) \neq 0$ for all $J \in (0, \infty)$.
\end{corollary}

\begin{remark}
As expected, being the choice of $\eu$ as direction of propagation completely arbitrary, it is possible to extrapolate this observation and to state that the $j=1$ and the $j=5$ characteristic fields are genuinely nonlinear in any direction of propagation $\hat{\nu} \in \R^d$, $|\hat{\nu}| = 1$, for all state variables $(U,v) \in \cU$ if and only if $h'''(J) \neq 0$ for all $J \in (0, \infty)$. In fact, a similar calculation yields
\[
(D_u a_j)^\top r = - \, \frac{1}{2 a_j^2} | (\Cof U) \hat{\nu} |^4 \, h'''(J), 
\]
for $j=1,5$ as the dedicated reader may verify.
\end{remark}

Consequently, we have the following characterization of the $1$-shock fronts in terms of the parameter $\alpha \neq 0$.
\begin{proposition}
\label{proppadonde}
For a Hadamard material satisfying \eqref{H1} and \eqref{H2} and for any given $(U^+, v^+) \in M_+^d \times \R^d$ as base state, let us define, for any given $\alpha \in (-\infty, 0) \cup (0, \alpha_{\mathrm{max}}^+)$, 
\begin{equation}
\label{alpshockcurve}
\begin{aligned}
U^- &= U^+ - \alpha ( (\Cof U^+)_1 \otimes \eu), \\
v^- &= v^+ + s \alpha (\Cof U^+)_1,\\
s &= - \sqrt{ \mu + \frac{1}{\alpha} (h'(J^+) - h'(J^-)) },
\end{aligned}
\end{equation}
for which, necessarily, $J^- = \det U^- = J^+ - \alpha | (\Cof U^+)_1 |^2$. Therefore we have:
\begin{itemize}
\item[(a)] In the case where $0 < \alpha < \alpha_{\mathrm{max}}^+$: if $h'''(J) > 0$ for all $J \in [J^-, J^+]$ then $(U^\pm, v^\pm,s)$ is a Lax 1-shock.
\item[(b)] In the case where $\alpha < 0$: if $h'''(J) <0$ for all $J \in [J^+, J^-]$ then $(U^\pm, v^\pm,s)$ is a Lax 1-shock.
\end{itemize}
\end{proposition}
\begin{proof}
Suppose $0 < \alpha < \alpha_{\mathrm{max}}^+$. If $h'''(J) > 0$ for all $J \in [J^-, J^+]$ then from \eqref{H2} and $\llb h'(J) \rrb/\alpha > 0$ we deduce that $s < - \sqrt{\mu}$. Also, from strict convexity of $h'$ and $J^+ > J^-$ we clearly have
\[
h''(J^+) > \frac{\llb h'(J) \rrb}{\alpha |(\Cof U^+)_1|^2},
\]
from which we deduce $- \sqrt{\mu + h''(J^+) |(\Cof U^+)_1|^2} < s$. A similar argument shows that $ s < - \sqrt{\mu + h''(J^-) |(\Cof U^-)_1|^2}$. Hence, the front is a Lax 1-shock. This proves (a). The proof of (b) is analogous.
\end{proof}

\begin{remark}
\label{remstrongshocks}
Observe that \eqref{alpshockcurve} determines the 1-shock curve (see \eqref{gshockcurve} in Appendix \ref{sechypall}) for all admissible values of $\alpha$ and not only for \emph{weak} shocks. Hence, we are able to consider shocks of \emph{arbitrary amplitude}, as there is no other restriction on $|\alpha|$ apart from the physical constraint $0 < \alpha < \alpha_{\mathrm{max}}^+$ on the positive side. For compressible Hadamard materials satisfying \eqref{H3} ($h''' < 0$ for all $J$), it is posible to construct arbitrarily large amplitude shocks for negative parameter values, $\alpha < 0$, with $ |\alpha| \gg 1$. It is to be observed that condition \eqref{H3} can be interpreted as the convexity of the hydrostatic pressure (see Remark \ref{remmatconvex} below) and, hence, the case in which $h''' > 0$ for \emph{all} $J$ turns out to be somewhat unphysical: most examples of energy densities in the literature (see, for example, section \S \ref{subsecexamples}) satisfy \eqref{H3} or, at most, they present changes in sign for $h'''(J)$. For simplicity, the latter concave/convex case is not considered in the stability analysis.
\end{remark}

\section{Normal modes analysis for elastic shocks}
\label{secnmodes}

In this section we perform the normal modes analysis prior to the establishment of the stability results. In particular, we compute all the necessary ingredients to assemble the Lopatinski\u{\i} determinant associated to a classical shock front (as described in Appendix \ref{sechypall}) for hyperelastic Hadamard materials.

Let $(U^\pm, v^\pm,s) \in \M_+^d \times \R^d \times \R$, with $(U^+,v) \neq (U^-,v^-)$ be an extreme Lax 1-shock propagating in the direction of $\hat{\nu} = \eu$ and satisfying Rankine-Hugoniot conditions \eqref{RHe1} and Lax entropy conditions \eqref{Laxece1v2}. Therefore, the analysis of normal mode solutions to the linearized problem around the shock of the form $e^{\lambda t} e^{\ii x \cdot \xi}$ is restricted to the open set of spatio-temporal frequencies,
\begin{equation}
\label{defGamp}
\Gamma^+ := \left\{ (\lambda, \txi) \in \C \times \R^{d-1} \, : \, \Re \lambda > 0, \, |\lambda|^2 + |\txi|^2 = 1\right\},
\end{equation}  
(see \eqref{genfreq}), where we have adopted the (now customary in the literature \cite{BS,JL}) notation for the Fourier frequencies,
\[
\xi = \begin{pmatrix}
0 \\ \txi
\end{pmatrix} \in \R^d, \qquad \txi = \begin{pmatrix} \xi_2 \\ \vdots \\ \xi_d
\end{pmatrix} \in \R^{d-1},
\]
with $\xi \cdot \eu = \xi^\top \eu = 0$. By a continuity of eigenprojections argument (cf. \cite{Kre70,M1,Me1}) the definition of the Lopatinski\u{\i} determinant on $\Gamma^+$ can be extended to its closure,
\begin{equation}
\label{defGam}
\Gamma := \left\{ (\lambda, \txi) \in \C \times \R^{d-1} \, : \, \Re \lambda \geq 0, \, |\lambda|^2 + |\txi|^2 = 1\right\}.
\end{equation} 
We are interested in normal modes of the matrix field
\begin{equation}
\label{caliA}
\cA(\lambda,\xim,U) = \Big(\lambda \In + \ii \sum\limits_{j\neq 1}\xi_{j} A^j(U) \Big) \Big(A^1(U) - s \In \Big)^{-1}, \qquad (\lambda, \xim, U) \in \Gamma^+ \times \M^d_+,
\end{equation}
under the assumption that $s \in \R$ is not characteristic with respect to $(\eu, U)$, that is, $s$ is not an eigenvalue of $A^1(U)$. This is particularly true in the case of the shock speed $s$ of a classical 1-shock with $U = U^\pm$.

\subsection{Calculation of the stable left bundle}

Following \cite{FrP1,FrP3,Pl2} and for convenience in the calculations to come, let us extend the definition of the acoustic tensor to allow complex frequencies. For each $(\omega, \widetilde{\omega}) \in \C \times \C^{d-1}$, $\omega_1 = \omega$, $\widetilde{\omega} = (\omega_2, \ldots, \omega_d)^\top$, we denote
\[
\tAT(\omega, \widetilde{\omega},U) := \sum_{i,j=1}^d \omega_i \omega_j B^i_j(U) = \omega^2 B_1^1(U) + \omega \sum_{j \neq 1} \omega_j \big( B^j_1(U) + B^1_j(U)\big) + \sum_{i,j \neq 1}^d \omega_i \omega_j B^i_j(U) \in \C^{d \times d}.
\]

Notice that, in view that the real acoustic tensor $\AT$ is symmetric, then $\tAT$ is endowed with the property $\tAT^*(\omega, \widetilde{\omega},U) = \tAT(\omega^*, \widetilde{\omega}^*,U)$. Yet, $\tAT$ is clearly invariant under simple transposition
\[
\tAT(\omega, \widetilde{\omega},U)^\top = \tAT(\omega, \widetilde{\omega},U),
\]
for all $(\omega, \widetilde{\omega},U) \in \C \times \C^{d-1} \times \M_+^d$, even though it is not Hermitian. Adopting this notation and from expression \eqref{exprAT} for a compressible Hadamard material, we readily obtain the following useful formula,
\begin{equation}
\label{tQbxiU}
\tAT(\ii \beta ,\xim, U) = \mu \big( - \beta^2 + |\xim|^2\big) \Id + h''(J) \left[ \Big( \Cof U \Big) \begin{pmatrix} \ii \beta \\ \xi_2 \\ \vdots \\ \xi_d \end{pmatrix} \otimes \Big( \Cof U \Big) \begin{pmatrix} \ii \beta \\ \xi_2 \\ \vdots \\ \xi_d \end{pmatrix}\right],
\end{equation}
for any $\beta \in \C$, $\txi \in \R^{d-1}$, $U \in \M_+^d$.

Next result characterizes the eigenvalues of the matrix field \eqref{caliA}.
\begin{lemma} 
\label{lemevalCA}
For any given $U \in \M^d_+$, $(\lambda, \xim) \in \Gamma$, the eigenvalues $\beta = \beta(\lambda,\txi) \in \C$ of matrix \eqref{caliA} are either:
\begin{itemize}
\item[(a)] $\beta=-\dfrac{\lambda}{s}$, with algebraic multiplicity $d^2 -d$; or
\item[(b)] $\beta$ is a root of
\begin{equation}\label{eqa}
\det\big((\lambda+\beta s)^2 \Id + \tAT(\ii \beta ,\xim, U)\big)=0.
\end{equation}
\end{itemize}
\end{lemma}

\begin{proof}
Given  $(\lambda,\txi,U) \in \M_+^d \times \Gamma^+$, we look for a left (row) eigenvector $l = l(\lambda,\txi,U) \in \C^{1 \times n}$, associated to an eigenvalue $\beta$ satisfying
\begin{equation}
\label{lAbetl}
l\Big((\lambda+\beta s)\In-\beta A^1(U) + \ii \sum_{j\neq1}\xi_{j}A^j(U)\Big)=0.
\end{equation}
Since $l\neq 0$ we arrive at the following characteristic equation,
\[
\phi(\lambda,\xim,\beta,U) := \det\Big((\lambda+\beta s)\In-\beta A^1(U) + \ii \sum_{j\neq1}\xi_{j}A^j(U)\Big)=0.
\]
The matrix appearing in last equation can be written in block form as
\begin{equation}
\label{detg}
(\lambda +\beta s)\In -\beta A^1(U) + \ii \sum\limits_{j\neq1}\xi_{j}A^j(U) =
\begin{pmatrix} 
 & & & \beta \Id \\
 & (\lambda + \beta s) \Idd & & - \ii \xi_2 \Id \\
 & & & \vdots \\
 & & & - \ii \xi_d \Id \\
- \cG_1 &  \cdots & - \cG_d  & (\lambda + \beta s) \Id
\end{pmatrix} \\
=: \begin{pmatrix}
\cS_1 & \cS_2 \\ \cS_3 & \cS_4
\end{pmatrix},
\end{equation}
with blocks $\cS_1 \in \C^{d^2 \times d^2}$, $\cS_2 \in \C^{d^2 \times d}$, $\cS_3 \in \C^{d \times d^2}$, $\cS_4 \in \C^{d \times d}$, and where the matrix fields $(\beta, \txi,U) \mapsto \cG_k$ are defined as 
\begin{equation}
\label{defGk}
\cG_k = \cG_k(\beta, \xim, U) := - \beta B_k^1(U) + \ii \sum_{j \neq 1} \xi_j B_k^j(U) \; \in \C^{d \times d}.
\end{equation}
Suppose for the moment that $\lambda+\beta s\neq0$. Then we may use the block formula
\[
\det \begin{pmatrix}
\cS_1 & \cS_2 \\ \cS_3 & \cS_4
\end{pmatrix} = \det \cS_1 \, \det (\cS_4 - \cS_3 (\cS_1)^{-1} \cS_2),
\]
to reduce the determinant of \eqref{detg}. A direct computation shows that
\[
\cS_3 (\cS_1)^{-1} \cS_2 =
(\lambda + \beta s)^{-1} \left( \cG_1, \cdots,  \cG_d \right)\left(\begin{array}c
-\beta \Id\\
\ii \xi_2 \Id\\
\vdots\\
\ii \xi_d \Id
\end{array}\right) 
= \,- \, (\lambda + \beta s)^{-1} \tAT(\ii \beta ,\xim, U),
\]
yielding
\[
\phi(\lambda,\xim,\beta,U) = (\lambda + \beta s)^{d^2 -d} \det\big((\lambda+\beta s)^2 \Id + \tAT(\ii \beta ,\xim, U)\big).
\]
From this expression we conclude that $\beta = - \lambda/s$ is an eigenvalue of \eqref{caliA} with algebraic multiplicity $d^2 -d$. Otherwise, if $\lambda + \beta s \neq 0$ then $\beta$ is a root of equation \eqref{eqa}. The lemma is proved.
\end{proof}

The following lemma provides an expression for the left (row) eigenvector associated to any eigenvalue $\beta$ of the matrix field \eqref{caliA}.
\begin{lemma}
\label{lemforml}
For given $U \in\M^d_+$, $(\lambda, \xim) \in \Gamma$, let $\beta \in \C$ be an eigenvalue of the matrix \eqref{caliA} such that $\lambda +\beta s \neq 0$. Then the associated left eigenvector $l$ has the form
\begin{equation}
\label{genforml}
l = \Big( q^\top \cG_1, \, \ldots, \, q^\top \cG_d, (\lambda + \beta s) q^\top \Big) \in \C^{1 \times (d^2 + d)},
\end{equation}
where $\cG_k = \cG_k(\beta, \xim,U)$, $1 \leq k \leq d$, are defined in \eqref{defGk} and $q \in \C^{d \times 1}$ is a column vector such that
\begin{equation}
\label{qTQ}
\tAT(\ii \beta ,\xim, U) q = - (\lambda + \beta s)^2 q,
\end{equation}
that is, $q$ is an eigenvector of $\tAT(\ii \beta ,\xim, U)$ with eigenvalue $- (\lambda + \beta s)^2$.
\end{lemma}
\begin{proof}
Take $U \in\M^d_+$, $(\lambda, \xim) \in \Gamma$ and let $\beta \in \C$ be an eigenvalue of $\cA$ with associated left eigenvector $l \in \C^{1 \times (d^2 + d)}$. Consider the matrix fields
\[
\cT = \cT(\lambda, \xim,U,\beta) := \beta A^1(U) - \ii \sum_{j \neq 1} \xi_j A^j(U) \in \C^{n \times n},
\]
with $n = d^2 + d$. Since $\C^n = \ker (\cT^\top) \oplus \mathrm{range} \, (\cT^\top)$ then either $l^\top \in \ker (\cT^\top)$ or $l^\top \in \mathrm{range} \, (\cT^\top)$. However, from $l \cA = \beta l$ we clearly have that expression \eqref{lAbetl} holds, yielding $\cT^\top l^\top = - (\lambda + \beta s) l^\top$. In view that $l \neq 0$ and $\lambda + \beta s \neq 0$ we then conclude that $l^\top \notin \ker (\cT^\top)$ and necessarily that $l^\top \in \mathrm{range} \, (\cT^\top)$. Let us now write
\[
l = \big( l_1, \ldots, l_d, l_{d+1} \big),
\]
where $l_k \in \C^{1 \times d}$ for each $1 \leq k \leq d+1$. Whence,
\begin{align}
l \cT = \big( l_1, \ldots, l_d, l_{d+1} \big) 
\begin{pmatrix} 
 & & & - \beta \Id \\
 & 0 & & \ii\xi_2 \Id \\
 & & & \vdots \\
 & & & \ii \xi_d \Id \\
\cG_1 &  \cdots & \cG_d  & 0
\end{pmatrix} &= \Big( l_{d+1} \cG_1, \, \ldots, \, l_{d+1} \cG_d, \, - \beta l_1 + \ii \sum_{j \neq 1} \xi_j l_j \Big) \nonumber\\
&=: \Big( l_{d+1} \cG_1, \, \ldots, \, l_{d+1} \cG_d, \, q^\top \Big). \label{imageT}
\end{align}
Use expression in \eqref{detg} and $l \cT = -(\lambda + \beta s) l$ to arrive at
\begin{equation}
\label{starr}
\begin{aligned}
\big( -q^\top + (\lambda + \beta s) l_{d+1} \big) \cG_k &= 0, & \quad 1 \leq k \leq d,\\
l_{d+1} \Big( \beta \cG_1 - \ii \sum_{j \neq 1} \xi_j \cG_j \Big) + (\lambda + \beta s) q^\top &= 0.
\end{aligned}
\end{equation}
The first $d$ equations in \eqref{starr} yield
\[
0 = \big( -q^\top + (\lambda + \beta s) l_{d+1} \big) \big( \cG_1, \ldots, \cG_d \big)
\begin{pmatrix}
\beta \Id \\ -\ii \xi_2 \Id \\ \vdots \\ -\ii \xi_d \Id
\end{pmatrix} \\
= \big( -q^\top + (\lambda + \beta s) l_{d+1} \big) \tAT(\ii \beta, \xim,U).
\]
The last equation in \eqref{starr} implies that
\[
l_{d+1} \Big( \beta \cG_1 - \ii \sum_{j \neq 1} \xi_j \cG_j \Big) = l_{d+1} \tAT(\ii \beta, \xim,U) = - (\lambda + \beta s) q^\top.
\]
Therefore we obtain
\[
q^\top \big( (\lambda + \beta s)^2 \Id + \tAT(\ii \beta, \xim,U) \big) = 0,
\]
that is, $q^\top$ is a left eigenvector of $\tAT(\ii \beta, \xim,U)$ with eigenvalue $- (\lambda + \beta s)^2$. Since $\tAT$ is invariant under simple transposition, $\tAT^\top = \tAT$, this is equivalent to \eqref{qTQ}. To find $l_{d+1}$ we notice that $\lambda + \beta s \neq 0$ and the first $d$ equations in \eqref{starr} imply that $l_{d+1} \cG_k = (\lambda + \beta s)^{-1} q^\top \cG_k$, for all $1 \leq k \leq d$. Substitute back into \eqref{imageT} to obtain
\[
l \cT = \Big( (\lambda + \beta s)^{-1} q^\top \cG_1, \ldots, (\lambda + \beta s)^{-1} q^\top \cG_d, q^\top \Big),
\]
and the general form of the left eigenvector is
\[
l = \big( q^\top \cG_1, \ldots, q^\top \cG_d, (\lambda + \beta s) q^\top \big),
\]
where $q$ is such that \eqref{qTQ} holds. This proves the lemma.
\end{proof}

Let us now focus on the 1-shock determined by $(U^\pm,v^\pm,s) \in \M_+^d \times \R^d \times \R$ satisfying \eqref{RHe1} and \eqref{Laxece1v2}. If we select $(U^+,v^+)$ as a base state then the shock is completely characterized by the parameter value $\alpha \neq 0$ described in Proposition \ref{proppadonde}. Let us define
\[
\cA^\pm(\lambda,\txi) := \cA(\lambda, \txi,U^\pm), \qquad (\lambda,\txi) \in \Gamma^+.
\]
From Hersh' lemma (see Appendix \ref{sechypall}, Remark \ref{remlopdetextr}), the stable eigenspace of $\cA^+(\lambda,\txi)$ has dimension equal to one for each $(\lambda,\txi) \in \Gamma^+$. Our goal is to compute the left (row) stable eigenvector $l_s^+(\lambda,\txi) \in \C^{1 \times n}$ of $\cA^+$ associated to the only stable eigenvalue $\beta$ with $\Re \beta < 0$. Thanks to Lemma \ref{lemforml}, this is equivalent to computing the column eigenvector $q^+$ of $\tAT^+(\ii \beta,\txi) := \tAT(\ii \beta,\txi,U^+)$.

In order to simplify the notation, let us write the cofactor matrix of $U^+$ as $V^+ := \Cof U^+ \in \M_+^d$, so that its $j$-th column is
\begin{equation}
\label{defVpj}
V_j^+ = (\Cof U^+)_j \in \R^{d \times 1},
\end{equation}
for each $1 \leq j \leq d$, and
\begin{equation}
\label{defkap2p}
(a_1^+)^2 = \kappa_2^+ = \mu + h''(J^+) |V_1^+|^2.
\end{equation}
Moreover, for any frequency vector $\xim = (\xi_2, \ldots, \xi_d)^\top \in \R^{d-1}$ we define the scalar (real) quantities,
\begin{equation}\label{omeeta}
\begin{aligned}
\eta^+(\xim) &:=  \sum_{j \neq 1} (V_1^+)^\top V_j^+ \xi_j,\\
\omega^+(\xim)& := \mu |\xim|^2 + h''(J^+)\left|V^+ \begin{pmatrix} 0 \\ \txi \end{pmatrix} \right|^2 = \mu |\xim|^2 + h''(J^+) \sum_{i,j \neq 1} (V_i^+)^{\top} V_j^+ \xi_i\xi_j,\\
\end{aligned}
\end{equation}
which depend only on the Fourier frequencies and on the elastic parameters of the material evaluated at the base state.

\begin{lemma}
\label{lemqcofU}
Let $\beta \in \C$ be the only stable eigenvalue with $\Re \beta < 0$ of the matrix field $\cA^+(\lambda, \xim)$, on $(\lambda,\xim) \in \Gamma^+$. Then the (column) eigenvector $q^+ \in \C^{d \times 1}$ of $\tAT^+(\ii \beta ,\xim)$ with associated eigenvalue $-(\lambda + \beta s)^2$, as described in Lemma \ref{lemforml}, can be uniquely selected (modulo scalings) as
\begin{equation}
\label{exqplus}
q^+ = q^+(\lambda,\txi):= (\Cof U^+) \begin{pmatrix} \ii \beta \\ \xi_2 \\ \vdots \\ \xi_d
\end{pmatrix}.
\end{equation}
Moreover, $\beta = \beta(\lambda,\txi)$ is a root of
\begin{equation}
\label{polybet}
\big(\kappa_2^+ - s^2\big) \beta^2 - 2\big(\lambda s + \ii h''(J^+) \eta^+(\txi) \big) \beta - \big(\lambda^2 + \omega^+(\txi)\big) = 0.
\end{equation}
\end{lemma}
\begin{proof}
In view that $s < 0$ and $\Re \lambda > 0$ then $\Re (- \lambda /s) > 0$ and consequently $\lambda + \beta s \neq 0$. Hence, from Lemma \ref{lemevalCA} we know that $\beta$ is a root of
\[
\det\big((\lambda+\beta s)^2 \Id + \tAT^+(\ii \beta ,\xim)\big)=0.
\]
Use expression \eqref{tQbxiU} and apply Sylvester's determinant formula (cf. \cite{AAM96}) to obtain
\[
\begin{aligned}
0 &= \det \big((\lambda+\beta s)^2 \Id + \tAT^+(\ii \beta ,\xim)\big)\\
&= \det \Big( \big[(\lambda + \beta s)^2 + \mu(- \beta^2 + |\xim|^2) \big] \Id + h''(J^+) q^+ \otimes q^+ \Big)\\
&= \Big( (\lambda + \beta s)^2 + \mu(-\beta^2 + |\xim|^2) \Big)^{d-1} \Big( (\lambda + \beta s)^2 + \mu(-\beta^2 + |\xim|^2) + h''(J^+) (q^+)^\top q^+ \Big),
\end{aligned}
\]
where $q^+$ is defined in \eqref{exqplus}. Now suppose that $(\lambda + \beta s)^2 + \mu(-\beta^2 + |\xim|^2) = 0$. Since $\Re \beta < 0$ for all frequencies in a connected set, $(\lambda, \xim) \in \Gamma^+$, by continuity it suffices to evaluate $\sgn (\Re \beta)$ at $\xim = 0$ and $\Re \lambda > 0$ with $|\lambda| = 1$. Substituting we obtain
\[
(\sqrt{\mu} \beta - \lambda - \beta s) (\sqrt{\mu} \beta + \lambda + \beta s) = 0,
\]
yielding the roots
\[
\beta = \frac{\lambda}{\sqrt{\mu} -s}, \quad \beta = - \, \frac{\lambda}{\sqrt{\mu} +s}.
\]
But both roots have $\Re \beta > 0$ because $s < - \sqrt{\mu} < 0$, a contradiction with $\Re \beta < 0$. Therefore, we conclude that $\beta$ must be a root of
\[
\varphi(\lambda,\xim,s,\beta) := (\lambda + \beta s)^2 + \mu (-\beta^2 + |\xim|^2) + h''(J^+) (q^+)^\top q^+ = 0.
\]
To double-check the form of $q^+$, from expression \eqref{tQbxiU} we immediately observe that
\[
\begin{aligned}
\big( (\lambda + \beta s)^2 \Id +  \tAT^+(\ii \beta ,\xim)\big) q^+ &= (\lambda + \beta s)^2 q^+ + \mu(-\beta^2 + |\xim|^2)q^+ + h''(J^+)(q^+ \otimes q^+)q^+ \\
&= \big[ (\lambda + \beta s)^2  + \mu(-\beta^2 + |\xim|^2) + h''(J^+) (q^+)^\top q^+ \big] q^+\\
&= \varphi(\lambda,\xim,s,\beta) q^+ \\
&= 0.
\end{aligned}
\]
Henceforth, we conclude that $\tAT^+(\ii \beta ,\xim)$ has an eigenvector of the form \eqref{exqplus} where $\beta$ is a solution to $\varphi(\lambda,\xim,s,\beta) = 0$. Since $\beta$ is the only stable eigenvalue of $\cA^+(\lambda, \xim)$ for any $(\lambda,\xim) \in \Gamma^+$ then the eigenvector $q^+$ can be uniquely determined (modulo scalings) by expression \eqref{exqplus}. To simplify the characteristic polynomial, notice that
\begin{equation}
\label{normq}
\begin{aligned}
|q^+|^2 = (q^+)^\top q^+ &= \big( \ii\beta, \; \xi_2, \; \cdots \;, \xi_d \big) (\Cof U^+)^\top (\Cof U^+) \begin{pmatrix} \ii \beta \\ \xi_2 \\ \vdots \\ \xi_d\end{pmatrix}\\
&= - \beta^2 |V_1^+|^2 + 2 \ii \beta \sum_{j \neq 1} (V_1^+)^\top V_j^+ \xi_j + \sum_{i,j \neq 1} (V_i^+)^{\top} V_j^+ \xi_i\xi_j, 
\end{aligned}
\end{equation}
yielding
\[
\begin{aligned}
- \varphi(\lambda,\xim,s,\beta) &= (\mu + h''(J^+) |V_1^+|^2 - s^2) \beta^2 - 2\beta \Big(\lambda s + \ii h''(J^+) \sum_{j \neq 1} \xi_j (V_1^+)^\top V_j^+  \Big) + \\
& \; - \Big(\lambda^2 + \mu |\xim|^2 + h''(J^+) \sum_{i,j \neq 1} \xi_i\xi_j  (V_i^+)^{\top} V_j^+ \Big)\\
&= \big(\kappa_2^+ - s^2\big) \beta^2 - 2\big(\lambda s + \ii h''(J^+) \eta^+(\txi) \big) \beta - \big(\lambda^2 + \omega^+(\txi)\big)  = 0,
\end{aligned}
\]
as claimed.
\end{proof}

\begin{remark}
\label{remnotbetlamss}
Notice that, from natural considerations, $\lambda + \beta s \neq 0$ for the stable eigenvalue $\beta$ with $\Re \beta < 0$. Another way to interpret this fact is that the eigenvalue $\beta = - \lambda / s$ is incompatible with the curl-free conditions \eqref{curlfree} (see the discussion in \cite{FrP1}) and, therefore, it should be excluded from the normal modes analysis.
\end{remark}

\subsection{Calculation of the ``jump" vector}
\label{secjumpv}

In the present case of a shock propagating in the $\hat{\nu} = \eu$ direction, the calculation of the Lopatinski\u{\i} determinant (see expression \eqref{gLopdet}) involves the computation of the following ``jump" vector,
\begin{equation}
\label{defjv}
\cK = \cK(\lambda,\txi) := \lambda \llb u \rrb + \ii \sum_{j \neq 1} \xi_j \llb f^j(u) \rrb,
\end{equation}
which is a complex vector field in the frequency space, $(\lambda, \txi) \mapsto \cK(\lambda,\txi)$, $\cK \in C^\infty(\Gamma^+;\C^{n \times 1})$, associated to the Rankine-Hugoniot jump conditions \eqref{RHe1} across the shock. Use \eqref{RHe1} and \eqref{jumpstatevar} to obtain,
\[
\cK(\lambda,\txi) = \begin{pmatrix}
\lambda \llb U_1 \rrb \\ \ii s \xi_2 \llb U_1 \rrb \\ \vdots \\ \ii s \xi_d \llb U_1 \rrb \\ -\lambda s \llb U_1 \rrb - \ii \sum_{j \neq 1} \xi_j \llb \sigma(U)_j \rrb
\end{pmatrix} = \begin{pmatrix}
\lambda \Id & 0 \\ \ii s \xi_2 \Id & 0 \\ \vdots & \vdots \\ \ii s \xi_d \Id & 0 \\ 0 & \Id
\end{pmatrix}
\begin{pmatrix}
\llb U_1 \rrb \\  -\lambda s \llb U_1 \rrb - \ii \sum_{j \neq 1} \xi_j \llb \sigma(U)_j \rrb
\end{pmatrix}.
\]
From expression \eqref{genforml} for the general form of a left eigenvector, $l \in \C^{1 \times n}$, of $\cA$, we have 
\[
\begin{aligned}
l \begin{pmatrix}
\lambda \Id & 0 \\ \ii s \xi_2 \Id & 0 \\ \vdots & \vdots \\ \ii s \xi_d \Id & 0 \\ 0 & \Id
\end{pmatrix} &= q^\top \Big( \cG_1, \ldots, \cG_d, (\lambda + \beta s) \Id \Big) 
\left[
\begin{pmatrix}
-s \beta \Id & 0 \\ \ii s \xi_2 \Id & 0 \\ \vdots & \vdots \\ \ii s \xi_d \Id & 0 \\ 0 & 0
\end{pmatrix} + 
\begin{pmatrix}
(\lambda + \beta s) \Id & 0 \\ 0 & 0 \\ \vdots & \vdots \\ 0 & 0 \\ 0 & \Id
\end{pmatrix} \right] \\
&= \Big( q^\top \big( - s\beta \cG_1 + \ii s \sum_{j \neq 1} \xi_j \cG_j \big), \; 0 \Big) + \Big( (\lambda + \beta s) q^\top \cG_1, \; (\lambda + \beta s) q^\top  \Big) \\
&= \Big( -s q^\top \tAT(\ii \beta, \xim,U) + (\lambda + \beta s) q^\top \cG_1, \; (\lambda + \beta s) q^\top \Big)\\
&= (\lambda + \beta s) q^\top \big( s(\lambda + \beta s) \Id + \cG_1, \, \Id),
\end{aligned}
\]
inasmuch as \eqref{qTQ} holds and $\tAT$ is invariant under simple trasposition. Therefore,
\[
\begin{aligned}
l\cK &= (\lambda + \beta s) q^\top \Big( s(\lambda + \beta s) \llb U_1 \rrb  + \cG_1 \llb U_1 \rrb - \lambda s \llb U_1 \rrb - \ii \sum_{j \neq 1} \xi_j \llb \sigma(U)_j \rrb\Big) \\
&= (\lambda + \beta s) q^\top \Big( (\beta s^2 \Id + \cG_1) \llb U_1 \rrb  - \ii \sum_{j \neq 1} \xi_j \llb \sigma(U)_j \rrb\Big).
\end{aligned}
\]
Hence, we have proved the following result, which will be useful later on.
\begin{proposition}
\label{propexprlQ}
If $\beta \in \C$ is an eigenvalue of $\cA(\lambda, \xim,U)$ with associated eigenvector $l$, then
\begin{equation}
\label{exprlQ}
l\cK = (\lambda + \beta s) q^\top \Big( (\beta s^2 \Id + \cG_1) \llb U_1 \rrb  - \ii \sum_{j \neq 1} \xi_j \llb \sigma(U)_j \rrb\Big),
\end{equation}
where $\cK$ is the ``jump" vector in \eqref{defjv}, $\cG_1$ is defined in \eqref{defGk} and $q \in \C^{d \times 1}$ is such that \eqref{qTQ} holds.
\end{proposition}

Let us now compute the elements involved in the definition of the jump vector field $\cK$. For simplicity, we introduce the notations
\[
{B^i_j}^+ := B^i_j(U^+) \in \R^{d \times d}, \quad \cG_k^+ := \cG_k(U^+) \in \C^{d \times d}, \qquad 1 \leq i,j,k \leq d.
\]
For later use we also compute (using formulae \eqref{exprBijs}, \eqref{defkap2p} and \eqref{defVpj}),
\begin{equation}
\label{sidB1V1}
\begin{aligned}
({B^1_1}^+-s^2 \Id ) V_1^+ &= \Big[  \mu \Id + h''(J^+) V_1^+ (V_1^+)^\top - s^2 \Id \Big] V_1^+ \\
&= (\mu-s^2 ) V_1^+ + h''(J^+) |V_1^+|^2 V_1^+\\
&= (\kappa_2^+-s^2 ) V_1^+, 
\end{aligned}
\end{equation}
as well as,
\begin{equation}
\label{B1jV1}
\begin{aligned}
{B_1^j}^+ V_1^+ &= \left[ h''(J^+) (V_j^+ \otimes V_1^+) + \frac{h'(J^+)}{J^+} \Big(V_j^+ \otimes V_1^+ - V_1^+ \otimes V_j^+\Big)\right] V_1^+ \\
&= \left( h''(J^+) + \frac{h'(J^+)}{J^+}\right) |V_1^+|^2 V_j^+ - \frac{h'(J^+)}{J^+} \big((V_j^+)^\top V_1^+\big) V_1^+, \\
&= \left[ \kappa_2^+-\mu + \frac{h'(J^+)}{J^+} |V_1^+|^2 \right] V_j^+ -  \frac{h'(J^+)}{J^+}  (V_j^+\cdot V_1^+) V_1^+,  \qquad \text{for all } \; j \neq 1.
\end{aligned}
\end{equation}

Now, from Rankine-Hugoniot conditions \eqref{RHe1}, relation \eqref{jumpCofU1} and Proposition \ref{proppadonde}, it is clear that
\begin{equation}
\label{workrels}
\begin{aligned}
\llb U_1 \rrb &= \alpha V_1^+,\\
\llb U_j \rrb &= 0, \qquad j \neq 1,\\
V_1^+ &= (\Cof U^+)_1 = (\Cof U^-)_1.
\end{aligned}
\end{equation}

Let us first compute the jump of the Piola-Kirchhoff stress tensor across the shock. From \eqref{HadPiolKir} we have
\[
\llb \sigma(U)_j \rrb = \mu \llb U_j \rrb + \llb h'(J) (\Cof U)_j \rrb = \alpha (s^2 - \mu) V_j^+ + h'(J^-) \llb (\Cof U)_j \rrb, \qquad \text{for } \; j \neq 1,
\]
after having substituted relation \eqref{speedsq}. Now, notice that from \eqref{RHmod} there holds
\[
U^- = U^+ - \alpha (V_1^+ \otimes \eu) = U^+ - \alpha \Big( V_1^+, \; 0, \; \cdots, \; 0 \Big),
\]
that is, $U^+$ and $U^-$ differ by a matrix with all columns equal to zero except for the first one (that is why, for instance, $(\Cof U^+)_1 = (\Cof U^-)_1 = V_1^+$). We shall use this information to find a suitable expression for the jump in the cofactor matrix column $\llb V_j \rrb = \llb (\Cof U)_j \rrb$, $j \neq 1$. For any $1 \leq i,j \leq d$, with $j \neq 1$, and by elementary properties of the determinant, the $(i,j)$-entry of $\Cof U^-$ is given by
\[
\begin{aligned}
(\Cof U^-)_{ij} &= (-1)^{i + j} \det \left[ \big( U^+ - \alpha (V_1^+ \otimes \eu)\big)'_{(i,j)}\right]\\
&= (-1)^{i + j} \det \left[ \big( U_1^+ - \alpha V_1^+, \; U_2^+, \; \cdots \;, U_d^+ \big)'_{(i,j)}\right]\\
&= (-1)^{i + j} \det \left[ \big( U_1^+, \; U_2^+, \; \cdots \;, U_d^+ \big)'_{(i,j)}\right] - \alpha (-1)^{i + j} \det \left[ \big( V_1^+, \; U_2^+, \; \cdots \;, U_d^+ \big)'_{(i,j)}\right]\\
&= (\Cof U^+)_{ij} - \alpha M_{ij}^+,
\end{aligned}
\]
where $M^+ \in \R^{d \times d}$ is the real $d \times d$ matrix whose first column is zero, $M_1^+ := 0$, and whose $(i,j)$-entry for any $1 \leq i,j \leq d$, with $j \neq 1$, is defined as
\begin{equation}
\label{defMij}
M_{ij}^+ :=  (-1)^{i + j} \det \left[ \big( V_1^+, \; U_2^+, \; \cdots \;, U_d^+ \big)'_{(i,j)}\right] = \Big( \Cof \big( V_1^+, \; U_2^+, \; \cdots \;, U_d^+ \big)\Big)_{ij}, \qquad j \neq 1.
\end{equation}
Henceforth we obtain,
\[
\llb (\Cof U)_1\rrb = \llb V_1 \rrb = 0, \qquad \llb (\Cof U)_j \rrb = \llb V_j \rrb = \alpha M_j^+, \quad j \neq 1.
\]
Upon substitution, we obtain the expressions for the jump of the stress tensor across the shock,
\begin{equation}
\label{jumpsigj1}
\llb \sigma(U)_j \rrb = \alpha \Big((s^2 - \mu) V_j^+ + h'(J^-) M_j^+\Big), \qquad \text{for } \; j \neq 1,
\end{equation}
and,
\[
\llb \sigma(U)_1 \rrb = \alpha s^2 V_1^+.
\]

\begin{remark}
\label{remMij}
The first column of $M^+$ is zero because $(\Cof U^+)_1 = (\Cof U^-)_1$. Notice that $M^+$ is a smooth function of the entries of $U^+$, $M^+ \in C^\infty(\M_+^d;\R^{d \times d})$. For example, in two spatial dimensions ($d = 2$) and after a straightforward computation one verifies that $\Cof U^- = \Cof U^+ - \alpha M^+$ where
\begin{equation}
\label{exprMd2}
M^+ = \begin{pmatrix}
0 & U_{12}^+ \\ 0 & U_{22}^+
\end{pmatrix} = U_2^+ \otimes \hat{e}_2 \, \in \R^{2 \times 2}.
\end{equation}
Likewise, when $d = 3$ one finds that 
\begin{equation}
\label{exprMd3}
M^+ = \Big( 0,  \quad U_3^+ \times V_1^+, \quad - U_2^+ \times V_1^+ \Big) \; \in \R^{3 \times 3}.
\end{equation}
\end{remark}

\subsection{Further simplifications}

In order to simplify the lengthy calculations to come, let us introduce the following notations. First, we write the scalar products of the columns of the cofactor matrix $V^+$ as
\begin{equation}
\label{defthetaij}
\theta_{ij} := (V_i^+)^\top V^+_j \in \R,
\end{equation}
for each $1 \leq i,j \leq d$. In this fashion, it is clear that $\theta^+_{jj} = |V^+_j|^2 > 0$, $\theta^+_{ij} = \theta^+_{ji}$ for all $i,j$, and that $\theta^+_{ij}$ is the $(i,j)$-entry of the real symmetric matrix $(V^+)^\top V^+$. Moreover, we define
\begin{equation}
\label{defThetaij}
\Theta^+_{ij} := \det \begin{pmatrix} 
\theta^+_{11} & \theta^+_{1j} \\ \theta^+_{i1} & \theta^+_{ij}
\end{pmatrix}, \qquad 1 \leq i,j \leq d.
\end{equation}
From its definition and Cauchy-Schwarz inequality it is clear that the matrix $\Theta^+ \in \R^{d \times d}$ satisfies
\begin{equation}
\label{propThetaij}
\begin{cases}
\Theta^+_{11} = \Theta^+_{j1} = \Theta^+_{1j} = 0, & 1 \leq j \leq d,\\
\Theta^+_{jj} > 0, &  j \neq 1,\\
\Theta^+_{ij} = \Theta^+_{ji}, & 1 \leq i,j \leq d.
\end{cases} 
\end{equation}

Next, we prove a result which significantly reduces the calculation of the large determinants involved in the products $(V_i^+)^\top M_j$ appearing in the assembly of the Lopatinski\u{\i} determinant.
\begin{lemma}
\label{lemprodMj}
For all $1 \leq i,j \leq d$, $d \geq 2$, there holds
\begin{equation}
\label{prodViMj}
(V_i^+)^\top M_j^+ = \frac{\Theta^+_{ij}}{J^+}.
\end{equation}
(In particular, we recover $(V_i^+)^\top M^+_1 \equiv 0$, for all $i$.)
\end{lemma}
\begin{proof}
Let us first verify formula \eqref{prodViMj} in the case of two space dimensions, $d = 2$. If $j \neq 1$ then $j=2$ and from \eqref{exprMd2} we have
\[
M_2^+ = \begin{pmatrix}
U^+_{12}\\ U_{22}^+
\end{pmatrix},  \quad 
V_1^+ = \begin{pmatrix} 
U_{22}^+ \\ - U_{12}^+
\end{pmatrix}, \quad 
V_2^+ = \begin{pmatrix} 
-U_{21}^+ \\ U_{11}^+
\end{pmatrix}.
\]
Thus, clearly, $(V_1^+)^\top M_2^+ = 0$ and $(V_2^+)^\top M_2^+ = J^+ > 0$. But from \eqref{propThetaij} and $\Theta_{22}^+ = \theta_{11}^+ \theta_{22}^+ - (\theta^+_{12})^2 = (J^+)^2$, we conclude that \eqref{prodViMj} holds.

Let us now suppose that $d \geq 3$. First, observe that since $V^+ = \Cof U^+$ then $(U^+)^\top V^+ = J^+ \Id$ and, thus, $(U^+)^\top V_1^+ = J^+ \eu$. Now, take any $j \neq 1$ and any $1 \leq i \leq d$. From the definition of $M^+$ (see \eqref{defMij}) and the basic properties, $(\Cof A^\top) = (\Cof A)^\top$ and $(\Cof A)^\top \Cof B = \Cof (A^\top B)$ for any $A, B \in \R^{d \times d}$,  we compute
\[
\begin{aligned}
(V_i^+)^\top M_j^+ &= \sum_{k=1}^d \big[ (\Cof U^+)^\top \big]_{ik} \left[ \Cof \big( V_1^+, \; U_2^+, \; \cdots, \, U^+_d \big)\right]_{kj} \\
&= \left[ (\Cof U^+)^\top \, \Cof \big( V_1^+, \; U_2^+. \; \cdots , \, U^+_d \big) \right]_{ij}\\
&= \left[ \Cof \big( (U^+)^\top V_1^+, \; (U^+)^\top U_2^+, \cdots, \, (U^+)^\top U^+_d \big) \right]_{ij}\\
&= (-1)^{i+j} \det \left( \Big( J^+ \eu, \; (U^+)^\top U^+_2, \cdots, \, (U^+)^\top U^+_d \Big)'_{(i,j)}\right)\\
&=: (-1)^{i+j} \det E'_{(i,j)}.
\end{aligned}
\]
To compute, for $j \neq 1$, this last determinant we expand along the first column to obtain
\[
\det E'_{(i,j)} = \det \left( \big( J^+ \eu, \; (U^+)^\top U^+_2, \cdots , \, (U^+)^\top U^+_d \big)'_{(i,j)}\right) = J^+ \det \left[ \big( (U^+)^\top U^+ \big)'_{(1i,1j)} \right],
\]
where for any matrix $A \in \R^{d \times d}$, with $d \geq 3$, $A'_{(1i,1j)}$ denotes the $(d-2) \times (d-2)$ submatrix formed by eliminating rows 1 and $i$, and columns 1 and $j$ from the original matrix $A$. Likewise, for any matrix $A$, $A_{(1i,1j)} \in \R^{2 \times 2}$ denotes the submatrix
\[
A_{(1i,1j)} = \begin{pmatrix} A_{11} & A_{1j} \\ A_{i1} & A_{ij} \end{pmatrix},
\]
for all $1 \leq i,j \leq d$. The computation of the $(d-2) \times (d-2)$ determinant of $A'_{(1i,ij)}$ is considerably reduced by the use of  Jacobi's formula (see Theorem 2.5.2 in Prasolov \cite{Praso94}, or Gradshteyn and Ryzhik \cite{GrRy7ed}, p. 1076):
\[
(-1)^{i+j} \det A \det A'_{(1i,1j)} = \det \big[ (\Cof A)_{(1i,1j)} \big].
\]
A direct application of last equation to the Cauchy-Green tensor $A = (U^+)^\top U^+$ yields,
\[
\begin{aligned}
(V_i^+)^\top M_j^+ &= (-1)^{i+j} \det E'_{(i,j)}\\
&= (-1)^{i+j} J^+ \det \left[ \big( (U^+)^\top U^+ \big)'_{(1i,1j)} \right]\\
&= (-1)^{i+j} J^+ (-1)^{-i-j} (\det (U^+)^\top U^+ )^{-1} \det \big[ \big(\Cof ((U^+)^\top U^+) \big)_{(1i,1j)} \big]\\
&= \frac{1}{J^+} \det \begin{pmatrix} \theta_{11}^+ & \theta_{1j}^+ \\ \theta^+_{1i} & \theta^+_{ij} \end{pmatrix}\\
&= \frac{\Theta^+_{ij}}{J^+},
\end{aligned}
\]
for the case $j \neq 1$ and $d \geq 3$. Moreover, notice that formula \eqref{prodViMj} is also valid for $j = 1$ because of \eqref{propThetaij} and $M^+_1 = 0$. The lemma is proved.
\end{proof}

\subsection{Summary}

To sum up, and for the convenience of the reader, we apply our simplified notation and gather in one place all the ingredients we have computed so far and which will be used to assemble the Lopatinski\u{\i} determinant in the next section. Indeed, use the short-cuts \eqref{omeeta}, \eqref{defVpj}, \eqref{defkap2p}, \eqref{defMij}, \eqref{defthetaij} and \eqref{defThetaij} to recast formulae \eqref{exqplus}, \eqref{defGk} with $k =1$, the first equation in \eqref{workrels}, \eqref{jumpsigj1}, \eqref{B1jV1}, \eqref{sidB1V1}, the first equation in \eqref{omeeta}, the second in \eqref{omeeta} and \eqref{speedsq} as,
\begin{equation}
\label{exqplusnew}
q^+(\lambda,\txi)^\top = \big( \ii \beta, \; \xi_2, \; \cdots, \; \xi_d \big) (V^+)^\top = \ii \beta (V_1^+)^\top + \sum_{i \neq 1} \xi_i (V_i^+)^\top \; \in \C^{1 \times d},
\end{equation}
\begin{equation}
\label{defG1new}
\cG_1^+ = \cG_1(\beta, \xim, U^+) = - \beta {B_1^1}^+ + \ii \sum_{j \neq 1} \xi_j {B_1^j}^+ \; \in \C^{d \times d},
\end{equation}
\begin{equation}
\label{saltoU1}
\llb U_1 \rrb = \alpha V_1^+ \; \in \R^{d \times 1},
\end{equation}
\begin{equation}
\label{jumpsigj1new}
\llb \sigma(U)_j \rrb = \alpha \Big((s^2 - \mu) V_j^+ + h'(J^-) M_j^+\Big) \; \in \R^{d \times 1},  \;\;\; j \neq 1,
\end{equation}
\begin{equation}
\label{B1jV1new}
{B_1^j}^+ V_1^+ = (\kappa_2^+-\mu) V_j^+ + \frac{h'(J^+)}{J^+} \big( \theta^+_{11} V_j^+ - \theta^+_{1j} V_1^+ \big), \in \R^{d \times 1}, \;\;\; j \neq 1,
\end{equation}
\begin{equation}
\label{sidB1V1new}
({B^1_1}^+-s^2 \Id ) V_1^+ = (\kappa_2^+-s^2 ) V_1^+ \; \in \R^{d \times 1}, 
\end{equation}
\begin{equation}
\label{etanew}
\eta^+(\xim) =  \sum_{j \neq 1} \xi_j \theta_{1j}^+ ,
\end{equation}
\begin{equation}
\label{omeganew}
\omega^+(\xim) = \mu |\xim|^2 + h''(J^+) \sum_{i,j \neq 1}  \xi_i\xi_j \theta_{ij}^+,
\end{equation}
and,
\begin{equation}
\label{speedsqnew}
\frac{1}{\alpha} \llb h'(J) \rrb = s^2 - \mu > 0,
\end{equation}
respectively. Finally, use formulae \eqref{sidB1V1}, \eqref{jumpsigj1new}, \eqref{prodViMj} and \eqref{B1jV1new} to further obtain:
\begin{equation}
\label{sidBiV1new}
(V_i^+)^\top \big({B_1^1}^+ - s^2 \Id \big) V_1^+ = (\kappa_2^+ - s^2) \theta_{i1}^+, \qquad 1 \leq i \leq d,
\end{equation}
\begin{align}
(V_i^+)^\top \big({B_1^j}^+ V_1^+ - \frac{1}{\alpha} \llb \sigma(U)_j \rrb \big) &=  (V_i^+)^\top \Big[ \Big( (\kappa_2^+ - \mu) + \frac{h'(J^+)}{J^+} \theta_{11}^+ \Big) V_j^+ - \frac{h'(J^+}{J^+} \theta_{j1}^+ V_1^+ \nonumber\\ 
& \qquad  - \big( (s^2 - \mu) V_j^+ + h'(J^-) M_j^+ \big) \Big] \nonumber \\
&= (\kappa_2^+ - s^2) \theta_{ij}^+ + \frac{h'(J^+)}{J^+} \big( \theta_{11}^+ \theta_{ij}^+ - \theta_{j1}^+ \theta_{i1} \big) - \frac{h'(J^-)}{J^+} \Theta_{ij}^+ \nonumber \\
&= (\kappa_2^+ - s^2) \theta_{ij}^+ + \alpha (s^2 - \mu) \frac{\Theta_{ij}^+}{J^+}, \label{secondstep2}
\end{align}
for all $1 \leq i,j \leq d$, $j \neq 1$. In particular, since $\Theta_{1j}^+ = 0$ we have, from last formula with $i=1$,
\begin{equation}
\label{secondstep2bis}
(V_1^+)^\top \big({B_1^j}^+ V_1^+ - \frac{1}{\alpha} \llb \sigma(U)_j \rrb \big) = (\kappa_2^+ - s^2) \theta_{1j}^+, \qquad j \neq 1.
\end{equation}

\section{Stability results}
\label{seclopdet}

\subsection{The Lopatinski\u{\i} determinant}
\label{secLopdet}

In this section, we calculate the Lopatinski\u{\i} determinant (or stability function) associated to a Lax 1-shock for compressible Hadamard materials. The main idea is to assemble different (yet equivalent) expressions, so that we can draw stability conclusions from them. In the present case of an extreme 1-shock, the stable subspace of $\mathcal{A}^+(\lambda, \xim)$ has dimension equal to one for all $(\lambda,\xim) \in \Gamma^+$ (see Remark \ref{remlopdetextr} in Appendix \ref{sechypall}). Therefore, the Lopatinski\u{\i} determinant reduces to the expression \eqref{Lopdetextr},
\[
\overline{\Delta}(\lambda, \xim) = l_+^s(\lambda,\xim) \cK(\lambda,\xim),
\]
where $l_+^s(\lambda,\xim)$ is the left stable (row) eigenvector of $\mathcal{A}^+(\lambda, \xim)$ associated to the only stable eigenvalue $\beta$ with $\Re \beta < 0$ and $\cK(\lambda,\xim)$ is the jump vector \eqref{defjv}. From Proposition \ref{propexprlQ} we obtain
\[
\overline{\Delta}(\lambda, \xim) = (\lambda + \beta s) \widehat{\Delta}(\lambda,\xim),
\]
where 
\begin{equation}
\label{trueLopdet}
\widehat{\Delta}(\lambda,\xim) := q^+(\lambda,\xim)^\top \Big( (\beta s^2 \Id + \mathcal{G}_1) \llb U_1 \rrb  - \ii \sum_{j \neq 1} \xi_j \llb \sigma(U)_j \rrb\Big), \qquad (\lambda, \xim) \in \Gamma^+,
\end{equation}
and $q^+$ is given by \eqref{exqplusnew}. In view that $\lambda + \beta s \neq 0$ for all $(\lambda,\xim) \in \Gamma^+$, the scalar complex field \eqref{trueLopdet} \emph{encodes all the information regarding the stability of the shock front} and, thus, we shall focus on determining the zeroes of $\widehat{\Delta}$ on $\Gamma$ (including, by continuity, the boundary $\partial \Gamma \subset \{\Re \lambda = 0\}$). We remind the reader that the frequency $\lambda = - \beta s$ is incompatible with the physical curl-free conditions \eqref{curlfree} and, therefore, we rule out the limit $\lim \beta = - \lim \lambda/s = - \Im \lambda /s$ as $\Re \lambda \to 0^+$ when considering zeroes of $\overline{\Delta}$ along the imaginary axis; see Remark \ref{remnotbetlamss}. 

Substitute \eqref{exqplusnew}, \eqref{sidB1V1new}, \eqref{etanew}, \eqref{defG1new}, \eqref{secondstep2}, \eqref{secondstep2bis} and \eqref{saltoU1} into \eqref{trueLopdet} to obtain
\begin{align}
\frac{\ii}{\alpha} \widehat{\Delta}(\lambda,\xim) &= \Big[ \ii \beta (V_1^+)^\top + \sum_{i \neq 1} \xi_i (V_i^+)^\top \Big] \Big[ - \ii \beta \big( {B_1^1}^+ - s^2 \Id \big)V_1^+ - \sum_{j \neq 1} \xi_j \big( {B_1^j}^+ V_1^+ -\frac{1}{\alpha} \llb \sigma(U)_j \rrb \big) \Big] \nonumber \\
&= \beta^2 (\kappa_2^+ - s^2) \theta_{11}^+ - 2 \ii \beta (\kappa_2^+ - s^2) \sum_{j \neq 1} \xi_j \theta_{1j}^+ - \sum_{i,j \neq 1} \xi_i \xi_j \Big( (\kappa_2^+ - s^2) \theta_{ij}^+ + \alpha(s^2 - \mu) \frac{\Theta_{ij}^+}{J^+}\Big). \label{thirdstep}
\end{align}

This is the main expression for the Lopatinski\u{\i} determinant we shall be working with. At this point we introduce the following material parameter which, in fact, determines the stability of the shock (see Theorems \ref{teoweakstab} and \ref{stabcriteria} below).
\begin{defin}[material stability parameter]
\label{defirho}
For any 1-shock in the $\eu$-direction for a compressible Hadamard material, we define
\begin{equation}
\label{defrho}
\rho(\alpha) :=  (s^2 - \mu) \left( \frac{1}{\theta_{11}^+}-\frac{\alpha}{J^+}   \right)-h''(J^+)  \; \in \, \R.
\end{equation}
\end{defin}

It is to be noticed that $\rho(\alpha)$ depends only on the shock parameters (the base state and of the shock strength) and on the elastic moduli of the material. It is, of course, independent of the Fourier frequencies $\xim \in \R^{d-1}$. We also define for notational convenience,
\begin{equation}
\label{defNp}
N^+(\xim)^2 := \left| V^+ \! \begin{pmatrix} 0 \\ \xim \end{pmatrix} \right|^2 = \sum_{i,j \neq 1} \xi_i \xi_j \theta_{ij}^+,
\end{equation}
for all $\xim \in \R^{d-1}$
\begin{lemma}[Lopatinski\u{\i} determinant, version 1]
\label{lemLopdetv1}
The Lopatinski\u{\i} determinant \eqref{thirdstep} can be recast as
\begin{equation}
\label{fourthstep}
\frac{\ii}{\alpha} \widehat{\Delta}(\lambda,\xim) = (\kappa_2^+ - s^2) \theta_{11}^+ \Big( \beta - \ii \frac{\eta^+(\xim)}{\theta_{11}^+}\Big)^2 + \rho(\alpha) \big( \theta_{11}^+ N^+(\xim)^2 - \eta^+(\xim)^2\big).
\end{equation}
\end{lemma}
\begin{proof}
Follows by direct computation and by noticing that the last term inside the sum in \eqref{thirdstep} is
\[
(\kappa_2^+ - s^2) \theta_{ij}^+ + \alpha(s^2 - \mu) \frac{\Theta_{ij}^+}{J^+} =- \rho(\alpha) \theta_{11}^+ \theta_{ij}^+ + \Big( \rho(\alpha) + h''(J^+) - \frac{s^2 - \mu}{\theta_{11}^+} \Big) \theta_{1j}^+ \theta_{i1}^+,
\]
after having substituted \eqref{defrho} and \eqref{defkap2p}. Using \eqref{defNp} and \eqref{etanew}, the Lopatinski\u{\i} determinant \eqref{thirdstep} can be written as
\[
\begin{aligned}
\frac{\ii}{\alpha} \widehat{\Delta}(\lambda,\xim) &= \beta^2 (\kappa_2^+ - s^2) \theta_{11}^+ - 2 \ii \beta (\kappa_2^+ - s^2) \sum_{j \neq 1} \xi_j \theta_{1j}^+ + \rho(\alpha) \theta_{11}^+ N^+(\xim)^2 + \\
& \quad - \Big( \rho(\alpha) + h''(J^+) - \frac{s^2 - \mu}{\theta_{11}^+} \Big) \eta^+(\xim)^2 \\
&= (\kappa_2^+ - s^2) \theta_{11}^+ \Big( \beta - \ii \frac{\eta^+(\xim)}{\theta_{11}^+}\Big)^2 + \rho(\alpha) \big( \theta_{11}^+ N^+(\xim)^2 - \eta^+(\xim)^2\big),
\end{aligned}
\]
as claimed. Notice that this formula is simply the completion of the square in the variable $\beta$.
\end{proof}

\subsection{Sufficient condition for weak stability}
\label{secsufcondws}

Based on the first version of the Lopatinski\u{\i} determinant, formula \eqref{fourthstep} above, we are ready to establish our first stability theorem. First, we need to prove the following elementary
\begin{lemma}
\label{lemauxsign}
For all $\xim \in \R^{d-1}$, there holds
\begin{equation}
\label{signNeta}
P^+(\xim) := \theta_{11}^+ N^+(\xim)^2 - \eta^+(\xim)^2 \geq 0.
\end{equation}
Moreover, equality holds only when $\xim = 0$.
\end{lemma}
\begin{proof}
Since $N^+(0)^2 = \eta^+(0)^2 = 0$ for $\xim = 0$, it suffices to prove that $\theta_{11}^+ N^+(\xim)^2 - \eta^+(\xim)^2 > 0$ for all $\xim \in \R^{d-1}$, $\xim \neq 0$. First, we write the above expression as a quadratic form
\[
\begin{aligned}
P^+(\xim) &= |V_1^+|^2 \left| V^+ \!\begin{pmatrix} 0 \\ \xim \end{pmatrix} \right|^2 - \left( V^+ \! \begin{pmatrix} 0 \\ \xim \end{pmatrix}\right)^\top \Big( V_1^+ \otimes V_1^+ \Big) V^+ \! \begin{pmatrix} 0 \\ \xim \end{pmatrix} \\
&= \left( V^+ \! \begin{pmatrix} 0 \\ \xim \end{pmatrix}\right)^\top \Big( |V_1^+|^2 \Id - V_1^+ \otimes V_1^+ \Big) V^+ \! \begin{pmatrix} 0 \\ \xim \end{pmatrix}.
\end{aligned}
\]
Notice that the eigenvalues of the matrix $|V_1^+|^2 \Id - V_1^+ \otimes V_1^+$ are $\widetilde{\nu} = 0$ and $\widetilde{\nu} = |V_1^+|^2 = \theta_{11}^+ > 0$. Indeed, for $\widetilde{\nu} \neq \theta_{11}^+$, use Sylvester's determinant formula to obtain
\[
\det \Big( (\theta_{11}^+ - \widetilde{\nu})\Id - V_1^+(V_1^+)^\top \Big) = - \widetilde{\nu} \big( \theta_{11}^+ - \widetilde{\nu}\big)^{d-1}.
\]
This implies that $\widetilde{\nu} = 0$ is a simple eigenvalue associated to the eigenvector $V_1^+$, inasmuch as $(|V_1^+|^2 \Id - V_1^+ \otimes V_1^+) V_1^+ = 0$. Hence, we conclude that $|V_1^+|^2 \Id - V_1^+ \otimes V_1^+$ is positive semi-definite and $P^+(\xim) \geq 0$ for all $\xim \in \R^{d-1}$. Now suppose that $P^+(\xim) = 0$ for some $\xim \neq 0$. Since $\widetilde{\nu} = 0$ is a simple eigenvalue, this implies that $V^+ \bigl( \begin{smallmatrix} 0 \\ \xim \end{smallmatrix} \bigr) = k V_1^+$ for some scalar $k$ or, in other words, that the columns of $V^+$ are linearly dependent, a contradiction. This proves the lemma.
\end{proof}

\begin{theo}[sufficient condition for weak stability]
\label{teoweakstab}
For a compressible hyperelastic Hadamard material satisfying assumptions \eqref{H1} -- \eqref{H3}, consider a classical Lax 1-shock with intensity $\alpha \neq 0$. Suppose that
\begin{equation}
\label{condstab1}
\rho(\alpha) \geq 0.
\end{equation}
Then the shock is, at least, \emph{weakly stable} (more precisely, there are no roots of the Lopatinski\u{\i} determinant in $\Gamma^+$).
\end{theo}
\begin{proof}
According to Proposition \ref{proppadonde}, given the base state $(U^+, v^+) \in \M_+^d \times \R^d$, the shock is completely characterized by the parameter $\alpha \in (-\infty,0) \cup (0, \alpha^+_{\mathrm{max}})$. Suppose that for a fixed value of $\alpha \neq 0$ (independently of its sign) condition \eqref{condstab1} holds\footnote{notice that under \eqref{H3} necessarily $\alpha < 0$, in view of Proposition \ref{proppadonde}; the result holds, however, independently of the sign of $\alpha$.}. Let us normalize the Lopatinski\u{\i} determinant as,
\[
\check{\Delta} (\lambda, \xim) := \frac{\ii}{\alpha} \frac{\widehat{\Delta}(\lambda,\xim)}{(\kappa_2^+ - s^2)\theta_{11}^+}, \qquad (\lambda, \xim) \in \Gamma^+.
\]
From Lemma \ref{lemauxsign} and Lax conditions, we have $P^+(\xim) \geq 0$, $\theta_{11}^+ > 0$ and $\kappa_2^+ - s^2 > 0$. Thus, using \eqref{condstab1} we may define
\[
\delta:= \sqrt{\frac{\rho(\alpha) P^+(\xim)}{(\kappa_2^+ - s^2) \theta_{11}^+}} \geq 0,
\]
for all $(\lambda,\xim) \in \Gamma^+$, and write
\[
\check{\Delta} (\lambda, \xim) = \Big( \beta - \ii \frac{\eta^+(\xim)}{\theta_{11}^+}\Big)^2+ \delta^2 = \Big(  \beta - \ii \frac{\eta^+(\xim)}{\theta_{11}^+} - \ii \delta \Big) \Big(  \beta - \ii \frac{\eta^+(\xim)}{\theta_{11}^+} + \ii \delta \Big).
\]
In view that the real part of each factor in last formula is negative ($\Re \beta < 0$ in $\Gamma^+$), we conclude that $\check{\Delta}$ never vanishes in $\Gamma^+$. 
\end{proof}

\subsection{Locating zeroes along the imaginary axis}

In order to locate zeroes of the Lopatinski\u{\i} determinant along the imaginary axis, we need to find a new expression for it. For that purpose, we examine in more detail the unique stable eigenvalue $\beta = \beta(\lambda,\xim)$ with $\Re \beta < 0$ of $\mathcal{A}^+$, $(\lambda,\xim) \in \Gamma^+$, and define an appropriate mapping in the spatio-temporal frequency space.

Recall that $\beta \in \C$ is a root of the second order characteristic polynomial \eqref{polybet} (see Lemma \ref{lemqcofU}), whose discriminant is,
\[
4 \Xi(\lambda,\xim) :=  4(\lambda s + \ii h''(J^+) \eta^+(\xim))^2 + 4 (\kappa_2^+ - s^2) (\lambda^2 + \omega^+(\xim)), \qquad (\lambda,\xim) \in \Gamma^+.
\]
This is a second order polynomial in $\lambda$. Completing the square in $\lambda$ yields,
\[
\Xi(\lambda,\xim) =  \left[ \left( \sqrt{\kappa_2^+} \lambda + \ii \, \frac{sh''(J^+) \eta^+(\xim)}{\sqrt{\kappa_2^+}}\right)^2 + (\kappa_2^+ - s^2)\zeta^+(\xim) \right],
\]
where
\begin{equation}
\label{defzetap}
\zeta^+(\xim) := \omega^+(\xim) - \frac{h''(J^+)^2}{\kappa_2^+} \eta^+(\xim)^2 \in \R.
\end{equation}
Therefore, the two $\beta$-roots of \eqref{polybet} are given by
\[
\beta = (\kappa_2^+ - s^2)^{-1} \Big( \lambda s + \ii h''(J^+) \eta^+(\xim) \pm \Xi(\lambda,\xim)^{1/2}\Big).
\]
To select the branch of the square root, we recall that the stable eigenvalue $\beta = \beta(\lambda,\xim)$ is continuous and $\Re \beta < 0$ in $\Gamma^+$. If $\xim = 0$ then $\omega^+(0) = \eta^+(0) = \zeta^+(0) = 0$ and $\Xi(\lambda,0)^{1/2} = (\kappa_2^+ \lambda^2)^{1/2}$ is continuous in $\Re \lambda > 0$. Hence, we may select $\Xi(\lambda,0)^{1/2} = \sqrt{\kappa_2^+} \lambda$ as the principal branch. Since $\kappa_2^+ > s^2$ (Lax conditions) and $\Re \lambda > 0$, the stable root at $(\lambda,0)$ is
\[
\beta(\lambda,0) = - \frac{\lambda}{\sqrt{\kappa_2^+} + s}.
\]
Consequently, the branch we select for the stable root is
\begin{equation}
\label{betain}
\beta(\lambda,\xim) = (\kappa_2^+ - s^2)^{-1} \Big( \lambda s + \ii h''(J^+) \eta^+(\xim) - \Xi(\lambda,\xim)^{1/2}\Big).
\end{equation}
We introduce here the following mapping in the frequency space,
\begin{equation}
\label{defgamma}
\left\{
\begin{aligned}
\Psi(\lambda,\xim) &:= (\gamma(\lambda,\xim), \xim),\\
\Psi : \Gamma^+ &\mapsto \C \times \R^{d-1}, \\
\gamma(\lambda,\xim) &:= \frac{1}{\sqrt{\kappa_2^+ - s^2}} \left( \lambda \sqrt{\kappa_2^+} + \ii \frac{s h''(J^+) \eta^+(\xim)}{\sqrt{\kappa_2^+}}\right).
\end{aligned}
\right.
\end{equation}
The goal is to express the Lopatinski\u{\i} determinant \eqref{fourthstep} as well as the stable eigenvalue \eqref{betain} in terms of the new frequency variables $(\gamma,\xim)$. 

\begin{lemma}
\label{lemgoodmap}
The frequency mapping $\Psi : (\lambda,\xim) \mapsto (\gamma,\xim)$ is injective and maps $\Gamma^+$ onto the set
\begin{equation}
\widetilde{\Gamma}^+ := \left\{ (\gamma,\xim) \in \C \times \R^{d-1} \, : \, \Re \gamma > 0, \, \left| \sqrt{(\kappa_2^+)^{-1}(\kappa_2^+ - s^2)} \, \gamma - \ii (\kappa_2^+)^{-1} s h''(J^+) \eta^+(\xim) \right|^2 + |\xim|^2 = 1 \right\}.
\end{equation}
\end{lemma}
\begin{proof}
Seen as a mapping from $(\Re \lambda, \Im \lambda, \xim^{\, \top}) \in \R^{d+1}$ to $\R^{d+1}$, $\Psi$ is of class $C^\infty$ and its Jacobian has the following structure
\[
D_{(\lambda,\xim)} \Psi = \begin{pmatrix} 
\frac{\sqrt{\kappa_2^+}}{\sqrt{\kappa_2^+ - s^2}}\mathbb{I}_2 & * \\ 0 & \mathbb{I}_{d-1} 
\end{pmatrix},
\]
which is clearly invertible. Notice that
\[
\Re \gamma = \sqrt{\kappa_2^+ (\kappa_2^+ - s^2)^{-1}} \, \Re \lambda,
\]
and, therefore, $\Re \lambda > 0$ if and only if $\Re \gamma > 0$. Hence, we conclude that $\Psi(\Gamma^+) = \widetilde{\Gamma}^+$.
\end{proof}

Let us substitute \eqref{defgamma} into \eqref{betain}. After straightforward algebra, the result is the stable eigenvalue $\beta$ as a function of the new frequency variables:
\[
\beta(\gamma, \xim) = \frac{s}{\sqrt{\kappa_2^+ (\kappa_2^+ - s^2)}} \left[ \gamma + \ii \frac{\sqrt{\kappa_2^+-s^2}}{s \sqrt{\kappa_2^+}} h''(J^+) \eta^+(\xim)  - \frac{\sqrt{\kappa_2^+}}{s} \Big( \gamma^2 + \zeta^+(\xim)\Big)^{1/2}\right].
\] 
Use $\kappa_2^+ = \mu + h''(J^+) \theta_{11}^+$ to obtain
\begin{equation}
\label{nbetaminus}
\beta - \ii \frac{\eta^+(\xim)}{\theta_{11}^+} = \frac{s}{\sqrt{\kappa_2^+ (\kappa_2^+ - s^2)}} \left[ \gamma - \frac{\sqrt{\kappa_2^+}}{s} \Big( \gamma^2 + \zeta^+(\xim)\Big)^{1/2} - \ii \frac{\mu \eta^+(\xim)}{\theta_{11}^+} \frac{\sqrt{\kappa_2^+ - s^2}}{s\sqrt{\kappa_2^+}} \right].
\end{equation}
Substitution of last expression into the first version of the Lopatinski\u{\i} determinant, equation \eqref{fourthstep}, yields
\begin{align*}
\frac{\ii}{\alpha} \widehat{\widehat{\Delta}}(\gamma,\xim) := \frac{\ii}{\alpha} {\widehat{\Delta}} (\lambda(\gamma,\xim),\xim) &= \frac{s^2 \theta_{11}^+}{\kappa_2^+} \left[ \gamma - \frac{\sqrt{\kappa_2^+}}{s} \Big( \gamma^2 + \zeta^+(\xim)\Big)^{1/2} - \ii \frac{\mu \eta^+(\xim)}{\theta_{11}^+} \frac{\sqrt{\kappa_2^+ - s^2}}{s\sqrt{\kappa_2^+}}\right]^2  + \\ & \;\;+ \rho(\alpha) \big( \theta_{11}^+ N^+(\xim)^2 - \eta^+(\xim)^2\big)\\
&= \frac{s^2 \theta_{11}^+}{\kappa_2^+} \left[ \left( \gamma - \frac{\sqrt{\kappa_2^+}}{s} \Big( \gamma^2 + \zeta^+(\xim)\Big)^{1/2} + \ii \tau^+ \eta^+(\xim)\right)^2 + \frac{\kappa_2^+}{s^2 \theta_{11}^+}\rho(\alpha) P^+(\xim) \right],
\end{align*}
where
\begin{equation}
\label{deftaup}
\tau^+ := - \, \frac{\mu \, \sqrt{\kappa_2^+ - s^2}}{s \sqrt{\kappa_2^+} \theta_{11}^+} > 0.
\end{equation}
Notice that $\tau^+$ is a positive constant (recall that $s < 0$) depending only on the parameters of the shock. $P^+(\xim)$ is defined in \eqref{signNeta}. Therefore, we have proved the following lemma.
\begin{lemma}[Lopatinski\u{\i} determinant, version 2]
\label{lemLopdetv2}
The Lopatinski\u{\i} determinant \eqref{fourthstep} can be rewritten and normalized as
\begin{equation}
\label{Lopadetv2}
\widetilde{\Delta}(\gamma,\xim) := \frac{\kappa_2^+}{s^2 \theta_{11}^+} \, \frac{\ii}{\alpha} \widehat{\widehat{\Delta}} (\gamma,\xim) = \left( \gamma - \frac{\sqrt{\kappa_2^+}}{s} \Big( \gamma^2 + \zeta^+(\xim)\Big)^{1/2} + \ii \tau^+ \eta^+(\xim)\right)^2 + \frac{\rho(\alpha)\kappa_2^+}{s^2\theta_{11}^+} P^+(\xim),
\end{equation}
for $(\gamma,\xim) \in \widetilde{\Gamma}^+$. It encodes the same stability information in the sense that $\widetilde{\Delta} = 0$ in $\widetilde{\Gamma}^+$ if and only if $\widehat{\Delta} = 0$ in $\Gamma^+$. Moreover, by continuity and thanks to the properties of the mapping ($\lambda,\xim) \mapsto (\gamma,\xim)$ (see Lemma \ref{lemgoodmap}), $\widetilde{\Delta}$ has a zero with $\gamma \in \ii \R$ if and only if $\widehat{\Delta}$ has a zero with $\lambda \in \ii \R$.
\end{lemma}
As a first consequence of the expression for the Lopatinski\u{\i} determinant \eqref{Lopadetv2} we have the following
\begin{corollary}[one-dimensional stability]
\label{cor1dstab}
For every Hadamard energy function of the form \eqref{Hadamardmat} satisfying \textcolor{red}{\eqref{H1} -- \eqref{H3}}, all classical shock fronts are uniformly stable with respect to one-dimensional perturbations. In particular, the Lopatinski\u{\i} determinant \eqref{trueLopdet} behaves for $\xim = 0$ as
\[
\frac{\ii}{\alpha} \widehat{\Delta}(\lambda,0) = \theta_{11}^+ \frac{\sqrt{\kappa_2^+} -s}{\sqrt{\kappa_2^+}+ s} \lambda^2 \neq 0,
\]
for any $(\lambda,0)  \in \Gamma^+$.
\end{corollary}
\begin{proof}
Set $\xim = 0$ and $(\gamma,0) \in \widetilde{\Gamma}^+$. Then $\Re \gamma > 0$ and $|\gamma|^2 = \kappa_2^+ /(\kappa_2^+ - s^2)$. This implies that
\[
\gamma = \frac{\sqrt{\kappa_2^+}}{\sqrt{\kappa_2^+ - s^2}} e^{\\i \upsilon}, \qquad \upsilon \in [0,2 \pi).
\]
Since $\zeta^+(0) = \eta^+(0) = P^+(0) = 0$ we have, upon substitution into \eqref{Lopadetv2},
\[
\widetilde{\Delta}(\gamma,0) = \frac{\kappa_2^+}{s^2} \frac{\sqrt{\kappa_2^+} -s}{\sqrt{\kappa_2^+}+ s} e^{\ii 2 \upsilon}.
\]
In view of the frequency transformation \eqref{defgamma} and the relation $(\ii/\alpha) \widehat{\Delta}(\lambda,0) = s^2 \theta_{11}^+ \widetilde{\Delta}(\gamma(\lambda,0),0)/\kappa_2^+$ we obtain the result for all $(\lambda,0) = (e^{\ii \upsilon},0) \in \Gamma^+$.
\end{proof}

\begin{remark}\label{zetapos}
Note that the behavior of the Lopatinski\u{\i} determinant in \eqref{Lopadetv2} strongly depends on the sign of $\zeta^+(\xim)$ because it determines the branches of the square root. Hence, it is worth observing that $\zeta^+(\xim) > 0$ for all $\xim \neq 0$ and $\zeta^+(0) = 0$ if and only if $\xim = 0$. Indeed, use \eqref{omeeta}, \eqref{defNp} and \eqref{defzetap} to recast $\zeta^+(\xim)$ as
\[
\begin{aligned}
 \zeta^+(\xim)&=\omega^+(\xim)-\frac{h''(J^+) ^2}{\kappa_2^+} \eta^+(\xim)^2\\
&=\mu|\xim|^2+h''(J^+)N^+(\xim)^2-\frac{h''(J^+)^2}{\kappa_2^+} \eta^+(\xim)^2\\
&=\mu|\xim|^2+\frac{h''(J^+)}{\theta_{11}^+}P^+(\xim)+\Big(1-\frac{\theta_{11}^+ h''(J^+)}{\kappa_2^+}\Big)\frac{h''(J^+)}{\theta_{11}^+} \eta^+(\xim)^2\\
&=\mu|\xim|^2+\frac{h''(J^+)}{\theta_{11}^+}P^+(\xim)+\frac{\mu h''(J^+)}{\kappa_2^+\theta_{11}^+} \eta^+(\xim)^2.
\end{aligned}
\]
Since $P^+(\xim) \geq 0$, Lemma \ref{lemauxsign}, $\mu>0$ and $h''>0$ (condition \eqref{H2}) we arrive at the conclusion.
\end{remark}

Notably, $\zeta^+(\xim)$ remains positive if we substract a suitable frequency expression depending on $\tau^+$. This is a useful property to locate the zeroes of the Lopatinski\u{\i} determinant along the imaginary axis.
\begin{lemma}
\label{lemtau1}
For every $\xim \in \R^{d-1}$ there holds,
\[
\zeta^+(\xim)-\big(\tau^+\eta^+(\xim)\big)^2=\mu|\xim|^2+\frac{h''(J^+)}{\theta_{11}^+}P^+(\xim)+\frac{\mu(s^2-\mu)}{s^2(\theta_{11}^+)^2}\eta^+(\xim)^2 \geq 0.
\]
Moreover, equality holds if and only if $\xim = 0$.
\end{lemma}
\begin{proof}
Follows from Remark \ref{zetapos}, the definition of $\tau^+$ and straightforward algebra:
\[
\begin{aligned}
 \zeta^+(\xim)-\big(\tau^+\eta^+(\xim)\big)^2&=\mu|\xim|^2+\frac{h''(J^+)}{\theta_{11}^+}P^+(\xim)+\Big(\frac{\mu h''(J^+)}{\kappa_2^+\theta_{11}^+}-\frac{\mu^2(\kappa_2^+-s^2)}{s^2\kappa_2^+(\theta_{11}^+)^2}\Big)\eta^+(\xim)^2\\
&=\mu|\xim|^2+\frac{h''(J^+)}{\theta_{11}^+}P^+(\xim)+\frac{\mu(s^2-\mu)}{s^2(\theta_{11}^+)^2}\eta^+(\xim)^2.
\end{aligned}
\]
The conclusion now follows.
\end{proof}

We proceed with the investigation of the possible zeroes of the Lopatinski\u{\i} determinant along the imaginary axis, which are associated to the existence of surface waves. Let us consider a zero of $\widetilde{\Delta}$ of the form $(\ii t,\xim)$, with $t\in\R$. Let us define $Y(t,\xim):=\widetilde{\Delta}(\ii t,\xim)$ for $t\in \R$, and  now we find conditions under which $Y$ has real zeros for a fixed frequency $\xim\in\R^{d-1}\setminus\{0\}$. By Lemma \ref{lemtau1}, $\zeta^+(\xim)$ is positive for all $\xim\in\R^{d-1}\setminus\{0\}$, so let us first consider
\[
t \in \left( - \sqrt{\zeta^+(\xim)}, \sqrt{\zeta^+(\xim)} \right).
\]
In this case we can write
\[
Y(t,\xim)=\Big(-\frac{  {\sqrt{\kappa_2^+}}  }{s}\sqrt{\zeta^+(\xim)-t^2} + \ii \big(t+\tau^+\eta^+\big) \Big)^2+\frac{\rho(\alpha)\kappa_2^+}{s^2\theta_{11}^+}P^+(\xim).
\]
Supposing that $Y(t, \xim) = 0$, its imaginary part vanishes, yielding
\[
-2\frac{ {\sqrt{\kappa_2^+}} }{s} \Big(t+\tau^+\eta^+(\xim)\Big) \sqrt{\zeta^+(\xim)-t^2}  = 0.
\]
By hypothesis, $\sqrt{\zeta^+(\xim)-t^2}\neq0$. Hence the imaginary part vanishes only if $t = - \tau^+ \eta^+(\xim)$. Notice that $t = - \tau^+ \eta^+(\xim) \in (- \sqrt{\zeta^+},\sqrt{\zeta^+})$ in view of Lemma \ref{lemtau1}. However,
\[
\begin{aligned}
Y(-\tau^+\eta^+(\xim),\xim) &= \left(-\frac{  {\sqrt{\kappa_2^+}} }{s}\sqrt{\zeta^+(\xim)-(\tau^+\eta^+(\xim))^2}\right)^2+\frac{\rho(\alpha)\kappa_2^+}{s^2\theta_{11}^+}P^+(\xim)\\
&=\frac{\kappa_2^+}{s^2} \Big(\zeta^+(\xim)-(\tau^+\eta^+(\xim))^2+\frac{\rho(\alpha)}{\theta_{11}^+}P^+(\xim)\Big)\\
&=\frac{\kappa_2^+}{s^2} \left(\mu|\xim|^2+\big(h''(J^+)+\rho(\alpha)\big)\frac{P^+(\xim)}{\theta_{11}^+}+\frac{\mu(s^2-\mu)}{s^2(\theta_{11}^+)^2}\eta^+(\xim)^2 \right)\\
&=\frac{\kappa_2^+}{s^2} \left(\mu|\xim|^2+(s^2-\mu)\Big(\frac{1}{\theta_{11}^+}-\frac{\alpha}{J^+}\Big)\frac{P^+(\xim)}{\theta_{11}^+}+\frac{\mu(s^2-\mu)}{s^2(\theta_{11}^+)^2}\eta^+(\xim)^2\right),
\end{aligned}
\]
which is strictly positive for all  $\xim\in\R^{d-1}\setminus \{0\}$ because $\mu>0$, $s^2>\mu$, $P^+(\xim)>0$ and 
\[
\frac{1}{\theta_{11}^+}-\frac{\alpha}{J^+}=\frac{J^-}{\theta_{11}^+J^+}>0.
\] 
Therefore, we conclude that $Y$ does not vanish on the interval $(-\sqrt{\zeta^+},\sqrt{\zeta^+})$. Let us now consider
\[
|t| \geq \sqrt{\zeta^+(\xim)}.
\]
In this case we have
\[
\sqrt{-t^2+\zeta^+(\xim)} = \ii \; \sgn(t) \sqrt{t^2-\zeta^+(\xim)},
\]
and hence
\[
Y(t,\xim)=-\Big(t-\frac{ {\sqrt{\kappa_2^+}} }{s}\sgn(t)\sqrt{t^2-\zeta^+(\xim)}+\tau^+\eta^+(\xim)\Big)^2+\frac{\rho(\alpha)\kappa_2^+}{\theta_{11}^+s^2}P^+(\xim).
\]
Observe that $\eta^+(-\xim)=-\eta^+(\xim)$, $P(-\xim)=P(\xim)$ and $\zeta^+(-\xim)=\zeta^+(\xim)$. Thus, the following property holds, $Y(-t,\xim)=Y(t,-\xim)$, and we can assume without loss of generality that $t \geq \sqrt{\zeta^+} > 0$ for $\xim \neq 0$. In this case, $Y$ takes the form
\[
Y(t,\xim)=-\Big(t-\frac{ {\sqrt{\kappa_2^+}} }{s}\sqrt{t^2-\zeta^+(\xim)}+\tau^+\eta^+(\xim)\Big)^2+\frac{\rho(\alpha)\kappa_2^+}{\theta_{11}^+s^2}P^+(\xim), \qquad \xim\in\R^{d-1}\setminus\{0\}.
\]
A straightforward computation then yields
\[
\frac{\partial Y(t,\xim)}{\partial t}=-2\Big(t-\frac{ {\sqrt{\kappa_2^+}} }{s}\sqrt{t^2-\zeta^+(\xim)}+\tau^+\eta^+(\xim)\Big)\Big(1-\frac{ {\sqrt{\kappa_2^+}} }{s}\frac{t}{\sqrt{t^2-\zeta^+}}\Big).
\]
We readily observe that since $s < 0$ then the last factor is positive. In view of Lemma \ref{lemtau1} it follows that $|\tau^+\eta^+|< \sqrt{\zeta^+} \leq t$ and, hence, the first factor is also positive. This shows that $Y$ is strictly decreasing as a function of $t > \sqrt{\zeta^+}$ for all $\xim\in\R^{d-1}\setminus\{0\}$. Moreover, $Y$ behaves as 
\[
Y\approx-t^2\left(1-\frac{ {\sqrt{\kappa_2^+}} }{s}\right)^2 < 0,
\]
as $t \to +\infty$ and for fixed $\xim \neq 0$.

Consequently, $Y$ has a unique zero of the form $(t,\xim)$ with $t \geq \sqrt{\zeta^+}$ if and only if there exists at least one frequency $\xim_0 \neq 0$ such that 
\[
Y \left( \mathsmaller{\sqrt{\zeta^+(\xim_0)}},\xim_0 \right) \geq 0,
\]
yielding the condition
\[
\Big(\mathsmaller{\sqrt{\zeta^+(\xim_0)}}+\tau^+\eta^+(\xim_0)\Big)^2-\frac{\rho(\alpha)\kappa_2^+}{s^2\theta_{11}^+}P^+(\xim_0) \leq 0.
\]
Otherwise there are no purely imaginary zeroes. Note that if  $\rho(\alpha) \leq 0$ then the left hand side of last expression is strictly positive for all $\xim_0 \neq 0$ in view of Lemma \ref{lemauxsign}. On account of the homogenity of $\widetilde{\Delta}$ in $\xim$ we may assume $|\xim|=1$. We summarize the observations of this section into the following
\begin{lemma}[existence of purely imaginary zeroes]
\label{imzeros} If $\rho(\alpha) \leq 0$ then $\widetilde{\Delta}$ has no zeroes of the form $(\ii t,\xim)$ with $t\in\R$. Conversely, if $\rho(\alpha)>0$ then $\widetilde{\Delta}$ has at least one zero of the form $(\ii t,\xim)$ if and only if there exist at least one frequency $\xim_0 \neq 0$ such that
\begin{equation}
\label{crit1}
\Big(\mathsmaller{\sqrt{\zeta^+(\tilde{\xi}_0)}}+\tau^+\eta^+(\xim_0)\Big)^2-\frac{\rho(\alpha)\kappa_2^+}{s^2\theta_{11}^+}P^+(\xim_0)\leq 0.
\end{equation}
\end{lemma}

\begin{remark}
\label{remimzeros}
From Theorem \ref{teoweakstab} we know that if $\rho(\alpha)\geq0$ then the shock is either weakly or strongly stable. Lemma \ref{imzeros} allows us to distinguish between the two cases. For instance, if the shock $(U^\pm, v^\pm,s)$ is such that $\rho(\alpha)=0$ then relation \eqref{crit1} is never satisfied for any frequency $\xim\in\R^{d-1}\setminus\{0\}$ and the shock is strongly stable (recall that $\zeta^+ > 0$ for $\xim \neq 0$ and, in view of Lemma \ref{lemtau1}, $\sqrt{\zeta^+} \geq |\tau^+ \eta^+| > 0$). When $\rho(\alpha)>0$ the stability is determined by the expression \eqref{crit1}, which can be considered as the condition for the transition from strong to weak stability.
\end{remark}

\subsection{The case $\rho(\alpha)<0$}

From Lemma \ref{imzeros} and Remark \ref{remimzeros}, we already know that $\widetilde{\Delta}$ has not purely imaginary roots when $\rho(\alpha)<0$. At the same time, Theorem \ref{teoweakstab} guarantees that if $\rho(\alpha) \geq 0$ then the shock is at least weakly stable and the transition from weak to strong stability is determined by condition \eqref{crit1}. Therefore, the only remaining task is to determine whether there exist zeroes of the form $(\gamma,\xim)$ with $\Re \gamma>0$ when $\rho(\alpha) < 0$. Following the proof of Theorem \ref{teoweakstab}, we exploit the fact that $\Re \beta <0$ in order to reduce the analysis to only one factor (a third version of the Lopatinski\u{\i} determinant) instead of the whole function $\widetilde{\Delta}$. Let us recall that
\[
\widetilde{\Delta}(\gamma,\xim) = \frac{\kappa_2^+}{s^2 \theta_{11}^+} \, \frac{\ii}{\alpha} \widehat{\widehat{\Delta}} (\gamma,\xim)= \frac{\kappa_2^+}{s^2 \theta_{11}^+} \, \frac{\ii}{\alpha} \widehat{\Delta} (\lambda(\gamma,\xim),\xim),
\]
so we come back to the expression of $\frac{\ii}{\alpha}\widehat{\Delta}$ defined in Lemma \ref{lemLopdetv1}, which can be written as:
\[
\begin{aligned}
\frac{\ii}{\alpha}\widehat{\Delta}(\lambda(\gamma,\xim),\xim)&=(\kappa_2^+-s^2)\theta_{11}^+\left(\Big(\beta(\lambda(\gamma,\xim),\xim) - \frac{\ii \eta^+(\xim)}{\theta_{11}^+}\Big)^2+\frac{\rho(\alpha)P^+(\xim)}{(\kappa_2^+-s^2)\theta_{11}^+}\right)\\
&=(\kappa_2^+-s^2)\theta_{11}^+\left(\Big(\beta - \frac{\ii \eta^+(\xim)}{\theta_{11}^+}\Big)^2-\delta^2\right)\\
&=(\kappa_2^+-s^2)\theta_{11}^+\left(\beta-\delta-\frac{\ii \eta^+(\xim)}{\theta_{11}^+}\right)\left(\beta+\delta-\frac{\ii \eta^+(\xim)}{\theta_{11}^+}\right),
\end{aligned}
\]
where now, with a slight abuse of notation,
\[
\delta = \sqrt{\frac{-\rho(\alpha)P^+(\xim)}{\theta_{11}^+(\kappa_2^+-s^2)}} > 0,
\]
in view that $\rho(\alpha)<0$. Except for the constant $(\kappa_2^+-s^2)\theta_{11}^+$, note that the real part of first factor in the expression of $\frac{\ii}{\alpha}\widehat{\Delta}$ is negative ($\Re \beta < 0$ in $\Gamma^+$ and, because of Lemma \ref{lemgoodmap}, $\Re \beta < 0$ in $\widetilde{\Gamma}^+$ as well). Hence, this factor never vanishes in $\widetilde{\Gamma}^+$. Necessarily, all possible zeroes $\gamma$ in $\widetilde{\Gamma}^+$ come from the last factor. Profiting from \eqref{nbetaminus}, we recast the latter as follows.
\begin{defin}[Lopatinski\u{\i} determinant, version 3]
In the case when $\rho(\alpha) < 0$, we define
\begin{equation}\label{subdet1}
\begin{split}
\Delta_1(\gamma,\xim):&=\frac{\sqrt{\kappa_2^+(\kappa_2^+-s^2)}}{s}\left(\beta + \delta - \frac{\ii \eta^+(\xim)}{\theta_{11}^+}\right)\\
&=\gamma-\frac{ {\sqrt{\kappa_2^+}} }{s}\sqrt{\gamma^2+\zeta^+(\xim)}+\ii \tau^+\eta^+(\xim) +\frac{ {\sqrt{\kappa_2^+}} }{s}\sqrt{\frac{-\rho(\alpha)P^+(\xim)}{\theta_{11}^+}}
\end{split}
\end{equation}
for each $(\gamma, \xim) \in \widetilde{\Gamma}^+$.
\end{defin}

From the preceding discussion, it suffices to study the zeroes of $\Delta_1$ on $\widetilde{\Gamma}^+$ to draw stability conclusions about the shock in the case $\rho(\alpha) < 0$. To that end, we apply the argument principle to count the number of roots of $\Delta_1$ in the right complex $\gamma$-half-plane. We proceed as in \cite{JL}, introducing polar coordinates $(R,\phi)$ and defining, for any fixed $\xim\neq0$, the function
\[
H(R,\phi)=H(w):=\Delta_1(w,\xim),\quad w=Re^{\ii \phi}.
\]

Consider $H(w)$ as $w$ varies counterclock-wise along the closed contour $\mathcal{C}$ consisting of a semicircle together with a vertical segment joining the ends; see Figure \ref{figcontour}. From Lemma \ref{imzeros} it is known that if $\rho(\alpha)<0$ then there are no roots of $\widetilde{\Delta}$ of the form $(\ii t,\xim)$ (and, consequently, of $\Delta_1$ as well). Therefore, the function $H$ does not have purely imaginary roots for any fixed $\xim \neq 0$ and we only have to avoid the branch cuts of the square root when we map this portion of the imaginary axis. We are interested in the behavior of the image of $\mathcal{C}$ under $H$ as $R\to\infty$. From expression \eqref{subdet1}, notice that the image of the circular portion for large $R$ behaves like
\[
H(R,\phi) \approx\left(1-\frac{s}{ {\sqrt{\kappa_2^+}} }\right) R e^{i\phi},
\]
as $R \to \infty$. Hence, the image is almost a circular portion too. Now we examine the mapping of the portion of $\mathcal{C}$ on the imaginary axis, that is, when $\phi=\pm \pi/2$. Substitution into \eqref{subdet1} yields
\[
H(R, \pm \tfrac{\pi}{2}) = \pm \ii R +\ii \tau^+\eta^+(\xim) + \frac{\sqrt{\kappa_2^+}}{s} \sqrt{\frac{-\rho(\alpha)P^+(\xim)}{\theta_{11}^+}} - \frac{\sqrt{\kappa_2^+}}{s} \cdot \begin{cases} \pm \ii \sqrt{R^2 - \zeta^+(\xim)}, & R^2 > \zeta^+(\xim), \\ \sqrt{\zeta^+(\xim)-R^2}, & R^2 \leq \zeta^+(\xim). \end{cases}
\]

\begin{figure}[t]
\begin{center}
\subfigure[Contour $\mathcal{C}$]{\label{figcontourC}\includegraphics[scale=.6, clip=true]{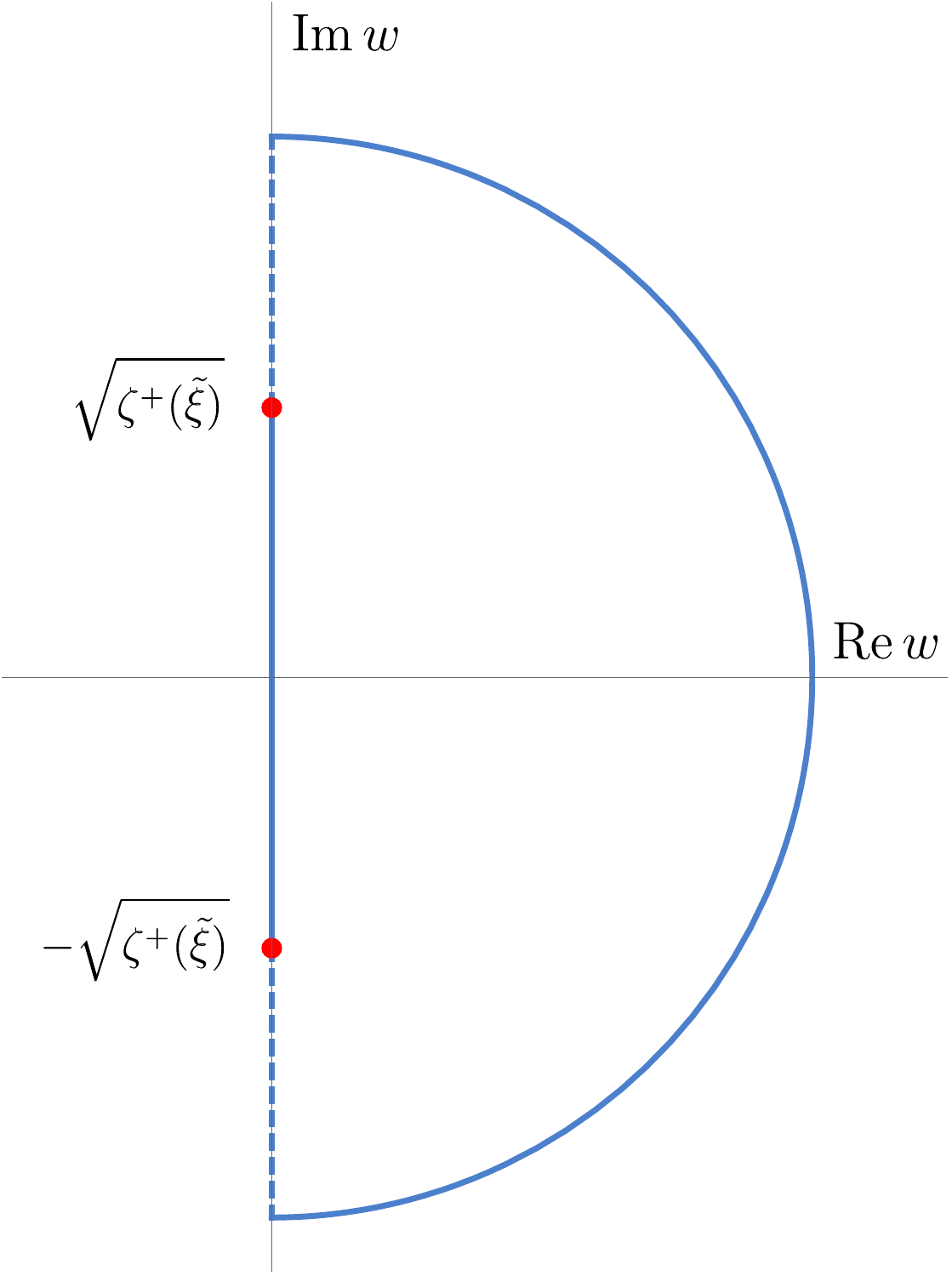}}
\subfigure[$H(\mathcal{C})$]{\label{figcontourHC}\includegraphics[scale=.6, clip=true]{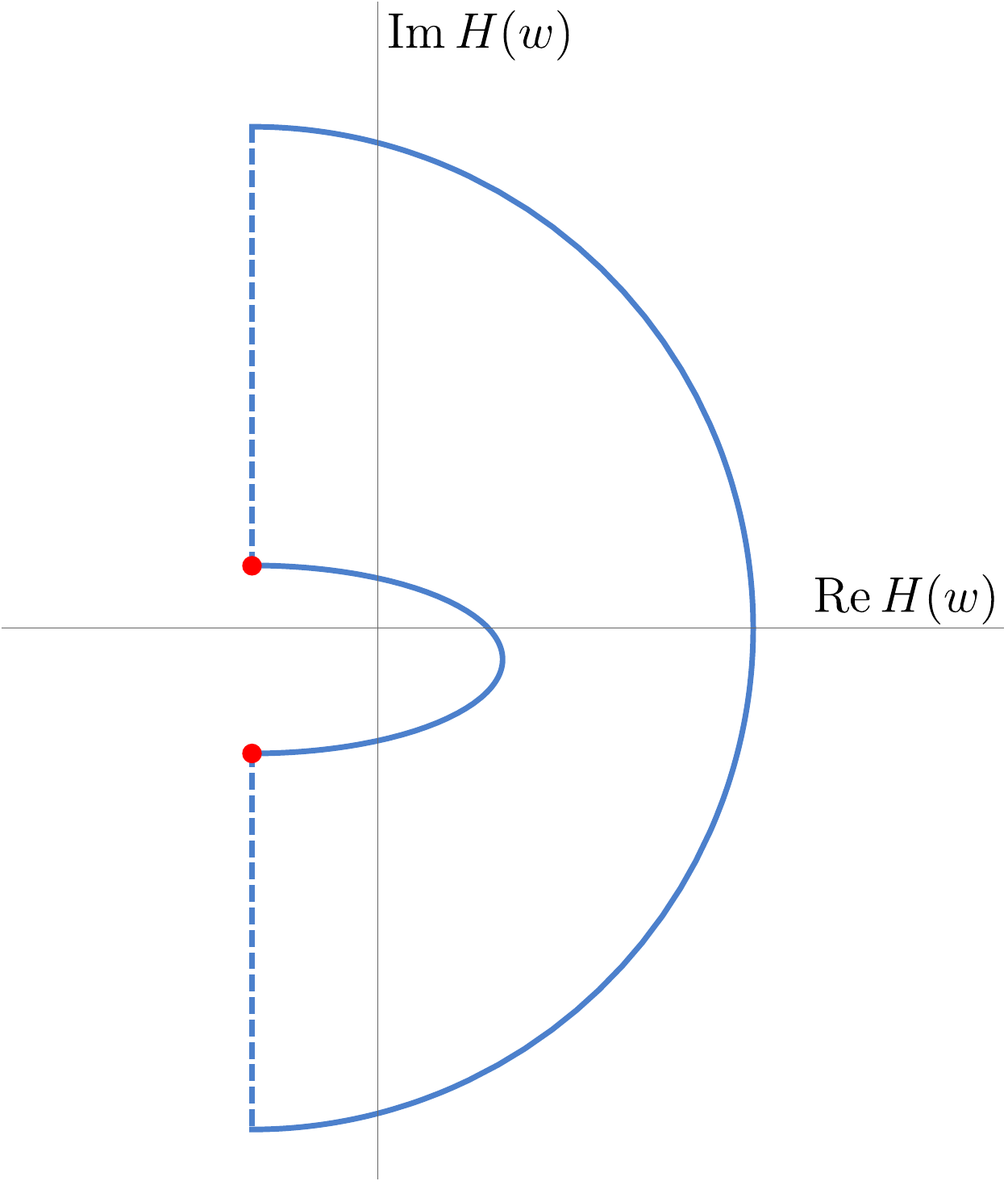}}
\end{center}
\caption{Illustration of the contour $\mathcal{C}$ in the $w$-complex plane (in blue; panel (a)) and of its image under the mapping $H$ (panel (b); color online).}\label{figcontour}
\end{figure}

Hence, $H$ maps the segment $(- \ii \mathsmaller{\sqrt{\zeta^+}}, \ii \mathsmaller{\sqrt{\zeta^+}})$ into the half right part of the following ellipse in the $XY$-plane,
\begin{equation}
\label{ellipse}
\left(-\frac{s}{\sqrt{\kappa_2^+}}X+\sqrt{\frac{-\rho(\alpha)P^+(\xim)}{\theta_{11}^+}}\right)^2+\Big(Y-\tau^+\eta^+(\xim)\Big)^2=\zeta^+(\xim),
\end{equation}
where $X=\Re H(w)$, $Y=\Im H(w)$. At the same time, $H$ maps the segment $(-\ii R,-\ii \mathsmaller{\sqrt{\zeta^+}})\cup (\ii \mathsmaller{\sqrt{\zeta^+}},\ii R)$ into the lines joining the upper an lower vertices of the ellipse with  points  $H(R,\frac{\pi}{2})$ and $H(R,-\frac{\pi}{2})$ respectively; see Figure \ref{figcontourHC}.

Note that the total change in the argument of $H$ on the contour $\mathcal{C}$ depends on whether or not the point $(X,Y) = (0,0)$ is inside the ellipse. Since $H$ has no purely imaginary zeros for all $\xim\neq0$, $(X,Y) = (0,0)$ does not lie on the ellipse in the $XY$-plane. It remains to check whether $(X,Y) = (0,0)$ is inside or outside the ellipse. For that purpose, we apply Lemma \ref{lemtau1} in order to write
\[
\big(\tau^+\eta^+(\xim)\big)^2=\zeta^+(\xim)-\Big(\mu|\xim|^2+\frac{h''(J^+)}{\theta_{11}^+}P^+(\xim)+\frac{\mu(s^2-\mu)}{s^2(\theta_{11}^+)^2}(\eta^+)^2\Big).
\]
Now if we substitute $X=0$, $Y=0$ into the right hand side of \eqref{ellipse} then we find that
\[
\begin{aligned}
\frac{-\rho(\alpha)P^+(\xim)}{\theta_{11}^+}+\big(\tau^+\eta^+(\xim)\big)^2 &= \zeta^+(\xim)-\Big(\mu|\xim|^2+(s^2-\mu)\Big(\frac{1}{\theta_{11}^+}-\frac{\alpha}{J^+}\Big)\frac{P^+(\xim)}{\theta_{11}^+}+\frac{\mu(s^2-\mu)}{s^2(\theta_{11}^+)^2}(\eta^+)^2\Big) \\ &< \zeta^+(\xim),
\end{aligned}
\]
for each $\xim \neq 0$. Hence, we conclude that the point $(X,Y) = (0,0)$ is inside the ellipse (or equivalently, it lies outside of the image of the contour under $H$, as illustrated in Figure  \ref{figcontourHC}). This implies that there is no change in the argument of $H(w)$ as $w$ varies counterclockwise along the closed contour $\mathcal{C}$ and that there are no roots with positive real part of $H$ for all $\xim\neq0$. The argument can be applied to any arbitrarily large radius $R > 0$. Therefore, as long as $\rho(\alpha)<0$, $\widetilde{\Delta}(\gamma,\xim)$ does not vanish for $\Re\gamma>0$. In view of Remark \ref{remimzeros}, we conclude that $\rho(\alpha)\leq 0$ \emph{is a sufficient condition for uniform (or strong) stability}. 

We summarize the last discussion and the precedent theorems into the following main result.

\begin{theo}[stability criteria]
\label{stabcriteria}
For a compressible hyperelastic Hadamard material satisfying assumptions \textcolor{red}{\eqref{H1} -- \eqref{H3}}, consider a classical (Lax) 1-shock with intensity $\alpha \neq 0$. 
\begin{itemize}
\item[(a)] If $\rho(\alpha)\leq0$ then the shock is \emph{uniformly stable}.
\item[(b)] In the case where $\rho(\alpha)>0$, the shock is \emph{uniformly stable} if and only if
\begin{equation}\label{critdef}
\Big(\mathsmaller{\sqrt{\zeta^+(\xim)}}+\tau^+\eta^+(\xim)\Big)^2-\frac{\rho(\alpha)\kappa_2^+}{s^2\theta_{11}^+}P^+(\xim)>0,\quad\text{for all}\:\:\xim \neq 0.
\end{equation}
Otherwise the shock is \emph{weakly stable}.
\end{itemize}
\end{theo}
\begin{remark}
\label{remtransition}
Being the left hand side of \eqref{critdef} of order $O(|\xim|^2)$, in most cases it constitutes a quadratic form in $\xim$ and there exists a real matrix $L^+\in \R^{d \times d}$ depending only on the shock and material parameters (that is, independent of the frequencies $\xim \in \R^{d-1}$) such that, in those cases, the transition from weak to strong stability condition can be recast as follows: when $\rho(\alpha) > 0$ the shock is uniformly stable if and only if the matrix $L^+$ restricted to the $d-1$ dimensional space, $\{ (0, \xim) \, : \, \xim \in \R^{d-1} \} \subset \R^{d}$, is positive definite, that is, if $(0, \, \xim)^\top L^+ \left( \begin{smallmatrix} 0 \\ \xim \end{smallmatrix} \right)> 0$ for all $\xim \neq 0$. In other words, one can state the transition condition \eqref{critdef} in terms of the shock and material parameters alone, as in the case of gas dynamics (cf. \cite{BS,M1}). However, the general form of the matrix $L^+$ is convoluted and, in practice, it is more convenient to verify \eqref{critdef} directly (see, for instance, the example in section  \S \ref{secCG} below).
\end{remark}

\subsection{Applications}
\label{secappli}

In order to illustrate the theoretical results, in this section we examine a couple of specific energy density functions describing compressible Hadamard materials and determine the conditions for shock stability.

\subsubsection{Two-dimensional Ciarlet-Geymonat model}
\label{secCG}
We begin by considering, in two space dimensions $d = 2$, the following volumetric energy density proposed by Ciarlet and Geymonat \cite{CiGe82} (see Appendix \ref{haddimate}, \S \ref{subsecexamples}, equation \eqref{hCG} below),
\begin{equation}
\label{hCGd2}
h(J) = - \mu - \mu \log J + \Big( \frac{\kappa - \mu}{2}\Big) (J-1)^2,
\end{equation}
where $\mu$ and $\kappa$ are the (constant) shear and bulk moduli, respectively, satisfying $\kappa > \mu > 0$. Energies of the form \eqref{hCG} model nearly incompressible materials (that is, they are proposed for small deformations) and they satisfy the free stress condition \eqref{relhmu} and the hydrostatic pressure condition \eqref{relkapmu} of Pence and Gou \cite{PnGo15}. In other words, these models are compressible extensions of neo-Hookean materials. This two-dimensional version of the Ciarlet-Geymonat energy, \eqref{hCGd2}, has been proposed by Trabelsi \cite{Trab03} to describe nonlinear thin plate materials modeling \emph{flexural shells}.

Given a base state $(U^+, v^+) \in \R^{2 \times 2} \times \R^2$, a Lax shock is completely determined by the parameter $\alpha \in \R$ (see Lemma \ref{lemsimpRH}). It can be shown (see section \S \ref{subsecexamples} below) that
\[
h''(J) = \frac{\mu}{J^2} + \kappa - \mu > 0, \qquad h'''(J) = - \frac{2\mu}{J^3} < 0,
\]
for all $J \in (0,\infty)$. Thus, this energy density satisfies \eqref{H1} -- \eqref{H3}.  In view of Proposition \ref{proppadonde}, in order to have a classical shock front we need $\alpha < 0$. Notice that $|\alpha|$ can be arbitrarily large, meaning that the shock can be of arbitrary amplitude. According to our notation
\[
V_1^+ = (\Cof U^+)_1 = \begin{pmatrix} U_{22}^+ \\ -U_{12}^+ \end{pmatrix} \in \R^2.
\]
A straightforward calculation (which we leave to the dedicated reader) yields
\[
\rho(\alpha) = - (\kappa - \mu) \frac{|V_1^+|^2}{J^+} \alpha >0.
\]
Therefore, from Theorem \ref{teoweakstab} we know that all classical shocks with intensity $\alpha < 0$ are, at least, weakly stable. In order to examine condition \eqref{crit1} and the emergence of surface waves, we set, for simplicity, $U^+ = \Ido$ (undeformed base state). Thus,
\[
\begin{aligned}
V_1^+ &= (\Cof U^+)_1 = \hat{e}_1 \in \R^2, \quad \theta_{11}^+ = |V_1^+|^2 = 1,\\
U^- &= U^+ - \alpha (V_1^+ \otimes \hat{e}_1) = \Ido - \alpha (\hat{e}_1 \otimes \hat{e}_1) = \begin{pmatrix} 1-\alpha & 0 \\ 0 & 1 \end{pmatrix},\\
J^+ &= 1, \quad J^- = 1-\alpha > 1.
\end{aligned}
\]
This yields $\rho(\alpha) = - (\kappa - \mu) \alpha$. Since the physical dimension is $d = 2$, the Fourier frequency is $\xim = \xi_2 \in \R$ and $(\lambda, \xi_2) \in \Gamma^+ = \{ \Re \lambda > 0, \, |\lambda|^2 + \xi_2^2 = 1\}$. After straightforward computations the reader may verify that
\[
\begin{aligned}
\kappa_2^+ &= \mu + \kappa,\\
s^2 &= \kappa + \frac{\mu}{1-\alpha}, \quad \text{with } \; s < 0,\\
\eta^+(\xi_2) &= 0, \quad P^+(\xi_2) = \xi_2^2,\\
\zeta^+(\xi_2) = \omega^+(\xi_2) &= (\mu + \kappa) \xi_2^2.
\end{aligned}
\]
Upon substitution into the left hand side of \eqref{crit1}, we obtain
\[
(\mu + \kappa) \xi_2^2 \left( 1 + \frac{\alpha (1-\alpha)(\kappa - \mu)}{\mu + (1-\alpha)\kappa}\right).
\]
Thus, the sign is determined by the function
\[
L(\alpha) = 1 + \frac{\alpha (1-\alpha)(\kappa - \mu)}{\mu + (1-\alpha)\kappa}, \qquad \alpha  <0.
\]
Clearly, $L(\alpha) > 0$ for $\alpha \approx 0^-$. Therefore, when $\xi_2 \neq 0$ condition \eqref{critdef} holds for $\alpha < 0$ and $|\alpha|$ small and the shock is uniformly stable. It is easily verified that $L(\alpha_*) = 0$ with $\alpha_* < 0$ only when
\begin{equation}
\label{defalpstCG}
\alpha_* = - \, \left( \frac{\mu + \sqrt{\mu^2 + 4(\kappa^2 - \mu^2)}}{2(\kappa - \mu)}\right) < 0.
\end{equation} 
Thanks to Theorem \ref{stabcriteria} we obtain the following
\begin{proposition}
\label{propCGt}
For the two-dimensional Ciarlet-Geymonat model \eqref{hCGd2}, classical shocks with base sate $U^+ = \Ido$ and intensity $\alpha < 0$ are uniformly stable if $\alpha \in (\alpha_*,0)$ and weakly stable if $\alpha \in (-\infty, \alpha_*]$, where the critical value $\alpha_*$ is given by \eqref{defalpstCG}.
\end{proposition}
 
To illustrate this behavior we compute the Lopatinski\u{\i} determinant, version 2 (see Lemma \ref{lemLopdetv2}) as a function of the transformed frequencies $(\gamma, \xi_2) \in \widetilde{\Gamma}^+$. Substituting the above parameters into \eqref{Lopadetv2} we obtain
\begin{equation}
\label{LopdetCGd2}
\widetilde{\Delta}(\gamma,\xi_2) = \left[ \gamma - (\mu+\kappa)^{1/2} \Big(\kappa + \frac{\mu}{1-\alpha}\Big)^{-1/2} \Big( \gamma^2 + (\mu+\kappa)\xi_2^2\Big)^{1/2}\right]^2 - \frac{\alpha(\kappa^2 - \mu^2)\xi_2^2}{\kappa + \frac{\mu}{1-\alpha}}.
\end{equation}

Set the shear and bulk moduli as $\kappa = 2 > \mu =1$. Hence the threshold $\alpha$-value for weak/uniform stability is $\alpha_* = -2.3028$. Since the condition for uniform to weak stability does not depend on $\xi_2$ we may assume that $|\xi_2| = 1$. Figure \ref{figCG} shows the 3D and contour plots of the Lopatinski\u{\i} determinant \eqref{LopdetCGd2} for the Ciarlet-Geymonat model \eqref{hCGd2} in dimension $d = 2$ as function of $\gamma \in \C$ with $\xi_2^2 = 1$, for the shock parameter value $\alpha = -0.3 \in (\alpha_*,0)$ in Figure \ref{figCGa}, and for $\alpha_* = -8 \in (-\infty,\alpha_*)$ in Figure \ref{figCGb}. Notice that the Lopatinski\u{\i} function does not vanish in $\Re \gamma \geq 0$ in case (a), whereas in case (b) two zeroes along the imaginary axis emerge (this is particularly noticeable in the 3D plot on the left). These figures illustrate the transition from uniform to weak stability stated in Proposition \ref{propCGt}.

\begin{figure}[t]
\begin{center}
\subfigure[$\alpha \in (\alpha_*,0)$]{\label{figCGa}\includegraphics[scale=.55, clip=true]{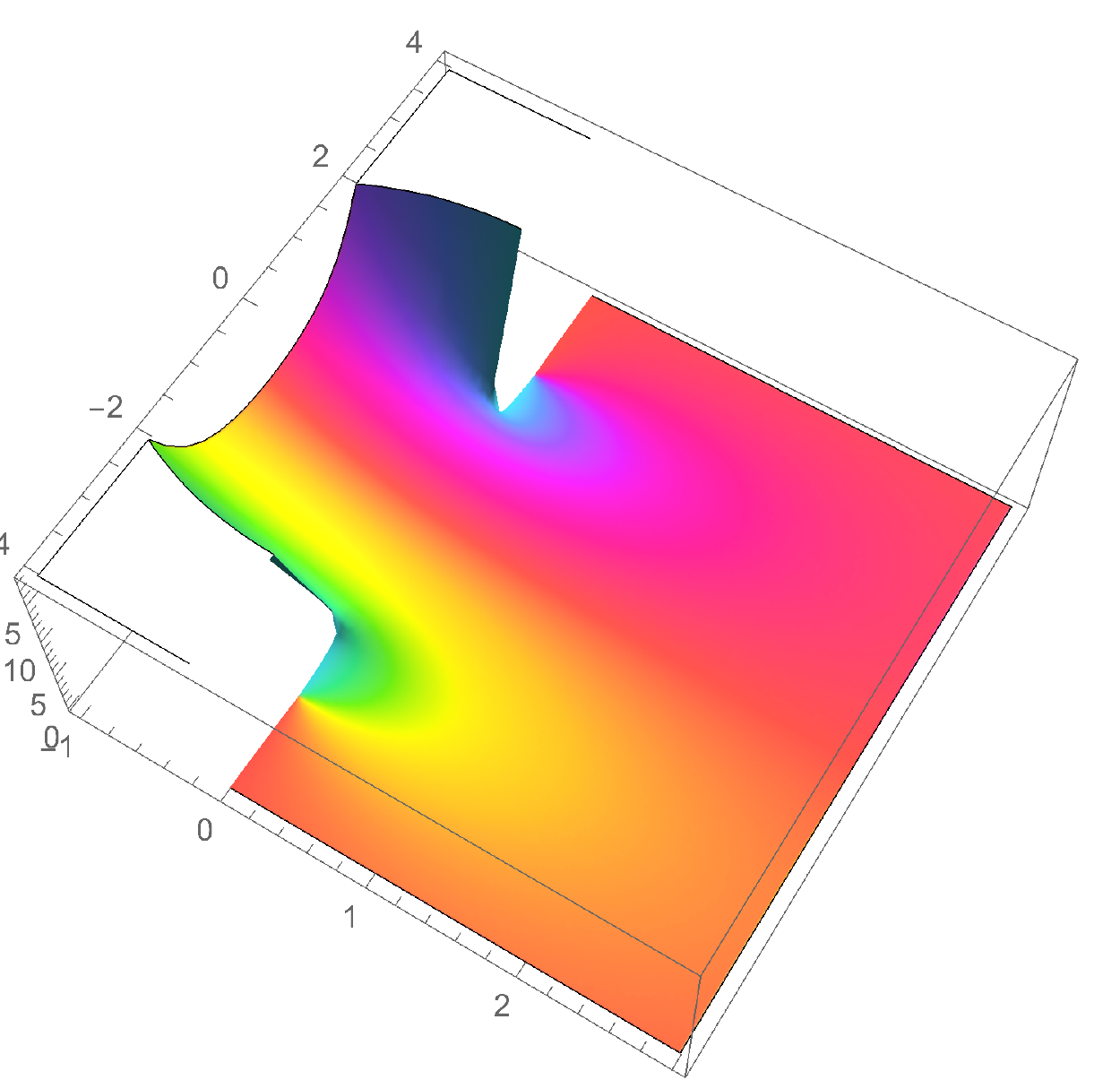}
\includegraphics[scale=.5, clip=true]{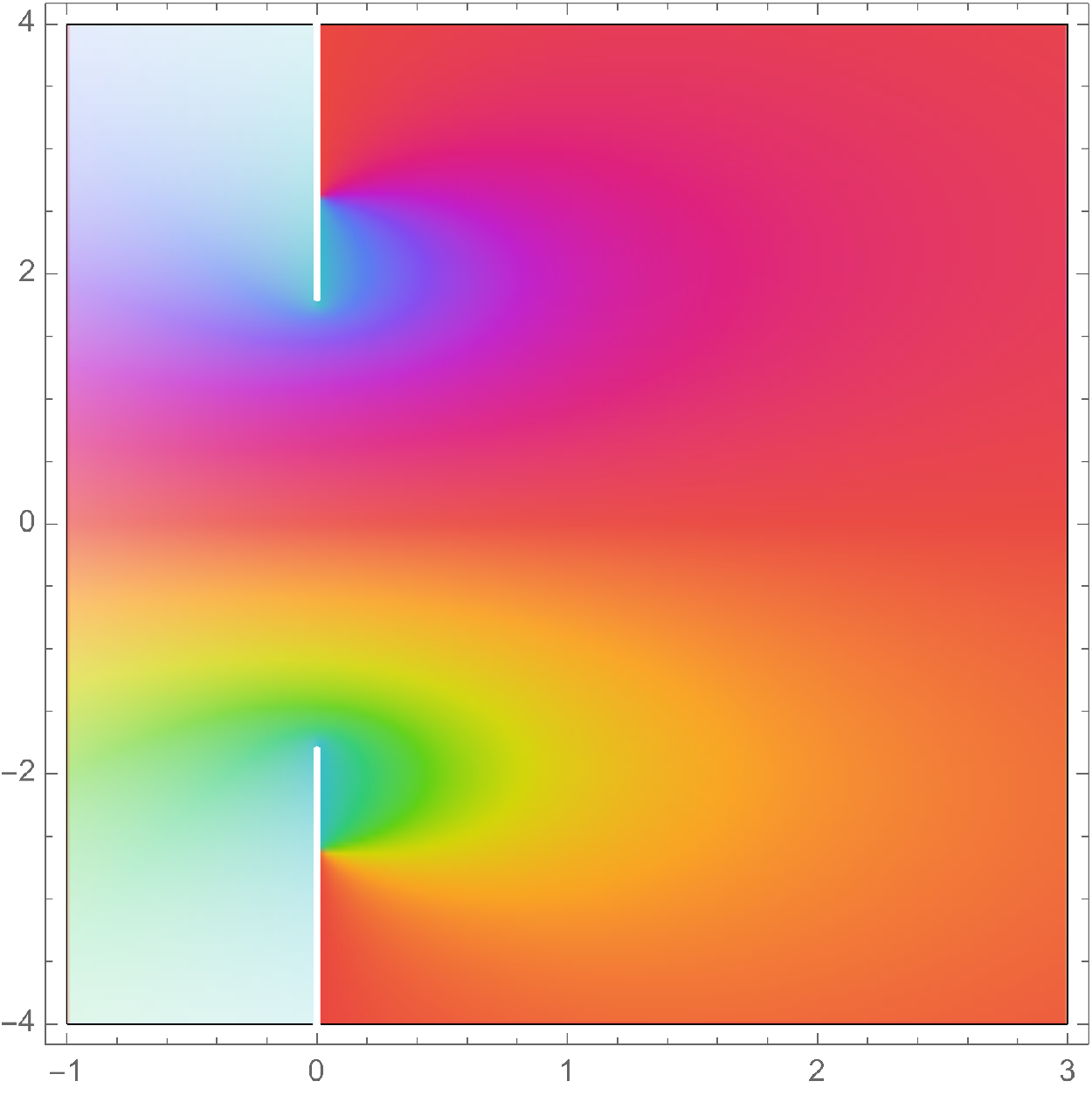}
\includegraphics[scale=.5, clip=true]{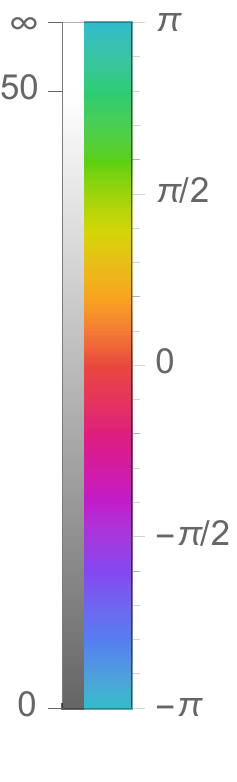}}
\subfigure[$\alpha \in (-\infty,\alpha_*)$]{\label{figCGb}\includegraphics[scale=.55, clip=true]{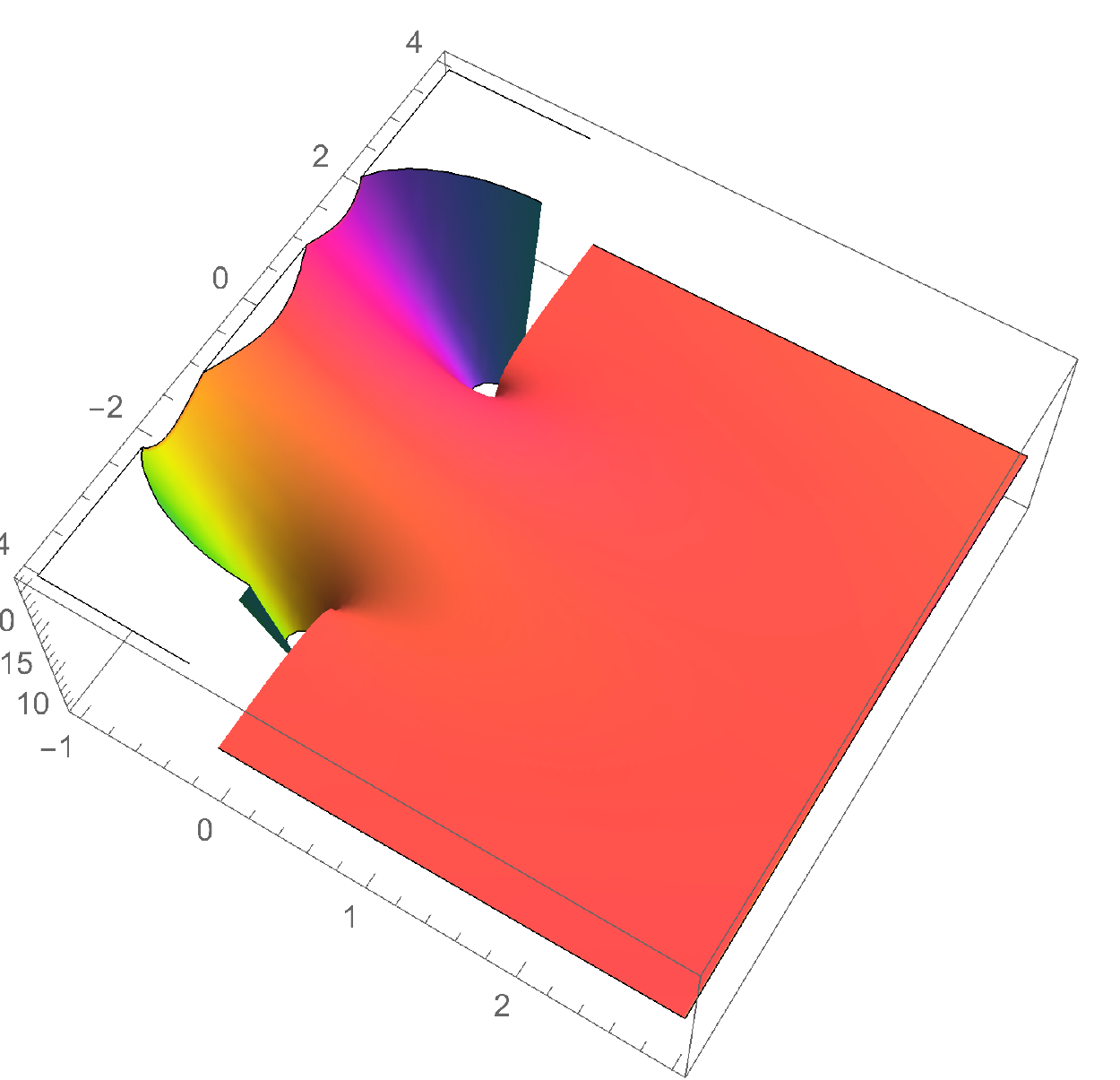}
\includegraphics[scale=.5, clip=true]{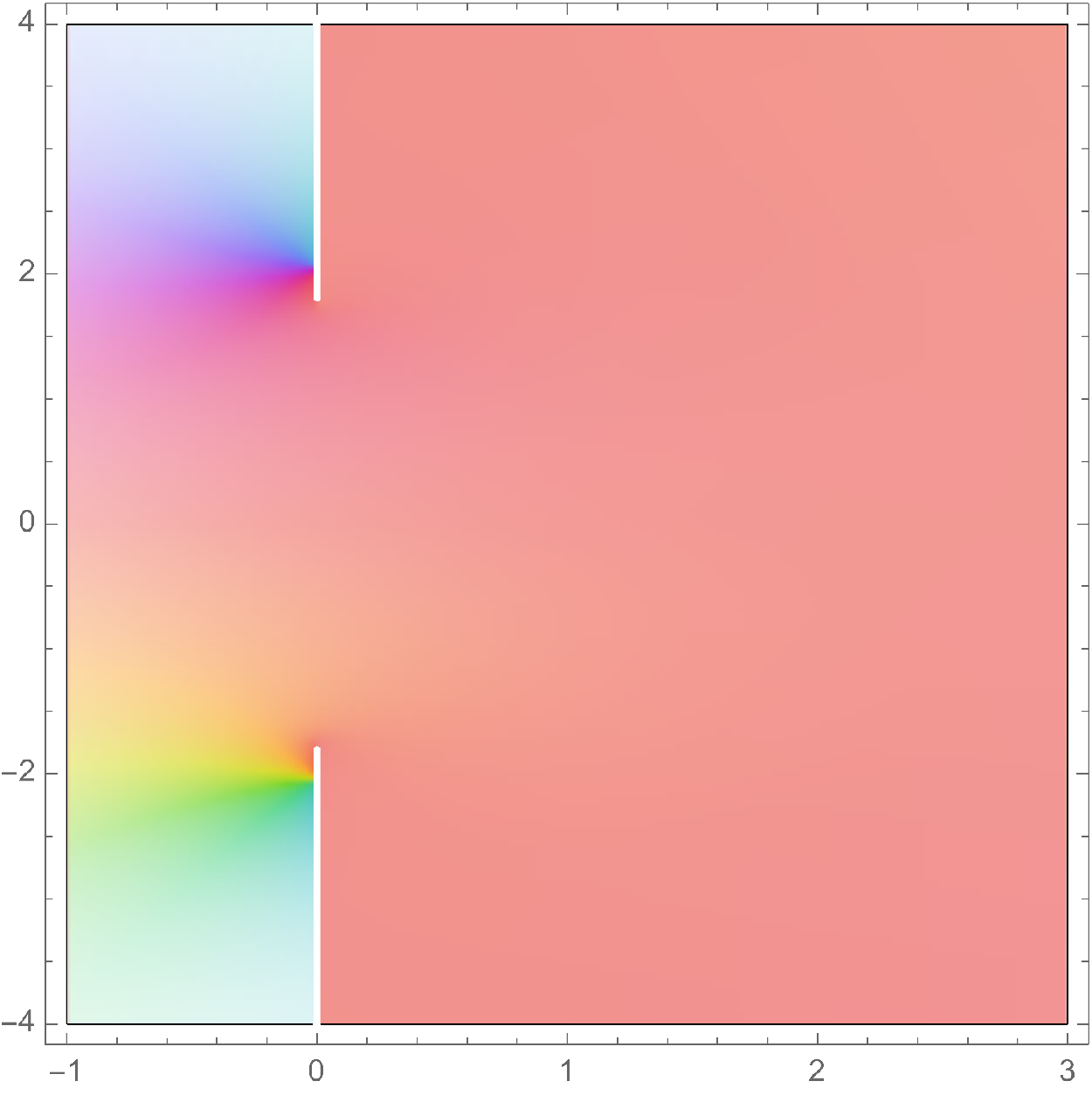}
\includegraphics[scale=.5, clip=true]{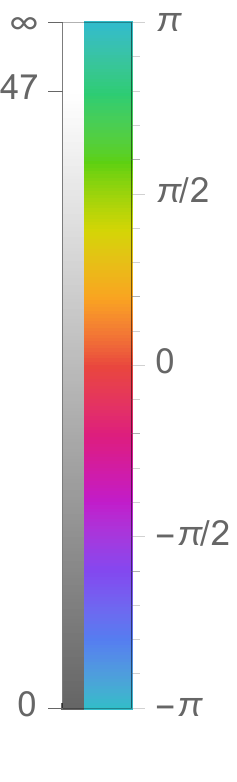}}
\end{center}
\caption{Complex plot (in 3D, left, and contour, right) of the Lopatinski\u{\i} determinant \eqref{LopdetCGd2} for the Ciarlet-Geymonat model \eqref{hCGd2} in dimension $d = 2$ as function of $\gamma \in \C$, with $\xi_2^2 = 1$, for elastic parameter values $\kappa =2$, $\mu = 1$ and for the shock parameter value $\alpha = -0.3 \in (\alpha_*,0)$ (panel (a)) and $\alpha = -8 \in (-\infty,\alpha_*)$ (panel (b)). The color mapping legend shows the modulus $|\Delta| \in (0,\infty)$ from dark to light tones of color and the phase from light blue ($\text{arg}(\gamma) = -\pi$) to green ($\text{arg}(\gamma) = \pi$). (Color online.)}\label{figCG}
\end{figure}

\subsubsection{Blatz model in dimension $d = 3$}

Let us now consider the model  proposed by Blatz \cite{Bla71} (see section \S \ref{subsecexamples}) in dimension $d = 3$,
\begin{equation}
\label{hB3d}
h(J) = - \frac{3}{2} \mu + \big( \kappa - \frac{2}{3} \mu\big) (J - 1) - \big( \kappa + \frac{\mu}{3} \big) \log J,
\end{equation}
where $\kappa > \tfrac{2}{3} \mu > 0$ are constant. This energy function, which models compressible elastomers, was studied in \cite{N-NSh71} from a numerical perspective. From \eqref{hB3d} we clearly have
\[
h''(J) = (\kappa + \tfrac{1}{3}\mu) \frac{1}{J^2} > 0, \qquad h'''(J) = - 2 (\kappa + \tfrac{1}{3}\mu) \frac{1}{J^3} < 0,
\]
for all $J \in (0,\infty)$ and conditions \eqref{H1} -- \eqref{H3} are satisfied. Thus, Proposition \ref{proppadonde} implies that for Lax shocks we require $\alpha < 0$. Use \eqref{RHmod} and $J^- = J^+ - \alpha |V_1^+|^2 = J^+ -\alpha \theta_{11}^+$ to write
\[
s^2 - \mu = (\kappa + \tfrac{1}{3}\mu ) \frac{\theta_{11}^+}{J^+ J^-},
\]
yielding, in turn,
\[
\rho(\alpha) = h''(J^+) - \frac{(s^2 - \mu)}{\theta_{11}^+} \frac{J^-}{J^+} \equiv 0.
\]
In view of Theorem \ref{stabcriteria} we obtain the following
\begin{proposition}
\label{propBd3}
For the three-dimensional Blatz model \eqref{hB3d} all classical elastic shocks are uniformly stable.
\end{proposition}

As before, for the sake of simplicity we consider an undeformed base state, $U^+ = \Itr$, and $\alpha < 0$ to define the shock. In this fashion, $J^+ =1$, $V_1^+ = \hat{e}_1 \in \R^3$ and
\[
U^- = \Itr - \alpha (\hat{e}_1 \otimes \hat{e}_1) = \begin{pmatrix} 1-\alpha & 0 & 0 \\ 0 & 1 & 0 \\ 0 & 0 & 1 \end{pmatrix}, \qquad J^- = 1-\alpha > 1.
\]
Here, the transversal frequencies vector is $\xim = (\xi_2,\xi_3)^\top \in \R^2$ and $V_j^+ = \hat{e}_j \in \R^3$, for $j = 2,3$. This yields, $\eta^+(\xim) = \sum_{j \neq 1} (V_1^+)^\top V_j^+ \xi_j = \sum_{j \neq 1} \hat{e}_1^\top \hat{e}_j \xi_j = 0$. Direct calculations lead to
\[
\kappa_2^+ = \kappa + \tfrac{4}{3}\mu > 0, \qquad s^2 = \mu + \frac{\kappa + \tfrac{1}{3}\mu}{1-\alpha}, \qquad \zeta^+(\xim) = (\kappa + \tfrac{4}{3}\mu) |\xim|^2,
\]
with $s < 0$. Let us define
\[
C_1(\kappa,\mu,\alpha) := -\frac{\sqrt{\kappa_2^+}}{s}  = \sqrt{\frac{(1-\alpha)(3\kappa + 4\mu)}{(4 - 3\alpha)\mu + \kappa}} > 0.
\]
Since $\rho(\alpha) = 0$ and $\eta^+(\xim) = 0$, the second version of the Lopatinski\u{\i} determinant \eqref{Lopadetv2} then reduces to
\[
\widetilde{\Delta}(\gamma,\xim) = \left( \gamma + C_1(\kappa,\mu,\alpha) \Big( \gamma^2 + \zeta^+(\xim)\Big)^{1/2}\right)^2,
\]
for $(\gamma,\xim) \in \widetilde{\Gamma}^+$. Since $\eta^+(\xim) = 0$ the set of remapped frequencies $(\gamma,\xim) \in \widetilde{\Gamma}^+$ is given by
\[
\Re \gamma > 0, \qquad \frac{-\alpha(\kappa + \tfrac{1}{3}\mu)}{(1-\alpha)(\kappa + \tfrac{4}{3}\mu)} |\gamma|^2 + |\xim|^2 = 1.
\]
Solving for $|\xim|^2$ and substituting into the Lopatinski\u{\i} determinant we obtain the following expression as a function of $\gamma \in \C$ alone,
\begin{equation}
\label{LopDetBd3}
\widetilde{\widetilde{\Delta}} (\gamma) := \widetilde{\Delta}(\gamma,\xim)_{|(\gamma, \xim) \in \widetilde{\Gamma}^+}
=\left[ \gamma + C_1(\kappa,\mu,\alpha) \left( \gamma^2 + (\kappa + \tfrac{4}{3}\mu) + \frac{\alpha (\kappa + \tfrac{1}{3}\mu)}{1-\alpha} |\gamma|^2 \right)^{1/2} \right]^2.
\end{equation}

Figure \ref{figBlatz3d} shows both the 3D and contour plots of the Lopatinski\u{\i} determinant \eqref{LopDetBd3} as a function of $\gamma \in \C$, for elastic parameter values $\kappa =1$, $\mu = 1$ and for the shock parameter value $\alpha = -5$. Notice that the function never vanishes for $\Re \gamma \geq 0$, confirming the uniform stability of the shock stated in Proposition \ref{propBd3}.

\begin{figure}[t]
\begin{center}
\includegraphics[scale=.55, clip=true]{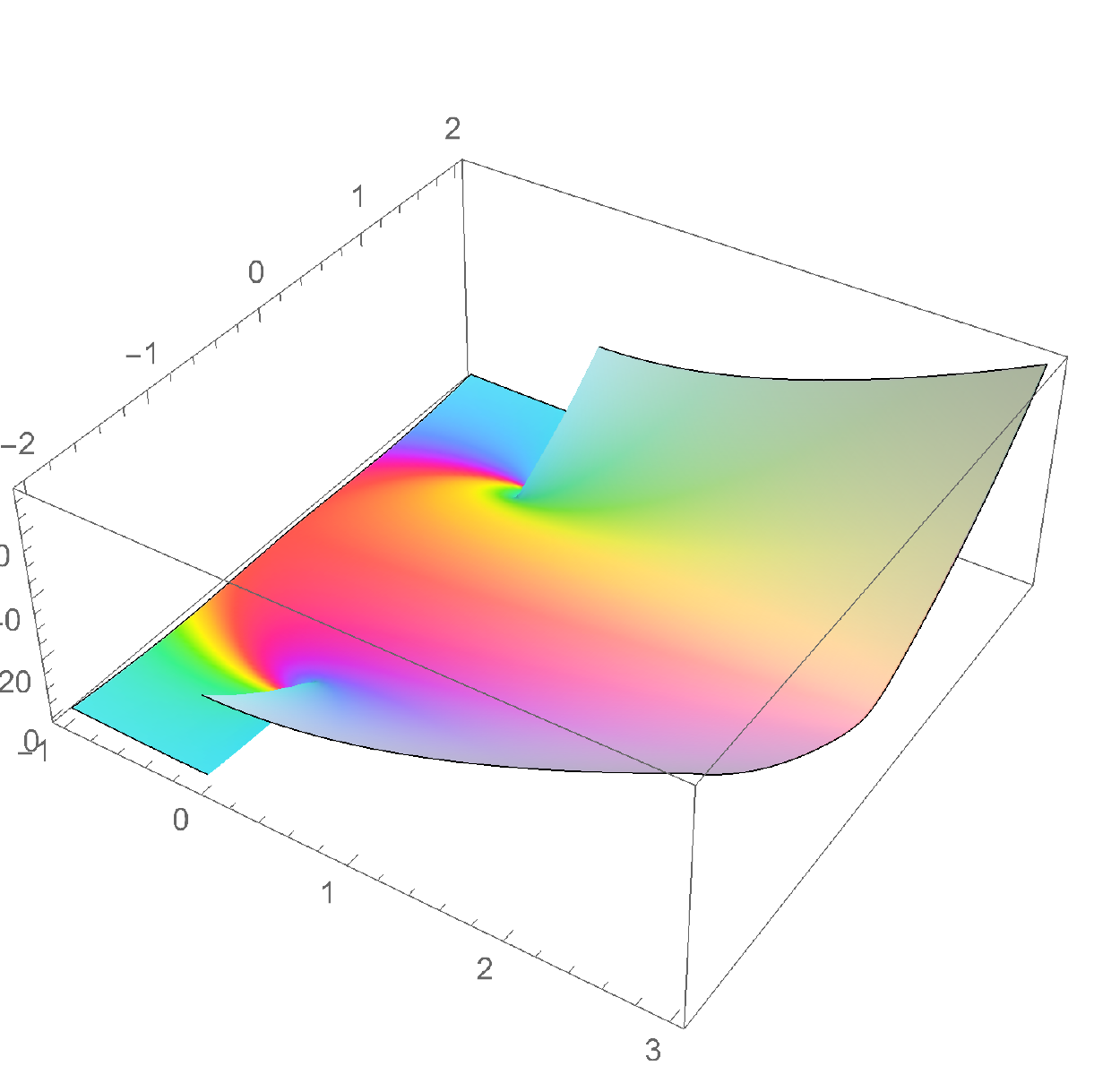}
\includegraphics[scale=.5, clip=true]{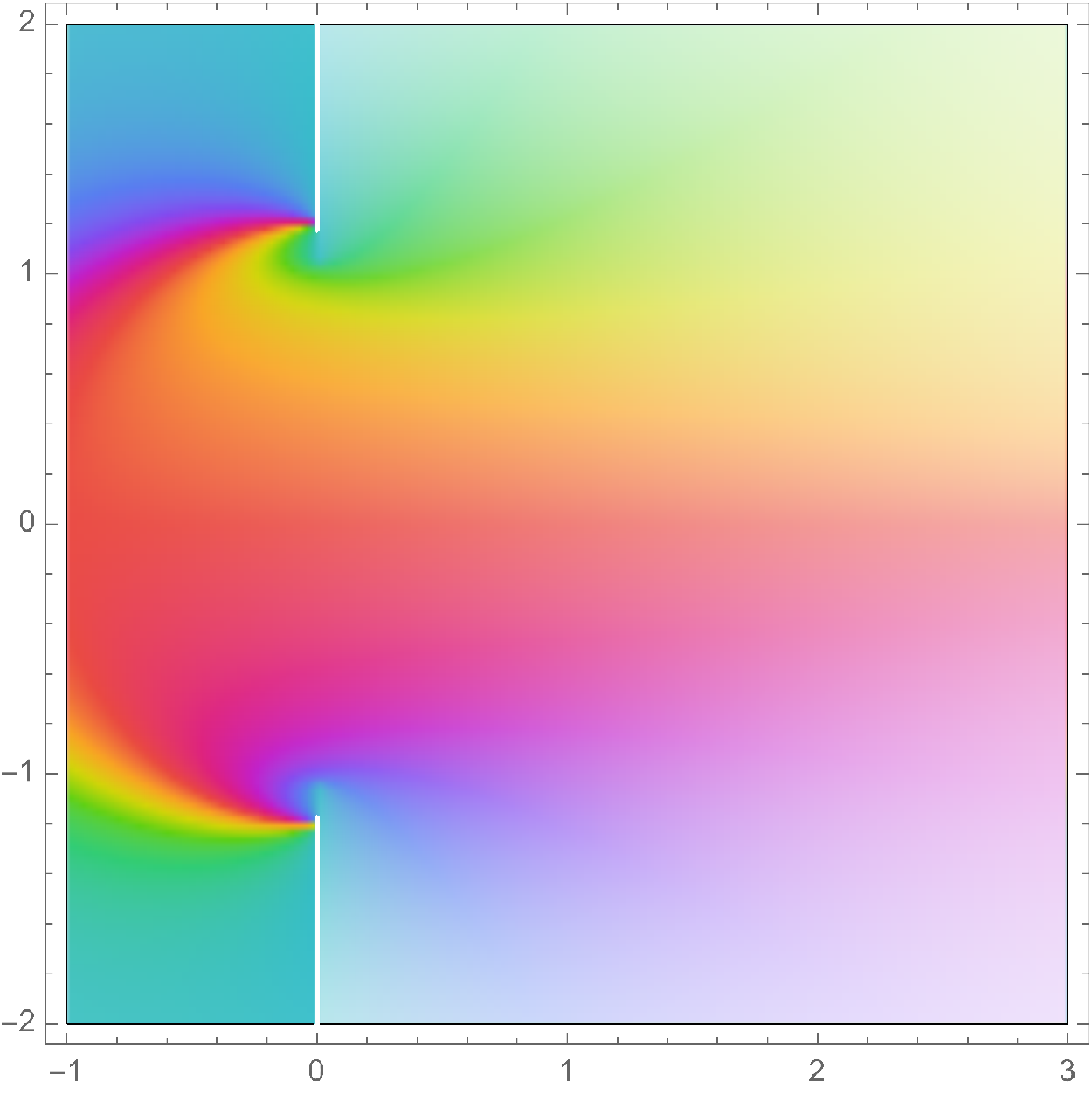}
\includegraphics[scale=.5, clip=true]{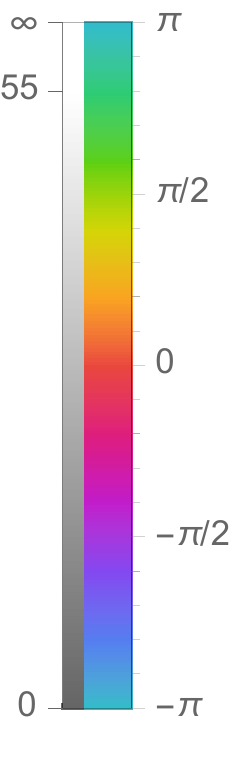}
\end{center}
\caption{Complex plot (in 3D, left, and contour, right) of the Lopatinski\u{\i} determinant \eqref{LopDetBd3} for the Blatz model \eqref{hB3d} in dimension $d = 3$ as function of $\gamma \in \C$ for elastic parameter values $\kappa =1$, $\mu = 1$ and for the shock parameter value $\alpha = -5$. The color mapping legend shows the modulus $|\Delta| \in (0,\infty)$ from dark to light tones of color and the phase from light blue ($\text{arg}(\gamma) = -\pi$) to green ($\text{arg}(\gamma) = \pi$). (Color online.)}\label{figBlatz3d}
\end{figure}

\section{Discussion}
\label{secdisc}

In this paper, we have explicitly computed and studied the Lopatinski\u{\i} determinant (or stability function) associated to classical planar shock fronts for compressible, non thermal, hyperelastic materials of Hadamard type in any space dimension. The stored energy density functions characterizing such materials have the form \eqref{Hadamardmat} and satisfy hypotheses \eqref{H1} and \eqref{H2}. Once a base state is selected, all elastic classical shocks can be described in terms of a shock parameter $\alpha \in \R\backslash \{0\}$ which determines the shock speed, the end state and the shock amplitude. For simplicity, we assume that the material further satisfies the material convexity condition \eqref{H3}. It is shown that for materials satisfying \eqref{H1} -- \eqref{H3} all classical shocks are, at least, weakly stable. This is tantamount to the fact that Hadamard-type ill-posed examples cannot be constructed for the linearized problem. In several space dimensions, it is known that the transition from a weakly stable to a strongly unstable shock is signaled by the instability with respect to one dimensional perturbations (see Serre \cite{Ser8}). Hence, Corollary \ref{cor1dstab} (which establishes the one-dimensional stability of all shocks) is consistent with the absence of violent multidimensional instabilities.

Moreover, the explicit calculation of the Lopatinski\u{\i} determinant as a function of the space-time frequencies allows to perform a complete (spectral) study of the constant coefficients problem analytically. We introduce a scalar stability parameter, $\rho(\alpha)$, depending solely on the shock parameters and on the elastic moduli of the material, which determines the transition from uniform to weak stability according to the condition \eqref{critdef}. In the cases where the shock is weakly stable, we introduced a mapping in the frequency space which allows to locate two zeroes along the imaginary axis. In the case where the uniform stability condition holds, one may directly conclude the nonlinear stability of the shock as well as the persistence of the front structure (local-in-time existence and uniqueness of the shock wave for the nonlinear system of equations), in view that the analyses of Majda \cite{M1,M2} and M\'{e}tivier \cite{Me5} apply. For that purpose, it is to be observed that the system of elasytodynamics satisfies the block structure assumption of Majda (see \cite{Co1}) and the constant multiplicity of M\'{e}tivier (see Corollary \ref{corHadLG} above), allowing the construction of Kreiss symmetrizers and the establishment of energy estimates for the linearized coefficients problem (see \cite{M1,M2,BS}). The nonlinear conclusion is, thus, at hand. The local-in-time existence of weakly stable shocks for hyperelastic materials remains an open problem. 

The explicit computation of the Lopatinski\u{\i} determinant presented here could be useful in the study of \emph{elastic phase boundaries} for Hadamard materials, which are structures associated to the case where the volumetric energy density $h$ has the shape of a double-well potential (for a recent contribution in this direction, see \cite{GrTr19}). Such investigation must follow the theoretical setup developed in \cite{FrP1} and (perhaps) the numerical approach of \cite{Pl2}, in order to deal with kinetic relations which are dissipative perturbations of the Maxwell equal area rule. This is a problem that warrants future investigations.

\section*{Acknowledgements}

The authors are warmly grateful to Heinrich Freist\"uhler for suggesting the problem and for many stimulating conversations. The authors also thank Andrea Corli for kindly calling their attention to his work in \cite{Co1}. Finally, the authors thank an anonymous referee whose careful revision and sharp comments improved the quality and scope of this paper. This research was supported by DGAPA-UNAM, program PAPIIT, grant IN-104922. The work of F. Vallejo was partially supported by CONACyT (Mexico) through a scholarship for doctoral studies, grant no. 740356/596608.

\appendix

\section{Multidimensional stability of planar shock fronts}
\label{sechypall}

For convenience of the reader, in this section we gather basic information about the stability conditions for multidimensional shock fronts. The reader is referred to the books by Benzoni-Gavage and Serre \cite{BS}, Majda \cite{M1,M3} and Serre \cite{Ser2} for more information. Consider a hyperbolic system of $n$ conservation laws in $d \geq 2$ space dimensions of the form,
\begin{equation}
\label{HSCL}
u_t + \sum_{j=1}^d f^j(u)_{x_j} = 0,
\end{equation}
where $x \in \R^d$ and $t \geq 0$ are space and time variables, respectively, and $u \in \cU \subset \R^n$ denotes the vector of $n$ conserved quantities (here $\cU$ denotes an open connected set). The flux functions $f^j \in C^2(\cU; \R^n)$, $j = 1, \ldots, d$, are supposed to be twice continuously differentiable and to determine the flux of the conserved quantities along the boundary of arbitrary volume elements. System \eqref{HSCL} is hyperbolic in $\cU$ if for any $u \in \cU$ and all $\xi \in \R^d$, $\xi \neq 0$, the matrix
\begin{equation}
\label{defA}
A(\xi,u) := \sum_{j=1}^d \xi_j A^j(u),
\end{equation}
where $A^j(u) := Df^j(u) \in \R^{n \times n}$ for each $j$, is diagonalizable over $\R$ with eigenvalues
\begin{equation}
\label{charspeeds}
a_1(\xi,u) \leq \ldots \leq a_n(\xi,u),
\end{equation}
of class at least $C^1(\cU \times \R^d; \R)$, called the \emph{characteristic speeds}. Each eigenvalue $a_j(\xi,u)$ is semi-simple (algebraic and geometric multiplicities coincide), with constant multiplicity for all $(u,\xi) \in \cU \times \R^d \backslash \{0\}$. The matrix $A(\xi,u)$ has a complete set of right (column) eigenvectors $r_1(\xi,u), \ldots, r_n(\xi,u) \in C^1(\R^d \times \cU;\R^{n \times 1})$, satisfying $A(\xi,u) r_j(\xi,u) = a_j(\xi,u) r_j(\xi,u)$ for each $j$, as well as a complete set of left (row) eigenvectors $l_1(\xi,u), \ldots, l_n(\xi,u) \in C^1(\R^d \times \cU;\R^{1 \times n})$, satisfying $l_j(\xi,u) A(\xi,u) = a_j(\xi,u) l_j(\xi,u)$.

An important class of weak solutions to \eqref{HSCL} are known as \emph{shock fronts}, which are configurations of the form
\begin{equation}
\label{gshock}
u(x,t) = \begin{cases}
u^+, & x \cdot \hat{\nu} > st,\\ u^-, & x \cdot \hat{\nu} < st,
\end{cases}
\end{equation}
where $u^\pm \in \cU$ are constant states, $u^+ \neq u^-$, and $\hat{\nu} = (\nu_1, \ldots, \nu_d) \in \R^d$, $|\hat{\nu}| = 1$ is a fixed direction of propagation. The shock speed $s \in \R$ is not arbitrary but determined by the classical Rankine-Hugoniot jump conditions \cite{La1,Da4e},
\begin{equation}
\label{gRH}
-s \llb u \rrb + \sum_{j=1}^d \llb f^j(u) \rrb \nu_j = 0,
\end{equation}
where the bracket $\llb \cdot \rrb$ denotes the jump across the interface or, more precisely,
\[
\llb g(u) \rrb := g(u^+) - g(u^-),
\]
for any (vector or matrix valued) function $g = g(u)$. Jump conditions \eqref{gRH} are necessary conditions for the configuration \eqref{gshock} to be a weak solution to \eqref{HSCL} and express conservation of the state variables $u$ across the interface, $\Sigma = \{x \cdot \hat{\nu} - st = 0 \}$. 

To circumvent the problem of non-uniqueness of weak solutions of the form \eqref{gshock} one further imposes an entropy condition. The shock front \eqref{gshock} is called an \emph{admissible} (or \emph{classical}) $p$-\emph{shock} if it satisfies Lax entropy condition (cf. \cite{La1,Da4e}): there exists an index $1 \leq p \leq n$ such that
\begin{equation}
\label{gLax}
\begin{aligned}
a_{p-1}(\hat{\nu},u^-) & < s  < a_p(\hat{\nu},u^-),\\
a_{p}(\hat{\nu},u^+) & < s  < a_{p+1}(\hat{\nu},u^+),
\end{aligned}
\end{equation}
where, by convention, if $p=1$ then $a_{p-1}(\hat{\nu}, u^-) := - \infty$, and if $p = n$ then $a_{p+1}(\hat{\nu}, u^+) := +\infty$. In the case where $p=1$ or $p=n$ the shock is called \emph{extreme}. The eigenvalue $a_p(\hat{\nu},u)$ is called the principal characteristic speed and $r_p(\hat{\nu},u)$ is the principal characteristic field. It is said that the former is \emph{genuinely nonlinear in the direction} $\hat{\nu}$ (cf. Majda \cite{M3}) if $D_u a_p(\hat{\nu},u)^\top r_p(\hat{\nu},u) \neq 0$ (or equivalently, $l_p(\hat{\nu},u) D_u a_p(\hat{\nu},u) \neq 0$) for all $u \in \cU$.

Given a base state $u^+ \in \cU$, the Hugoniot locus is defined as the set of all states in $\cU$ that can be connected to $u^+$ with a speed satisfying the jump conditions \eqref{gRH}. The intersection of the Hugoniot locus with those states for which one can find a shock speed satisfying Lax entropy condition \eqref{gLax} for some $1 \leq p \leq n$ is referred to as the $p$-\emph{shock curve}. If, in addition, $u^+ \in \cU$ is a point of genuine nonlinearity of the $p$-th characteristic family in direction of $\hat{\nu}$, for which $a_p(\hat{\nu}, u^+)$ is a simple eigenvalue and 
\begin{equation}
\label{ggnl}
D_u a_p(\hat{\nu}, u^+)^\top r_p(\hat{\nu}, u^+) > 0, \qquad \text{(respectively, $< 0$),}
\end{equation}
then the $p$-shock curve locally behaves like 
\begin{equation}
\label{gshockcurve}
\begin{aligned}
u^- &= u^+ + \epsilon \, r_p(\hat{\nu}, u^+) + O(\epsilon^2),\\
s &= a_p(\hat{\nu}, u^+) + \tfrac{1}{2} \epsilon \, D_u a_p(\hat{\nu}, u^+)^\top r_p(\hat{\nu}, u^+) + O(\epsilon^2),
\end{aligned}
\end{equation}
and satisfies Lax entropy condition \eqref{gLax} if and only if $\epsilon < 0$ (respectively, $\epsilon > 0$). The parameter $\epsilon$ measures the strength of the shock, $|u^+ - u^-| = O(|\epsilon|)$.

It is well known that the nonlinear stability behavior of shock fronts of the form \eqref{gshock} is determined by the so called \emph{uniform and weak Lopatinski\u{\i} conditions} (see Benzoni-Gavage and Serre \cite{BS}, Majda \cite{M1,M2}, M\'etivier \cite{Me2,Me1,Me5} and the references therein). The analysis to obtain the former departs from a Fourier-Laplace decomposition of the constant-coefficient linearized problem associated with \eqref{HSCL} at the configuration \eqref{gshock}. By considering single normal modes of the form $u \sim e^{\lambda t} e^{\ii \xi \cdot x}$ with spatio-temporal frequencies lying on the set
\begin{equation}
\label{genfreq}
\Gamma_{\hat{\nu}}^+ = \left\{ (\lambda,\xi) \in \C \times \R^d \, : \, \Re \lambda > 0, \, \xi \cdot \hat{\nu} = 0, \, |\lambda|^2 + |\xi|^2 = 1 \right\},
\end{equation}
as solutions to the linearized problem around the shock \eqref{gshock}, one arrives at the \emph{Lopatinski\u{\i} determinant} or stability function
\begin{equation}
\label{gLopdet}
\Delta(\lambda,\xi) = \det \Big( \cR_1^-, \ldots, \cR_{p-1}^-, \, \lambda \llb u \rrb + \ii \sum_{j=1}^d  \xi_j \llb f^j(u) \rrb, \cR_{p+1}^+, \ldots, \cR_n^+ \Big), \qquad (\lambda,\xi) \in \Gamma_{\hat{\nu}}^+,
\end{equation}
where $\cR_1^-(\lambda,\xi), \ldots, \cR_{p-1}^-(\lambda,\xi) \in \C^{n \times 1}$ denotes a basis of the stable subspace of $\cA^-(\lambda,\xi)$, and $\cR_{p+1}^+(\lambda,\xi), \ldots, \cR_{n}^+(\lambda,\xi) \in \C^{n \times 1}$ denotes a basis of the unstable subspace of $\cA^+(\lambda,\xi)$, whereupon we define the matrix fields
\[
\cA^\pm(\lambda,\xi) := (\lambda \In + \ii A(\xi, u^\pm))(A(\hat{\nu}, u^\pm) - s\In)^{-1} \in \C^{n \times n}, \qquad (\lambda,\xi) \in \Gamma_{\hat{\nu}}^+.
\]
Notice that, in view of Lax entropy conditions, the shock is not characteristic with $s \neq a_p^\pm$ and hence the matrices $A(\hat{\nu}, u^\pm) - s\In$ are not singular. The fact that the stable subspace of $\cA^-(\lambda,\xi)$ and the unstable subspace of $\cA^+(\lambda,\xi)$ have exactly dimensions $p-1$ and $n-p$, respectively, follows from the hyperbolicity of the matrix fields $\cA^\pm(\lambda,\xi)$ on the set $\Gamma_{\hat{\nu}}^+$. This result is known in the literature as \emph{Hersh' lemma} \cite{Her63} (see also \cite{BS,JL,Ser2}). 

The function $\Delta$ is jointly analytic in $(\lambda,\xi) \in \Gamma_{\hat{\nu}}^+$ and homogeneous of degree one.  Also, by continuity of the eigenprojections, the Lopatinski\u{\i} determinant can be defined for all frequencies within the set
\[
\Gamma_{\hat{\nu}} := \left\{ (\lambda,\xi) \in \C \times \R^d \, : \, \Re \lambda \geq 0, \, \xi \cdot \hat{\nu} = 0, \, |\lambda|^2 + |\xi|^2 = 1 \right\},
\]
(see \cite{Kre70,M2,M1,Me2} for further information). The stability function $\Delta$ determines the solvability of the linearized equations by wave solutions that violate an $L^2$ well-posedness estimate. Whenever a zero of $\Delta$ occurs then there exist spatially decaying solutions with time growth rate $\exp (t \, \Re \lambda)$. Thus, a necessary condition for well-posedness of the linearized problem is that $\Delta$ does not vanish in the open set $\Gamma_{\hat{\nu}}^+$. A stronger condition requires $\Delta$ not to vanish in the whole frequency set $\Gamma_{\hat{\nu}}$ (allowing time frequencies with $\Re \lambda = 0$) and it is sufficient for the well-posedness of the nonlinear system, as the analyses of Majda \cite{M2,M1} and M\'etivier \cite{Me1,Me2} show. To sum up, we have the following

\begin{defin}
\label{defgstab}
Consider a planar shock wave of the form \eqref{gshock} and its corresponding Lopatinski\u{\i} determinant defined in \eqref{gLopdet}. If $\Delta$ has no zeroes $(\lambda,\xi)$ in $\Gamma_{\hat{\nu}}$ the shock is called \emph{uniformly stable} (uniform Lopatinski\u{\i} condition). If $\Delta$ has a zero $(\lambda,\xi)$ in $\Gamma_{\hat{\nu}}^+$ (with $\Re \lambda > 0$) the shock is referred to as \emph{strongly unstable}. In the intermediate case where $\Delta$ has some zero $(\lambda,\xi)$ with $\Re \lambda = 0$ but no zero in $\Gamma_{\hat{\nu}}^+$ the shock is said to be \emph{weakly stable} (weak Lopatinski\u{\i} condition).
\end{defin}
\begin{remark}
\label{remlopdetextr}
When a shock is extreme with $p=1$ then there is no stable subspace of $\mathcal{A}^-(\lambda,\txi)$ for $(\lambda,\txi) \in \Gamma_{\hat\nu}^+$ and the unstable subspace of $\mathcal{A}^+(\lambda,\txi)$ has dimension $n-1$. Therefore, the left \emph{stable} subspace of $\mathcal{A}^+(\lambda,\txi)$ is generated by a single (row) vector $l_+^s(\lambda,\txi)$ associated to a unique stable eigenvalue $\beta(\lambda,\txi)$ with $\Re \beta < 0$. In such a case the expression for the Lopatinski\u{\i} determinant simplifies to
\begin{equation}
\label{Lopdetextr}
\overline{\Delta}(\lambda,\txi) = l_+^s(\lambda,\txi) \Big( \lambda \llb u \rrb + \ii \sum_{j=1}^d  \xi_j \llb f^j(u) \rrb \Big), \qquad (\lambda,\txi) \in \Gamma_{\hat\nu}^+,
\end{equation}
in the sense that $\Delta = 0$ in $\Gamma_{\hat\nu}^+$ if and only if $\overline{\Delta} = 0$ in $\Gamma_{\hat\nu}^+$ (see \cite{BS,JL,Ser2}).
\end{remark}

When a shock is strongly unstable, the instability is of Hadamard type \cite{Hada1902,Ser2} and it is so violent that we practically never observe the shock evolve in time. In contrast, any small initial perturbation around a strongly stable shock (that is, a small wave impinging on the interface), compatible with the conservation laws and the jump conditions, produces a (local-in-time) solution to the nonlinear system with the same wave structure, that is, made of smooth regions separated by a (modified or curved) shock front. As shown by Majda \cite{M1}, the strong stability condition ensures the well-posedness of a non-standard constant coefficient initial boundary value problem. The intermediate case of a weakly stable shock for which there exist zeroes of the Lopatinski\u{\i} determinant on the imaginary axis ($\Delta(\ii \tau,\txi) = 0$, for frequencies $(\ii \tau, \xi) \in \partial \Gamma_{\hat{\nu}}^+$, $\tau \in \R$) refers to the existence of \emph{surface wave} solutions localized near the shock, having the form $\Phi(|x\cdot \hat{\nu}|) e^{\ii (\tau t + x\cdot \xi)}$ and with amplitude $\Phi$ decaying exponentially as we move away from the interface, $|x\cdot \hat{\nu}| \to \infty$.

\section{Compressible hyperelastic materials of Hadamard type}
\label{haddimate}

An hyperelastic material of \emph{Hadamard type} (cf. \cite{Hay68,Joh66}) is defined as an elastic material whose energy density function has the general form \eqref{Hadamardmat}, where $\mu > 0$ is a constant and $h : (0,\infty) \to \R$ is a function of class $C^3$. According to custom let us denote
\[
I^{(1)} =  \Tr (U^\top U),  \quad I^{(d)} = \det (U^\top U), \quad J = \sqrt{I^{(d)}} = \det U.
\]
$I^{(1)}$ and $I^{(d)}$ are well-known principal invariants of the right Cauchy-Green tensor, $C = U^\top U$, for any given deformation gradient $U \in \M^d_+$. Hence, energy densities for compressible Hadamard materials have the (Rivlin-Ericksen) form 
\begin{equation}
\label{HM-RErep}
W(U) = \overline{W}(I^{(1)}, J) = \frac{\mu}{2} I^{(1)} + h(J).
\end{equation}
The constant $\mu > 0$ is the classical shear modulus in the reference configuration, describing an object's tendency to deform its shape at constant volume when acted upon opposing forces. The energy density \eqref{HM-RErep} consists of two contributions: the first term is the \emph{isochoric} part of the energy, quantifying energy changes at constant volume and depending only on $I^{(1)}$, whereas the second one, the \emph{volumetric} function $h = h(J)$, quantifies energy changes due to changes in volume, and depends only on $J = \det U \in (0,\infty)$. In this paper, it is assumed that the function $h$ satisfies the regularity assumption \eqref{H1} ($h \in C^3$), the convexity condition for the energy \eqref{H2} ($h'' > 0$) and the material convexity condition \eqref{H3} ($h''' < 0$).

\begin{remark}
Hayes \cite{Hay68} calls \emph{restricted Hadamard materials} to those which, in addition to \eqref{H1} and \eqref{H2}, satisfy
\begin{equation}
\label{hdecreasing}
h'(J) \leq 0,  \quad \text{for all } \, J > 0,
\end{equation}
a condition which guarantees that the elastic medium fulfills the \emph{ordered forces inequality} of Coleman and Noll \cite{ColNol59}. Even though some of the examples of elastic materials presented in this paper satisfy inequality \eqref{hdecreasing}, the latter plays no role in the shock stability analysis.
\end{remark}

\subsection{Stress fields} We now derive the first Piola-Kirchhoff and Cauchy stress tensors from any energy density function of the form \eqref{Hadamardmat}.

\begin{lemma}
\label{lemfirstpk}
For a general compressible elastic model with energy density of the form $W = \overline{W}(I^{(1)}, J)$ in any dimension $d \geq 2$, the first Piola-Kirchhoff stress tensor is given by
\begin{equation}
\label{PiolKir}
\sigma(U) = 2 \frac{\partial \overline{W}}{\partial {I^{(1)}}} U + \frac{\partial \overline{W}}{\partial J} \big( \Cof U\big), \qquad U \in \M^d_+.
\end{equation}
Moreover, the Cauchy stress tensor is
\begin{equation}
\label{Cauchystress}
T(U) = \frac{2}{J} \frac{\partial \overline{W}}{\partial {I^{(1)}}} UU^\top + \frac{\partial \overline{W}}{\partial J} \Id, \qquad U \in \M^d_+.
\end{equation}
\end{lemma}
\begin{proof}
Follows from elementary computations: since $I^{(1)} = \sum_{h,k=1}^d U_{hk}^2$ then clearly $\partial_{U_{ij}} I^{(1)} = 2 U_{ij}$, $1 \leq i, j \leq d$; on the other hand, expression \eqref{usefula} above yields
\[
\frac{\partial W}{\partial U_{ij}} = 2 \frac{\partial \overline{W}}{\partial {I^{(1)}}} U_{ij} + \frac{\partial \overline{W}}{\partial J} (\Cof U)_{ij}, \qquad 1 \leq i,j \leq d.
\]
This shows \eqref{PiolKir}. Now, since the Cauchy stress tensor $T$ is related to $\sigma$ by $\sigma = J T U^{-\top}$ (cf. \cite{Ba4,Ci}), apply \eqref{exprcof} to obtain \eqref{Cauchystress}, as claimed.
\end{proof}

We have an immediate
\begin{corollary}
\label{corPKHad}
For compressible hyperelastic materials of Hadamard type, the first Piola-Kirchhoff stress tensor is given by
\begin{equation}
\label{HadPiolKir}
\sigma(U) = \mu U + h'(J) \, \Cof U , \qquad U \in \M^d_+.
\end{equation}
Furthermore, the Cauchy stress tensor is
\begin{equation}
\label{HadCauchystress}
T(U) = \frac{\mu}{J} UU^\top + h'(J) \Id, , \qquad U \in \M^d_+.
\end{equation}
\end{corollary}
\begin{proof}
Follows directly from \eqref{HM-RErep} and Lemma \ref{lemfirstpk}.
\end{proof}

Given any deformation gradient $U \in \M^d_+$, the principal stretches $\vartheta_j > 0$, $j = 1, \ldots, d$, are the square roots of the eigenvalues of the symmetric right Cauchy-Green tensor. Therefore,
\[
I^{(1)} = \Tr (U^\top U) = \sum_{j=1}^d \vartheta_j^2, \quad J = \det U = \prod_{j=1}^d \vartheta_j.
\]
The following observation is a generalization of the result established by Currie \cite{Cur04} in dimension $d = 3$.
\begin{proposition}
For any $d \geq 2$ the possible range for $I^{(1)}$ is given by
\[
\mathcal{D} = \{ (I^{(1)}, J) \in \R \times (0, \infty) \, : \, I^{(1)} \geq d J^{2/d}\}.
\]
\end{proposition}
\begin{proof}
It is a straightforward application of the inequality of arithmetic and geometric means on the principal stretches,
\[
I^{(1)} = \Tr (U^\top U) = \vartheta_1^2 + \ldots + \vartheta_d^2 \geq d \, \Big( \vartheta_1^2 \cdots \vartheta_d^2 \Big)^{1/d} = d \, (\det U)^{2/d} = d J^{2/d}.
\]
\end{proof}

The boundary of the domain $\partial \mathcal{D} = \{ (I^{(1)},J) \, : \, I^{(1)} = d J^{2/d}\}$ is associated to pure pressure deformations, and the value $(I^{(1)},J) = (d,1) \in \partial \mathcal{D}$ corresponds to no deformations, $U = \Id$, with a reference configuration in which $\vartheta_j = 1$ for all $1 \leq j \leq d$.

\subsection{Compressible neo-Hookean materials}

The simplest interpretation of an elastic Hadamard material is as a compressible extension of a neo-Hookean incompressible solid. Incompressible hyperelasticity is restricted to isochoric (volume preserving) deformations with $J = \det U =1$, which is a kinematic constraint. The best known incompressible hyperelastic model is the neo-Hookean material \cite{Kubo48,Riv48-I,TrNo3ed}, whose energy function (in arbitrary space dimensions) is given by
\begin{equation}
\label{WnH}
W_{\mathrm{nH}}(U) = \overline{W}_{\mathrm{nH}}(I^{(1)}) = \frac{\mu}{2} (I^{(1)} - d).
\end{equation}

This strain-energy function provides a reliable and
mathematically simple constitutive model for the nonlinear deformation behavior
of isotropic hyperelastic materials, such as vulcanized rubber, similar to Hooke's law. It predicts typical effects known from nonlinear elasticity within the small strain domain (in contrast to linear elastic materials the stress-strain curve for a neo-Hookean material is not linear). It was first proposed by Rivlin in 1948 \cite{Riv48-I}.
Notably, the energy function \eqref{WnH} may also be derived from
statistical theory, in which rubber is regarded as a three-dimensional 
network of long-chain molecules that are connected at a few points (cf. \cite{BiArGr01,Holz00}). 

The incompressibility hypothesis works well for vulcanized rubber (under very high hydrostatic 
pressure the material undergoes very small volume changes). There are other materials, however, which are either slightly compressible, or which may undergo considerable volume changes (like foamed rubber). Therefore, compressible models are needed in order to describe these elastic responses. Furthermore, it is known that incompressibility can cause numerical difficulties in the analysis of finite elements, and in such cases nearly incompressible models are often used \cite{HorSa04,LeTal94}. As a result, either motivated by numerical or by physical considerations, compressibility is often accounted by the addition of a strain energy describing the purely volumetric elastic response. In the case of the neo-Hookean model, compressible extensions have the form
\[
W(U) = \overline{W}(I^{(1)},J) = \overline{W}_{\!\mathrm{nH}}(I^{(1)}) + \overline{W}_{\!\mathrm{vol}}(J).
\]
This decoupled representation of the energy as the sum of isochoric and volumetric energies is very common for isothermal deformations. A compressible extension should satisfy $\overline{W}(I^{(1)},1) = \overline{W}_{\!\mathrm{nH}}(I^{(1)})$, that is, $\overline{W}_{\!\mathrm{vol}}(1) = 0$. In the case of energies of the form \eqref{Hadamardmat} we clearly have an isochoric contribution given by the neo-Hookean energy density \eqref{WnH} and a volumetric response given by $\overline{W}_{\!\mathrm{vol}}(J) = h(J) +\tfrac{1}{2}\mu d$. Pence and Gou \cite{PnGo15} discuss nearly incompressible versions of the neo-Hookean model, as well as the requirements on the material moduli for the models to be compatible with the small-strain regime. In the next section we review such requirements and extrapolate them to arbitrary space dimensions.

\subsection{Compressible theory of infinitesimal strain}

Since undeformed configurations are stress free, one requires that $\sigma = 0$ whenever $U = \Id$. In the case of a Hadamard material, this requirement leads, upon substitution into formula \eqref{HadPiolKir}, to the following relation between the shear modulus and the function $h$,
\begin{equation}
\label{relhmu}
h'(1) = - \mu.
\end{equation}
This relation can be interpreted as a free stress condition for no deformations in the incompressible boundary, precisely at $(I^{(1)},J) = (d,1) \in \partial \mathcal{D}$. 

The mean pressure field is defined as (see, e.g., \cite{TruTo60}, p. 545),
\[
\overline{p} := - \frac{1}{d} \Tr (T(U)) = - \frac{1}{d} \Tr \Big( \frac{2}{J} \frac{\partial \overline{W}}{\partial I^{(1)}} UU^\top + \frac{\partial \overline{W}}{\partial J} \Id\Big) = - h'(J) - \frac{\mu}{d} \frac{I^{(1)}}{J}.
\] 
For symmetric deformation states, $U = J^{1/d} \Id$ (or equivalently, $(I^{(1)},J) \in \partial \mathcal{D}$), Pence and Gou \cite{PnGo15} define 
\[
- \,\hat{p}(J) := - \overline{p}(dJ^{2/d},J) = h'(J) + \mu J^{\frac{2}{d}-1} = -p_{\mathrm{hyd}}(J) + \mu J^{\frac{2}{d}-1},
\]
where
\begin{equation}
\label{derhydp}
p_{\mathrm{hyd}} (J) = - \frac{\partial \overline{W}_{\!\mathrm{vol}}}{\partial J} = - h'(J),
\end{equation}
is the \emph{hydrostatic pressure} (cf. \cite{Holz00,TrNo3ed}), or the pressure the material experiences when the shear strain is zero. The appropriate definition of the \emph{bulk modulus} of infinitesimal strain theory is therefore
\[
\kappa := - \left. \frac{d \hat{p}}{dJ} \right|_{J=1} = \hat{p}'(1),
\]
describing volumetric elasticity or how resistant to compression the elastic medium is. Consequently, for a Hadamard material with strain energy of the form \eqref{Hadamardmat} we have $\partial \overline{W} / \partial I^{(1)} = \tfrac{\mu}{2}$ and $\partial \overline{W} / \partial J = h'(J)$, yielding
\begin{equation*}
{- \hat{p}'(J) =}  \mu \Big( \frac{2}{d} -1\Big) J^{\frac{2}{d}-2} + h''(J), 
\end{equation*}
and the following relation between the bulk and shear moduli
\begin{equation}
\label{relkapmu}
\kappa = \mu \Big( \frac{2}{d} -1\Big)  + h''(1).
\end{equation}

Since the strain energy must be positive for small strains (linear physical theory for small deformations), on restriction to infinitesimal deformations the shear and bulk moduli must be positive to ensure compatibility with the linear response (cf. \cite{CMR94}). The Poisson ratio can then be defined in arbitrary dimensions as
\[
\overline{\nu} := \frac{d \kappa - 2 \mu}{2 \mu + d(d-1) \kappa},
\]
measuring the ratio of strain in the direction of load over the strain in orthogonal directions. This definition extends the well known formulae for the Poisson ratio in dimension $d=2$, $\overline{\nu} = \frac{\kappa - \mu}{\kappa + \mu}$, and in dimension $d = 3$, $\overline{\nu} = \frac{3\kappa - 2\mu}{2(3\kappa + 2\mu)}$
(see \cite{MeGa01,ThJa92}). Although the admissible thermodynamic range for the Poisson ratio is $-1 \leq \overline{\nu} \leq 1/2$ in dimension $d = 3$ \cite{PnGo15}, and $-1 \leq \overline{\nu} \leq 1$ in dimension $d = 2$ \cite{MeGa01}, the standard range for consideration is $\overline{\nu} > 0$ ($\overline{\nu}$ is usually positive for most materials\footnote{with the exception, of course, of \emph{auxetic} materials for which the Poisson ratio can be negative.} because interatomic bonds realign with deformation). To sum up, in this paper it is assumed that
\begin{equation}
\label{Poissonpos}
\mu > 0,  \quad \kappa > \frac{2}{d} \mu > 0.
\end{equation}
The classical Lam\'e moduli of an elastic material are the shear modulus $\mu > 0$ (second Lam\'e parameter) and $\Lambda$ (first Lam\'e parameter)\footnote{the first Lam\'e constant is usually denoted in the literature with the Greek letter $\lambda$; however, in order to avoid confusion with the frequency $\lambda \in \C$ in the shock stability analysis, we use a different symbol for it.}; the former can be related to the bulk and shear moduli by 
\[
\Lambda = \kappa - \frac{2\mu}{d};
\]
see \cite{Ci,TrNo3ed}. Notice that, under assumption \eqref{Poissonpos}, $\Lambda > 0$ .

\begin{remark}
\label{remmatconvex}
In view of \eqref{derhydp}, condition \eqref{H3} implies that $p_{\mathrm{hyd}}'' (J) = - h'''(J) > 0$ for all $J \in (0, \infty)$. Hence, hypothesis \eqref{H3} can be interpreted as a material convexity condition for zero shear strain.
\end{remark}

\subsection{Examples}  
\label{subsecexamples}

The following models belong to the class of compressible hyperelastic materials of Hadamard type, whose energy density functions have the form \eqref{Hadamardmat} and satisfy assumptions \eqref{H1} and \eqref{H2}. They have been proposed in the materials science literature to describe different elastic responses. It is worth mentioning that there exist compressible models with energies of the form \eqref{Hadamardmat} but which do not satisfy the convexity assumption \eqref{H2} for all deformations $J \in (0,\infty)$, such as the original Simo-Pister model \cite{SiPi84} (see also \cite{Hrtmn10}), or the Ogden $\beta$-log model \cite{Og72b} (see eq. (6.137), p. 244 in \cite{Holz00}).

\subsubsection*{$\mathrm{(a)}$ Ciarlet-Geymonat model} As a first example consider the following volumetric strain energy function
\begin{equation}
\label{hCG}
h_{\mathrm{CG}}(J) = - \frac{d}{2} \mu - \mu \log J + \left( \frac{\kappa}{2} - \frac{\mu}{d}\right) (J -1 )^2,
\end{equation}
where $\mu$ and $\kappa$ are the shear and bulk moduli, respectively, satisfying \eqref{Poissonpos}. Notice that $h_{\mathrm{CG}}(1) = -d\mu/2$ and therefore the energy density $\overline{W}_{\mathrm{CG}} = \frac{\mu}{2} I^{(1)} + h_{\mathrm{CG}}(J)$ is normalized as $\overline{W}_{\mathrm{CG}}(d,1) = 0$. It also satisfies \eqref{relhmu} and \eqref{relkapmu} as the reader may easily verify. Finally, in view of \eqref{Poissonpos} there holds the convexity condition \eqref{H2} as
\[
h_{\mathrm{CG}}''(J) = \frac{\mu}{J^2} +  \Big({\kappa} - \frac{2 \mu}{d}\Big) > 0, \qquad J \in (0, \infty).
\]
In addition, there holds
\[
h_{\mathrm{CG}}'''(J) = - \frac{2 \mu}{J^3} < 0,
\]
for all $J \in (0,\infty)$. This model is an extension to arbitrary spatial dimensions of the strain energy
\[
\overline{W} = \frac{\mu}{2} (I^{(1)} -3) +  \left( \frac{\kappa}{2} - \frac{\mu}{3}\right) (J -1 )^2 - \mu \log J,
\]
proposed by Ciarlet and Geymonat \cite{CiGe82} (see also \cite{Og84}) in dimension $d = 3$.  It is a special form of the family of compressible Mooney-Rivlin materials (see Ciarlet \cite{Ci}, section 4.10, p. 189, formula (iii) in the limit $b \to 0$). $h_{\mathrm{CG}}$ is defined for all deformations $J \in (0,\infty)$ and satisfies $h_{\mathrm{CG}} \to \infty$ as $J \to \infty$ and as $J \to 0^+$. 

\subsubsection*{$\mathrm{(b)}$ Blatz model} The energy function
\begin{equation}
\label{hB}
h_{\mathrm{B}}(J) = - \frac{d}{2} \mu + \Big(\kappa - \frac{2}{d} \mu \Big) \big( J-1\big)  - \Big( \kappa + \Big(\frac{d-2}{d}\Big) \mu \Big) \log J, 
\end{equation}
where, once again, $\mu$ and $\kappa$ are the shear and bulk moduli, respectively, generalizes to arbitrary dimensions $d \geq 2$ the modified compressible neo-Hookean form of the energy proposed by Blatz \cite{Bla71} (see eq. (48), p. 36), in dimension $d = 3$:
\[
\overline{W} = \frac{\mu}{2} (I^{(1)} - 3) + \Big(\kappa - \frac{2}{3} \mu \Big) \big( J-1\big)  - \Big( \kappa + \frac{\mu}{3} \Big) \log J.
\]
This function fulfills normalization, $h_{\mathrm{B}}(1) = -d\mu/2$, as well as conditions \eqref{relhmu} and \eqref{relkapmu}, as it is easily verified. Moreover, 
\[
h_{\mathrm{B}}''(J) = \frac{1}{J^2} \Big( \kappa + \frac{(d-2) \mu}{d} \Big) > 0, \quad h_{\mathrm{B}}'''(J) = -\frac{2}{J^3} \Big( \kappa + \frac{(d-2) \mu}{d} \Big) < 0,
\]
for all $J \in (0,\infty)$. Notice that $h_{\mathrm{B}} \to \infty$ as $J \to \infty$ or as $J \to 0^+$. This energy was selected by Blatz as a candidate strain energy density to describe thermostatic properties of homogeneous isotropic continuous \emph{elastomers} (elastic polymers).

\subsubsection*{$\mathrm{(c)}$ Neo-Hookean Ogden compressible foam material} The energy function
\begin{equation}
\label{hO}
h_{\mathrm{O}}(J) = - \frac{d}{2} \mu + \frac{\mu}{2 c_1} \big( J^{-2c_1} - 1\big),
\end{equation}
where
\[
c_1 = \frac{\overline{\nu}}{1- (d-1) \overline{\nu}} = \frac{d \kappa - 2 \mu}{2 d \mu} > 0,
\]
was proposed by Ogden \cite{Og72b} to model highly compressible rubber-like materials for which significantly volume changes can occur with relatively little stress (such as foams). It belongs to what is known in the literature as the family of Ogden compressible rubber foam materials (see \cite{MacDon2e}, p. 161):
\[
\overline{W} = \sum_{p=1}^N \frac{\mu_p}{\alpha_p} \Big( \sum_{j=1}^d \vartheta_j^{\alpha_p} - d \Big) \, + \, \sum_{p=1}^N \frac{\mu_p}{\alpha_p c_p} (J^{- \alpha_p c_p}-1),
\]
specialized here to $N =1$ (neo-Hookean), $\mu_1 = \mu > 0$, $\alpha_1 = 2$ and $c_1$ given above. This neo-Hookean element of the family has been used as a basis for residually stressed extensions for energies that account for elastic responses of blood arteries in medical applications (cf. \cite{GoWa10}). Notice that $h_{\mathrm{O}}(1) = -d\mu/2$ (normalization) and relations \eqref{relhmu} and \eqref{relkapmu} hold. Moreover, the convexity condition holds as
\[
h_{\mathrm{O}}''(J) = \frac{\mu( 2c_1 +1)}{J^{2(c_1+1)}} > 0, \quad h_{\mathrm{O}}'''(J) =  - \frac{2\mu(c_1+1)(2c_1 +1)}{J^{2c_1+3}} < 0,
\]
for all $J \in (0,\infty)$. Notably $h_{\mathrm{O}} \to \infty$ as $J \to 0^+$ but $\lim_{J \to \infty} h_{\mathrm{O}}(J)$ exists.

\subsubsection*{$\mathrm{(d)}$ Levinson-Burgess model} Consider the following volumetric function
\begin{equation}
\label{hLB}
h_{\mathrm{LB}}(J) = - \frac{d}{2} \mu + \frac{\mu}{2} \Big( \overline{c} (J^2 -1) + 2(\overline{c} + 1) (1 - J) \Big),
\end{equation}
where
\[
\overline{c} = \frac{\kappa}{\mu} - \frac{2}{d} + 1 > 0.
\]
This is a generalization to any space dimension $d \geq 2$ of the three dimensional material considered by Kirkinis \emph{et al.} \cite{KOH04},
\[
\overline{W} = \frac{\mu}{2} \left( I^{(1)} -3 + \left( \frac{\kappa}{\mu} + \frac{1}{3}\right) (J^2 -1) - 2 \left( \frac{\kappa}{\mu} + \frac{1}{3} + 1\right) (J-1) \right),
\]
which is, in turn, a special case of a compressible polynomial material introduced by Levinson and Burgess \cite{LeBu71} to account for weakly compressible elastic media with Poisson ratio close to $\tfrac{1}{2}$ (in dimension $d = 3$). Notice that $h_{\mathrm{LB}}(1) = -d\mu/2$ (normalization), it satisfies \eqref{relhmu} and \eqref{relkapmu}, and 
\[
h_{\mathrm{LB}}''(J) = \mu \overline{c} > 0, \quad h_{\mathrm{LB}}'''(J) \equiv 0,
\]
for all $J \in (0,\infty)$.
 \subsubsection*{$\mathrm{(e)}$ Simo-Taylor material} The Simo-Taylor model \cite{SiTa91} (see also \cite{Hrtmn10}),
 \begin{equation}
\label{hST}
h_{\mathrm{ST}}(J) = - \frac{d}{2} \mu - \mu \log J + \frac{\Lambda}{2} \Big( \frac{J^2}{2} - \log J - \frac{1}{2}\Big),
\end{equation}
where $\mu$ is the shear modulus and $\Lambda = \kappa - 2\mu/d > 0$ is the first Lam\'e parameter, clearly satisfies $h_{\mathrm{ST}}(1) = -d\mu/2$ (normalization) and conditions \eqref{relhmu} and \eqref{relkapmu}. Furthermore, the convexity condition \eqref{H2} holds, as
\[
h_{\mathrm{ST}}''(J) = \frac{\Lambda}{2} + \big( \mu + \frac{\Lambda}{2} \big) \frac{1}{J^2} > 0,
\]
for all $J \in (0,\infty)$. Observe also that
\[
h_{\mathrm{ST}}'''(J) = - (2 \mu + \Lambda) \frac{1}{J^3} < 0, \qquad J \in (0,+\infty).
\]
When $J \to 0^+$ or $J \to \infty$, $h_{\mathrm{ST}}$ grows unboundedly. This energy form can be derived from (Gaussian) statistical mechanics of long-chain molecules with entropic sources of compressibility modeled thorough the logarithmic terms (cf. Bischoff \emph{et al.} \cite{BiArGr01}).
 
\subsubsection*{$\mathrm{(f)}$ Special compressible Ogden-Hill material} The volumetric response function
\begin{equation}
\label{hOH}
h_{\mathrm{OH}} (J) = - \frac{d}{2} \mu + \frac{1}{b} \big( J -1)^2,
\end{equation} 
where $\mu > 0$ is the shear modulus and $b > 0$ is an empirical coefficient, yields an energy density $\overline{W}_{\mathrm{OH}} = \frac{\mu}{2} I^{(1)} + h_{\mathrm{OH}} (J)$ that also belongs to the class of compressible Hadamard materials. Notice that $W_{\mathrm{OH}}(d,1) = 0$ (normalization) but $h_{\mathrm{OH}}'(1) = 0$ and, thus, it does not satisfy the free stress condition \eqref{relhmu}. It does satisfy the convexity condition as 
\[
h_{\mathrm{OH}}''(J) = \frac{2}{b} > 0, \quad h_{\mathrm{OH}}'''(J) \equiv 0,
\]
for all $J \in (0,\infty)$. Also, $h_{\mathrm{OH}} \to \infty$ as $J \to \infty$, whereas $h_{\mathrm{OH}} (0^+)$ is well-defined. This model is a particular case of the well-known family of compressible Ogden-Hill materials \cite{Hill68a,Hill68b,Og72b}
\[
\overline{W} = \sum_{p=1}^N \frac{\mu_p}{\alpha_p} \Big( \sum_{j=1}^d \vartheta_j^{\alpha_p} - d \Big) \, + \, \sum_{p=1}^N \frac{1}{b_p^2} (J-1)^{2N},
\]
specialized to $N =1$, $\mu_1 = \mu > 0$, $\alpha_1 = 2$ and $b_1 = b > 0$. The family was proposed to model highly compressible materials such as low density polymer foams (cf. \cite{Mills07,ECLe18}). The parameter $b > 0$ is adjusted from experimental data. It is a modulus that measures compressibility: if $b$ is small then the material is highly compressible, whereas if $b$ is large then the material can be considered as nearly incompressible. It is used in the analysis of elastomers, as well as in the design of O-rings, seals and other industrial products \cite{MacDon2e}.

\subsubsection*{$\mathrm{(g)}$ Simo-Miehe model} The following energy function proposed by Simo and Miehe \cite{SiMi92} (see also \cite{Holz00}),
\begin{equation}
\label{hSM}
h_{\mathrm{SM}}(J) = - \frac{d}{2} \mu + \frac{\kappa}{4} \big( J^2 - 1 - 2 \log J \big),
\end{equation}
was introduced in the context of finite-strain viscoplasticity. Note that this volumetric energy attains a minimum at $J=1$, with $h_{\mathrm{SM}}'(1) = 0$, and therefore it does not satisfy the free stress condition \eqref{relhmu}.  It does, however, satisfy the convexity condition as 
\[
h_{\mathrm{SM}}''(J) = \frac{\kappa}{2} \Big(1 + \frac{1}{J^2}\Big) > 0,
\]
for all deformations. Moreover,
\[
h_{\mathrm{SM}}'''(J) = - \frac{\kappa}{J^3}< 0, \qquad J \in (0,\infty).
\]
Also, $h_{\mathrm{SM}}$ increases unboundedly as $J \to 0^+$ and as $J \to \infty$. 

\subsubsection*{$\mathrm{(h)}$ Bischoff, Arruda and Grosh model} Bischoff \emph{et al.} \cite{BiArGr01} proposed the following volumetric response function
\begin{equation}
\label{hBAG}
h_{\mathrm{BAG}}(J) = - \frac{d}{2} \mu + \frac{\overline{c}}{b^2} \big( \cosh(b(J-1)) - 1\big),
\end{equation}
where the constants $\overline{c}, b$ are positive empirical constants which should be calibrated from experimental data. Notice that $h_{\mathrm{BAG}}'(1) = 0$ and $J = 1$ is a minimum; thus, it does not satisfy \eqref{relhmu}. The convexity condition holds as,
\[
h_{\mathrm{BAG}}''(J) = \overline{c} \cosh( b(J-1)) > 0,
\]
for all $J \in (0,\infty)$. However,
\[
h_{\mathrm{BAG}}'''(J) = \overline{c} b \sinh( b(J-1)),
\]
yielding $h_{\mathrm{BAG}}'''(1) = 0$, as well as $h_{\mathrm{BAG}}'''(J) > 0$ if $J > 1$ and $h_{\mathrm{BAG}}'''(J) < 0$ if $J < 1$. Note also that $h_{\mathrm{BAG}} \to \infty$ as $J \to \infty$ but $h_{\mathrm{BAG}}(0^+)$ is well defined. This model was proposed to account for the contributions of entropy and initial energy to volume change. Its derivation follows non-Gaussian statistics of long chain molecules, which is necessary for large deformations. It can be interpreted as a non-Gaussian, higher order representation of the Ogden-Hill model \eqref{hOH} in the small volume changes regime, inasmuch as the series expansion around $J = 1$ yields
\[
h_{\mathrm{BAG}}(J) = - \frac{d}{2} \mu + \frac{\overline{c}}{2}(J-1)^2 + O((J-1)^4).
\]

\begin{remark}
The energy densities presented above are divided into two categories. Models (a) thru (e) can be interpreted as compressible versions of the neo-Hookean material in the sense described by Pence and Gou \cite{PnGo15}: they satisfy the free stress condition \eqref{relhmu} and the hydrostatic pressure condition \eqref{relkapmu}, both at the incompressible limit with no deformation, and represent materials which are nearly incompressible. In contrast, models (f) thru (h) are designed to fit experimental data involving phenomenological observations such as, for example, when foam polymers undergo large changes in volume \cite{HorSa04}. In these models, $h'(1) = 0$, so that the volumetric function $h$ provides a direct penalization of volume departing from $J =1$. All models (a) thru (h) provide neo-Hookean behavior in the incompressible limit, namely, $\overline{W}(I^{(1)},1) = \overline{W}_{\!\mathrm{nH}}(I^{(1)})$, and reduce to the standard linearly elastic material response when deformations are small (that is, when $| \tfrac{1}{2}(U^\top U - \Id) | \ll 1$).
\end{remark}

\begin{remark}
All the model examples presented here are physically motivated energy functions that satisfy assumptions \eqref{H1} and \eqref{H2} for all possible deformations and, therefore, they belong to the general class of compressible hyperelastic Hadamard materials. (It is to be observed that the family does not include other hyperelastic models found in the literature, such as the compressible versions of the Blatz-Ko, Murnaghan or Varga models, just to mention a few; see \cite{Holz00,Og84} and the references therein.) Notably, the convexity of the energy (property \eqref{H2}), implies that all energy functions are rank-one convex in the whole domain of $U$ with $\det U > 0$, making the elastodynamics equations hyperbolic in the whole domain of their state variables. The stability results of this paper apply to materials which, in addition, satisfy the material convexity condition \eqref{H3}.
\end{remark}

\def\cprime{$'\!\!$}

%
%

\end{document}